\documentclass[12pt]{amsart}
\usepackage{mathrsfs,amsmath,amssymb,amsthm,ascmac}
\usepackage[pdftex]{graphicx}
\usepackage[all]{xy}
\usepackage{stmaryrd}
\theoremstyle{plain}
\newtheorem{theorem}{Theorem}[section]
\newtheorem{lemma}[theorem]{Lemma}
\newtheorem{prop}[theorem]{Proposition}
\newtheorem{fact}[theorem]{Fact}

\newtheorem{cor}[theorem]{Corollary}
\newtheorem{obs}[theorem]{Observation}

\newtheorem{question}{Question}

\theoremstyle{definition}
\newtheorem{definition}[theorem]{Definition}
\newtheorem{example}[theorem]{Example}
\newtheorem*{ack}{Acknowledgements}
\newtheorem*{notation}{Notation}
\newtheorem*{discussion}{Disscussion}

\newtheorem*{claim}{Claim}
\newtheorem*{remark}{Remark}

\numberwithin{equation}{section}

\newcommand{\om}{\omega}
\newcommand{\upto}{\upharpoonright}
\newcommand{\tto}{\rightrightarrows}

\newcommand{\N}{\mathbb{N}}
\newcommand{\U}{\mathcal{U}}
\newcommand{\V}{\mathcal{V}}
\newcommand{\ep}{\varepsilon}
\newcommand{\arr}{\rightarrowtriangle}
\newcommand{\iffarr}{\leftrightarrowtriangle}

\newcommand{\mono}{\rightarrowtail}
\newcommand{\embed}{\hookrightarrow}

\newcommand{\Me}{{\tt Merlin} }
\newcommand{\Mer}{{\tt Merlin}}
\newcommand{\Ar}{{\tt Arthur} }
\newcommand{\Art}{{\tt Arthur}}
\newcommand{\Ni}{{\tt Nimue} }
\newcommand{\Nim}{{\tt Nimue}}

\newcommand{\pcolon}{\colon\!\!\!\subseteq}
\newcommand{\str}{\varsigma}

\newcommand\tboldsymbol[1]{%
\protect\raisebox{0pt}[0pt][0pt]{%
$\underset{\widetilde{}}{\mathbf{#1}}$}\mbox{\hskip 1pt}}

\newcommand{\eval}[1]{\llbracket #1 \rrbracket}
\newcommand{\tpbf}[1]{\tboldsymbol{#1}}
\newcommand{\tpbfxy}[1]{\underset{\sim}{\mathbf{#1}}}
\newcommand{\tplf}[1]{{\sf #1}}

\title[Rethinking the notion of oracle]{Rethinking the notion of oracle:\\A prequel to Lawvere-Tierney topologies for computability theorists}

\author{Takayuki Kihara}

\begin{document}
\maketitle

\begin{abstract}
We present three different perspectives of oracle.
First, an oracle is a blackbox; second, an oracle is a tool to change the way we access mathematical objects; and third, an oracle is a factor that causes a change in truth values.
Formally, the second perspective advocates that an oracle is an endofunctor on the category of coded sets (preserving underlying sets) -- we associate it with a universal closure operator.
The third perspective advocates that an oracle is an operation on the object of truth values -- we associate it with a Lawvere-Tierney topology.
These three perspectives create a link between the three fields, computability theory, synthetic descriptive set theory, and effective topos theory.
\end{abstract}

\setcounter{tocdepth}{1}
\tableofcontents
\section{Introduction}

As the subtitle suggests, this article is closely tied to an earlier article by the author, ``Lawvere-Tierney topologies for computability theorists \cite{Kih21}'', but it is not a sequel and does not assume knowledge of that work.
Rather, this article lays the foundation for the previous work and is intended to be read prior to \cite{Kih21}.

\subsection{Three perspectives of oracle}\label{sec:intro-oracle}

``{\it Oracle}'' is a fundamental notion in the theory of computation/computability, originating in Turing's notion of O-machine \cite{Tur39}.
In this article, let us reconsider what an oracle is.
At least three different perspectives of oracle can be presented.

\subsubsection*{Perspective I}
The first perspective is the most standard one, which is to think of an oracle as a ``{\it blackbox}'', represented as a set, a function, an infinite string, etc.
If we think of a blackbox as just a container to store an input data  (whose data type is stream), as some people say, an oracle is merely an input stream.
Such a view is also quite standard nowadays.

%The second perspective is based on a recent approach taken e.g.~by de Brecht and Pauly \cite{dB14,PaBr15} to develop {\it synthetic descriptive set theory}, which is, according to them \cite{PaBr15}, the idea that descriptive set theory can be reinterpreted as the study of certain endofunctors and derived concepts, primarily in the category of {\it represented spaces}.
%We interpret this key idea of synthetic descriptive set theory as relativizing topological notions by (higher-type) oracles.
%In this approach, an oracle is considered to be a functor that allows us to {\it change the way we access spaces}:
%A point $x$ in a space without a computable name can still have a computable name relative to some oracle $\alpha$, which means that by going through the oracle $\alpha$, we have access to that point $x$.
%In this way, an oracle yields an endofunctor on the category of represented spaces (see Section \ref{sec:rep-sp-main} for the details).

\subsubsection*{Perspective II}
The second perspective is based on a recent approach concerning ``{\em change of coding}''.
In modern computation theory, one can discuss computability for various mathematical objects, including those of continuous datatypes.
The basic idea is to access an abstract mathematical object via a coding system (e.g., a numbering \cite{Er99}, a represented space \cite{Handbook-CCA}, an assembly \cite{vOBook}, etc.), and then reduce the discussion on computability regarding mathematical objects to computability regarding codes.
In this setting, in general, an element $x$ in a set may not have a computable name (with respect to a given coding system), but such an element can still have a computable name relative to some oracle $\alpha$.
This means that by going through the oracle $\alpha$, we have access to that element $x$.
In this way, an oracle $\alpha$ is considered to be a  functor that allows us to {\em change the way we access mathematical objects}.

In computability analysis, this idea was first introduced as the notion of ``{jump of representation} \cite{Zie07,dB14}'' and later used by de Brecht and Pauly \cite{PaBr15} to develop synthetic descriptive set theory.
As explained above, an oracle induces a ``change of coding (jump of representation)'', which gives an arrow from one coding system to another.
To be precise, an oracle yields a (set-preserving) endofunctor on the category of coded sets (which could be numbered sets, represented spaces, or assemblies); see Section \ref{sec:rep-sp-main} for the details.
This is also related to Longley's notion of applicative morphism \cite{LoPhD95}, which yields a set-preserving regular functor on coding systems (assemblies); see Section \ref{sec:applicative-morphism}.

%For instance, {$\tpbf{\Delta}^0_2(X)$}, the hyperspace of {$\tpbf{\Delta}^0_2$ sets} in $X$, is thought of as the result of applying some {endofunctor $\nabla$} on ${\bf Rep}$ to {$\mathcal{O}(X)$}, the represented hyperspace of {open sets} in $X$.
%Through this particular endofunctor $\nabla$, various topological notions (e.g.~Hausdorffness, compactness, overtness) can be lifted to notions about $\tpbf{\Delta}^0_2$ sets.
%As a sample result, de Brecht-Pauly \cite{dBPa17} showed that a quasi-Polish space is {$\nabla$-compact} if and only if it is {Noetherian}, i.e. every strictly ascending chain of open sets is finite.
%Of course, there are endofunctors not only for $\tpbf{\Delta}^0_2$, but for various descriptive set-theoretic notions as well.
\subsubsection*{Perspective III}
The third perspective of oracle is the one that we promote in this article.
In this third perspective, we consider an oracle to be  an ``{\it operation on truth-values}'' that may cause a transformation of one world into another:
A mathematical statement $\varphi$ may be false in computable mathematics, but $\varphi$ can be true in computable mathematics relative to an oracle $\alpha$.
This means that the oracle $\alpha$ caused a change in the truth value of the statement $\varphi$, and also caused a change from the computable world to the $\alpha$-relative computable world.

For a better understanding of this idea, let us consider two major semantics of intuitionistic logic, the Kripke semantics and Kleene's realizability interpretation (the semantics in computable mathematics).
The factor that causes the change in Kripke semantics is the notion of Grothendieck coverage (on a poset) or equivalently, nucleus  \cite{GuHo19} (see also Section \ref{sec:modest-modality}), while the factor that causes the change in Kleene realizability is the notion of oracle, as described above.
Remarkably, as we give a detailed analysis in this article, the two factors are essentially identical, differing only in the underlying toposes on which they act.
The former is the topos of presheaves on a poset, and the latter is the effective topos.
In both cases, the factor that causes the change from some topos to another topos is Lawvere-Tierney topology (which is formally some kind of closure operator on truth values).

It was Hyland \cite{Hey82} that linked a Turing oracle to a Lawvere-Tierney topology, and this connection was subsequently investigated in several articles, including \cite{FavO14}.
However, Turing oracles occupy only a small part of the topologies on the effective topos.
The analysis of the overall structure of the topologies was carried out by Lee and van Oosten \cite{LvO}, and later by the author \cite{Kih21}, who presented a novel notion of oracle that fully correspond to the Lawvere-Tierney topologies, giving a complete third perspective on oracle.

One might say that this third perspective is based on the idea that there is a correspondence between ``{\em computations using oracles}'' and ``{\em proofs using non-constructive axioms}''.
A related idea is partially used as a very standard technique in, for example, classical reverse mathematics \cite{SOSOA:Simpson}:
It is common practice to extend a model consisting only of computable entities, by adding a solution to an algorithmically undecidable problem as an oracle, to a model that satisfies a non-constructive axiom (see also \cite{Sho10}).

Our approach is similar, but with a newer perspective that deals more directly with operations on truth-values.
Using our approach (in combination with a result in \cite{HiJo16}), the above idea of $\om$-model construction in classical reverse mathematics can also be incorporated into the idea of using an operation on the truth-values.
The author's idea was derived in an attempt \cite{Kih20} to generalize Lifschitz realizability in order to separate non-constructive principles in constructive reverse mathematics (i.e., reverse mathematics based on intuitionistic logic; see \cite{Die18,Diener2021}).
Interestingly, the attempt to generalize Lifschitz realizability was also independently studied by Rathjen-Swan \cite{RaSw20} as a factor to decorate Kripke semantics.
See Section \ref{sec:realizability-relative-to-oracle} for a detailed discussion of the results outlined in this paragraph.

\subsubsection*{Summary}
In this article, we clarify the connection between these three perspectives of oracle, and in particular gives a detailed analysis of the third perspective described above.
In this way, we attempt to bridge the gap between computability theory, synthetic descriptive set theory, realizability theory and effective topos theory.

This article uses only elementary notions that require little prior knowledge (except for elementary computability theory) and does not make unnecessary generalizations (whenever possible).
%An earlier version of this paper (arXiv:2202.00188v3) was returned without proceeding to the peer review process of a journal on the grounds that unfortunately few researchers are familiar with the theory of realizability toposes.
%The author agrees with this decision, and realizability theory should be written in an elementary form that can be understood by as many readers as possible.
The intended audience for this article is not only experts on realizability theory, but also a wide range of logicians, theoretical computer scientists, and mathematicians, including all computability theorists.
The most important thing is to lower the barriers to entry to our research.
Less prior knowledge is preferred over generality.
%Once understood with appropriate elementary examples, generalization is easy.
In the words of Hyland, quoted in the Preface to the book \cite{vOBook}, ``One good example is worth a host of generalities.''

\subsection{Computability-Theoretic Background}\label{sec:background}

\subsubsection{Oracle and Turing reducibility}
One of the most important subjects of study in computability theory is the structure of degrees of (algorithmic) unsolvability.
The notion of degrees of unsolvability is defined via the notion of relative computation with oracles.
Historically, it was Turing \cite{Tur39} who first introduced the concept of oracle computation.
From a modern perspective, an oracle computation is merely a computation that accepts data of type {\tt nat\,->\,nat},  or alternatively, {\tt stream} (i.e., a possibly infinite sequence of symbols) as input, and any such data is traditionally called an {\em oracle}.
Formally, an oracle machine is a Turing machine with an extra input tape for writing data of {\tt stream} type sequentially over time, or a standard computer program with variables of type {\tt nat\,->\,nat}, which is a finitistic device with a finite description:
At each step of the oracle computation, the input stream is only read up to a finite number of digits; therefore, each oracle computation is always performed using only a finite amount of information.

Let us now consider any program $\varphi({\tt f},{\tt n})$ that accepts data ${\tt f}$ of type {\tt nat\,->\,nat} (or {\tt stream}) and data ${\tt n}$ of type {\tt nat} (i.e., a natural number) as input.
Any function $g\colon\N\to\N$ can be input to this program as an oracle (stream data type).
Then, a function $h\colon\N\to\N$ is said to be {\em computable relative to oracle $g$} if
\[(\exists \varphi\mbox{ program})(\forall n\in\N)\;\varphi(g,n)=h(n).\]

In this case, it is also said that a function {\em $h\colon\N\to\N$ is Turing reducible to $g$} (written $f\leq_Tg$).
By currying, the type $\N^\N\times\N\to\N$ of $\varphi$ is often identified with $\N^\N\to\N^\N$; that is, one may consider an oracle program $\varphi$ as a partial function on $\N^\N$, where $\varphi(g)$ is defined as $\lambda n.\varphi(g,n)$.
Then $h\leq_Tg$ if and only if $\varphi(g)=h$ for some program $\varphi$.
This preorder $\leq_T$ can be thought of as providing a measure for comparing the computability-theoretic strength of functions on natural numbers. 
The preorder $\leq_T$ induces an equivalence relation $\equiv_T$, each equivalence class of which is called a {\em Turing degree}.
From Turing's time to the present, the structure of Turing degrees has been investigated to an extremely deep level.
As a result, a vast amount of research results are known (see e.g.~\cite{DHBook,Handbook,SoareBook} for the tip of the iceberg).

\subsubsection{Mass problems}\label{sec:intro-mass-problems}
In the theory of computation, two-valued functions on $\N$ (or subsets of $\N$) are often called {\em decision problems}.
There are other types of problems in computability theory whose algorithmic unsolvability is explored in depth, one of which is called a {\em mass problem}.
This is a problem with possibly many solutions, and we usually discuss the difficulty of finding any one of them.
A mass problem is identified with the set of its solutions; for example, an empty set is considered a mass problem without a solution.
In some setting, we consider a mass problem whose solutions are functions on $\N$, and thus a mass problem is introduced as a set of functions, $P\subseteq\N^\N$.
For mass problems $P,Q\subseteq\N^\N$, {\em $P$ is Medvedev reducible to $Q$} (written $P\leq_MQ$) if 
\[(\exists\varphi\mbox{ program})(\forall g\in Q)\;\varphi(g)\in P.\]

Historically, Medvedev \cite{Medvedev} used this notion to give an interpretation of intuitionistic propositional calculus (IPC), and later Kuyper \cite{Kuy15} extended Medvedev's interpretation to first-order predicate logic (IQC) using Lawvere's notion of hyperdoctrine.
The interpretation of IPC/IQC using mass problems is also important as one of the rigorous implementations of Kolmogorov's ``calculus of problems'' \cite{Kol32} (see e.g.~\cite{Kuy15,BaSi16} for detailed discussion).
Another type of reducibility for mass problems is Muchnik reducibility, a non-uniform version of Medvedev reducibility (which induces the domination preorder on the opposite of the Turing preoreder).
Simpson (e.g.~\cite{Sim11}) pointed out the connection between this reducibility and the study of $\om$-models in classical reverse mathematics  \cite{SOSOA:Simpson}, leading to a revival of the study of mass problems from a contemporary perspective.
Simpson et al.~\cite{Si15,BaSi16} also discuss the foundational importance of Medvedev and Muchnik degrees.

\subsubsection{Search problems}\label{sec:intro-search-problems}
A mass problem is considered to be a special kind of search problem.
A more general search problem is given by a multi-valued function on $\tpbf{N}$ (i.e., a $\mathcal{P}(\tpbf{N})$-valued function; see also Section \ref{sec:notations}), where $\tpbf{N}$ is either $\N$ or $\N^\N$.
Each multi-valued function $f$ represents a collection of problems, where $f(x)$ is the solution set for the $x$-th problem.
There are various known methods for comparing algorithmic unsolvability of search problems (multi-valued functions), so let us start with the simplest ones.
For this purpose, we first recall the definition of many-one reducibility for decision problems.
For decision problems, i.e., two-valued functions $f,g\colon\tpbf{N}\to 2$, $f$ is many-one reducible to $g$ if there exists a program $\varphi$ such that $g\circ\varphi=f$ holds.

The notion of many-one reducibility for multivalued functions is defined by replacing the equality relation in the usual many-one reducibility with the refinement relation:
For multi-valued functions $f$ and $g$ on $\tpbf{N}$, we say that {\em $g$ refines $f$} (written $g\preceq f$) if for any $x$, $x\in{\rm dom}(g)$ implies $x\in{\rm dom}(f)$ and $g(x)\subseteq f(x)$.
Then we say that $f$ is {\em many-one reducible to $g$} ({written $f\leq_mg$}) if there exists a program $\varphi$ such that $g\circ\varphi\preceq f$.
In other words, for any $x$ and $y$,
\[\mbox{$x\in{\rm dom}(f)\implies\varphi(x)\in{\rm dom}(g)$;\quad and\quad  $y\in g(\varphi(x))\implies y\in f(x)$}\]

While a many-one reduction only transforms the input data, there is a reducibility notion that also transforms the output data:
$f$ is {\em strong Weihrauch reducible to $g$} ({written $f\leq_{sW}g$}) if there exist programs $\varphi_-$ and $\varphi_+$ such that for any $x$ and $y$,
\[\mbox{$x\in{\rm dom}(f)\implies\varphi_-(x)\in{\rm dom}(g)$;\quad  and\quad  $y\in g(\varphi_-(x))\implies \varphi_+(y)\in f(x)$}\]

A similar notion is called a generalized Galois-Tukey connection in the study of cardinal invariants (see e.g.~Blass \cite[Section 4]{Bl10}).

The most important reducibility in this paper is the following:
$f$ is {\em Weihrauch reducible to $g$} ({written $f\leq_{W}g$}) if there exist programs $\varphi_-$ and $\varphi_+$ such that for any $x$ and $y$,
\[\mbox{$x\in{\rm dom}(f)\implies\varphi_-(x)\in{\rm dom}(g)$;\quad  and\quad  $y\in g(\varphi_-(x))\implies \varphi_+(x,y)\in f(x)$}\]

Although seemingly technical, it is in fact a quite natural concept, corresponding to a relative computation that makes exactly one query to the oracle.
This reducibility notion has become one of the most important research topics in modern computable analysis and related areas.
For the origin of Weihrauch reducibility, see Brattka \cite{Bra22}.
Also, see the survey \cite{CCA} for the basic results on Weihrauch degrees.
A similar notion has been studied in Hirsch \cite{Hir90} in a more abstract setting.

\subsubsection{Reverse mathematics and logic}
Often, partial multi-valued functions are identified with $\forall\exists$-statements (see also Section \ref{sec:notations}), and through this identification, the classification of partial multi-valued functions by Weihrauch reducibility is sometimes regarded as a handy analogue of reverse mathematics; see e.g.~\cite{BrGh11,GM09,KMP20,CCA}.
Here, reverse mathematics \cite{SOSOA:Simpson} is a branch of mathematical logic that aims to classify mathematical theorems in algebra, geometry, analysis, combinatorics, etc.~by measuring their logical strength, while Weihrauch degrees can classify mathematical theorems (of the $\forall\exists$-forms) by measuring their computational strength.
For example, some researchers have established the following Weihrauch-style classification of mathematical theorems:
\begin{multline*}
\mbox{``The intermediate value theorem''}<_W\mbox{``the Brouwer fixed point theorem''}\\
<_W\mbox{``the Radon-Nikodym theorem''}<_W\mbox{``the Bolzano-Weierstra{\ss} theorem''}
\end{multline*}
under appropriate formalizations, where, for example, the fixed point theorem is formalized as ``the problem of searching for a fixed point''; see \cite{CCA}.

Weihrauch-style reverse mathematics is sometimes considered to be a refinement of one aspect of ordinary reverse mathematics (or constructive reverse mathematics \cite{Die18,Diener2021}), specifically, it is different from ordinary reverse mathematics in that it is {\em resource-sensitive}.
It should be noted, however, that this is only a refinement of one aspect of reverse mathematics, and many important aspects (such as measuring consistency strength) have been ignored.
For detailed discussions of the relationship between Weihrauch reducibility and (constructive) reverse mathematics, see e.g.~\cite{Fuj21,Uft21}.

Another logical aspect of the Weihrauch lattice is its relation to linear logic, which is also frequently discussed due to its resource-sensitive nature; see e.g.~\cite{BrPa18,BrGh20,Uft21}.
In a related vein, the similarity between Weihrauch reducibility and the Dialectica interpretation is also frequently pointed out; see e.g.~\cite{CCA,Uft21}.
These relationships are implicit in Blass \cite{Bl95}, albeit in a slightly different context.
%\footnote{that is, in the context of cardinal invariants, but cardinal invariants have also been studied in computability theory \cite{Rup,Bre}, and their relevance to strong Weihrauch reductions is also mentioned implicitly in \cite{Kih18} and explicitly in \cite{Green}.}.
Like Medvedev and Muchnik reducibility, the notion of Weihrauch reducibility is also closely related to Kolmogorov's ``calculus of problems''; therefore, Weihrauch reducibility is also of at least as foundational importance as Medvedev and Muchnik reducibilities.
For the relevance of Kolmogorov's ``calculus of problems'', the Dialectica interpretation and so on, see also de Paiva-da Silva \cite{dPdS21}.
%\footnote{Here, the currently used definition of the Dialectica category ${\bf GC}$ is not very appropriate for Weihrauch reducibility; what is appropriate is the original definition of the Dialectica category ${\bf DC}$, which straightforwardly uses G\"odel's Dialectica interpretation of the implication \cite{dPaaa}.}.

\subsubsection{Turing-like reducibility for search problems}
Weihrauch reducibility is defined in terms of a resource-sensitive oracle-computation for search problems, but there is also a reducibility notion for search problems that is not resource-sensitive (like Turing reduction).
This notion is usually referred to as {\em generalized Weihrauch reducibility} \cite{HiJo16}.
However, this should be the most standard reducibility for search problems (multi-valued functions), and a complex name like this is not desirable.
%Personally, I hope that the term generalized Weihrauch reducibility will be dropped and a simpler name will be assigned.
The definition of this reducibility is exactly the same as that of Turing reducibility, except that a partial multi-valued function $g$ is used as an oracle.
A finitistic device called an oracle machine is exactly the same as in Turing reducibility, except that the behavior when $g(z)$ is accessed during the computation is different:
If $g(z)$ is undefined (i.e., $z\not\in{\rm dom}(g)$), the computation will never terminate\footnote{In the case where $g$ is a partial {\em single-valued} function, essentially the same model of relative computation has been studied, e.g.~by Sasso \cite{Sas75}, Goodman \cite{Goo78}, and van Oosten \cite[Section 1.4.5]{vOBook}. Note that this relative computation model performs ``serial'' computations (i.e., no parallel computations are allowed when accessing the oracle $g$).}.
If $g(z)$ has more than one value, the computation will branch into several, resulting in a non-deterministic computation.

Roughly speaking, this notion is a fusion of Weihrauch reducibility and Turing reducibility, so this article will refer to this reducibility as Turing-Weihrauch reducibility instead of generalized Weihrauch reducibility.
To be more precise, we say that $f$ is {\em Turing-Weihrauch reducible to $g$} (written $f\leq_{TW}g$)\footnote{If $f$ and $g$ are partial single-valued functions on $\N$, note that $f$ is Turing-Weihrauch reducible to $g$ if and only if $f$ can be extended to a partial function which is Turing reducible to $g$ in the sense of Sasso \cite{Sas75}. It might be better to call this ``$f$ is Turing sub-reducible to $g$'', as suggested by Madore \cite{Mad12}. Essentially the same reducibility notion (on any partial combinatory algebra) has also been studied, e.g.~in van Oosten \cite{vO06}.} if, in a $g$-relative computation $\varphi^g$ as described above, if $x\in{\rm dom}(f)$ then, along any path of the nondeterministic computation $\varphi^g$, the computation halts and outputs a solution for $f(x)$.
For the details of the definition, see e.g.~\cite{HiJo16,Kih21} and Definition \ref{def:game-Turing-Weihrauch-reducibility}.
The importance of this concept will be discussed later.

\subsubsection{Related notions in other areas}
Although we have discussed reducibility notions in computability theory, similar reducibility notions have of course been studied in computational complexity theory, for example, polynomial-time Turing reduction is known as {\em Cook reduction} and polynomial-time many-one reduction as {\em Karp reduction}.
Reducibility for search problems has also been studied in computational complexity theory; for instance, one of the conditions required for {\em Levin reduction} is precisely polynomial-time Weihrauch reduction.
For reducibilities in complexity theory, see e.g.~\cite{Goldreich}.
Kawamura-Cook \cite{KaCo10} also introduced polynomial-time Weihrauch reducibility on continuous objects, which has since been widely used in the study of computational complexity in analysis.

This article also focuses on oracle relativization and reducibility in descriptive set theory.
The parallels between descriptive set theory and computability theory have long been noted.
For instance, it was in the 1950s that Addison \cite{Add55} pointed out the now well-known similarity between the (hyper-)arithmetical hierarchy and the Borel hierarchy.
Later in this field, Wadge \cite{Wad83} introduced a continuous version of many-one reducibility, known as {\em Wadge reducibility}, which led to a surprisingly deep theory, unimaginable given the brevity of the definition; see e.g.~\cite{Cabal12}.
This reducibility is defined in the definition of many-one reducibility above as a reduction $\varphi$ being a continuous function rather than a program, where $\tpbf{N}=\N^\N$ (endowed with the standard Baire topology).
Similarly, in computable analysis, strong Weihrauch reducibility and Weihrauch reducibility with continuous functions $\varphi_-$ and $\varphi_+$ have also been studied.
These reducibility notions are also related to Pauly-de Brecht's synthetic descriptive set theory \cite{PaBr15}, as we will see later.
Synthetic descriptive set theory can be interpreted as a framework for treating several notions of descriptive set theory as oracle relativization, and since this is one of the subjects of this article, it will be carefully explained step by step in the later sections.

%The arithmetic hierarchy is seen as a hierarchy of relativization to the halting problem oracle.
%More precisely, it is a hyperarithmetical hierarchy with infinitary formulas with a well-founded syntax tree rather than a formula with a finite syntax tree, but this analogy was subsequently studied at the level of more complex syntax trees, such as quantification (game quantifer, generalized quantifier) with an ill-founded sytanx tree, as well as computational analogies in description set theory.
%In a sense, these seem to be understandable in terms of relativization to more complex oracles.

\subsubsection{Topos theory}
One of the aims of this article is to unify these notions using a topos-theoretic notion called Lawvere-Tierney topology, which can be viewed as a generalization of Grothendieck topology \cite{SGL}.
Various toposes associated with computability theory have been studied, including the effective topos \cite{Hey82}, realizability toposes \cite{vOBook}, the recursive topos \cite{Mul82}, the Muchnik topos \cite{BaSi16}, and sheaf toposes for realizability \cite{AwBa08}.
The first topos, the effective topos, is often considered to be the world of computable mathematics, where every function is computable, and every property is computably witnessed.
The second topos is a generalization of the first topos, and this article also mentions a further generalization, relative realizability toposes (e.g., the Kleene-Vesley topos \cite{vOBook}).
The first two (and its relatives) are quite natural toposes based on Kleene's realizability interpretation (i.e., an interpretation of intuitionistic arithmetic using computable functions; see e.g.~\cite{Tro98,vO02}) and are the most deeply studied of these; see also \cite{vOBook}.

All others are Grothendieck toposes, the third \cite{Mul82} is the sheaf topos over the canonical Grothendieck topology on the monoid of total computable functions on $\N$, the fourth \cite{BaSi16} is on the poset of Muchnik degrees (the frame of Alexandrov open sets on the poset of Turing degrees), and the fifth  \cite[Definition 2.13]{AwBa08} (see also \cite[Section 4.3]{Bau00}) can be seen as a good combination of the third and fourth ones.
%but based on the category of the mass problems (over a relative partial combinatory algebra) and the Medvedev reductions (so one may call it the Medvedev topos).
%, and the sixth \cite{XuPhD15} is based on the monoid of uniformly continuous functions on the Cantor space. 
So, these are easy to define, but it seems difficult to precisely relate Grothendieck topologies to computability-theoretic notions like oracles.
Nevertheless, it is possible to outline our idea in terms of Grothendieck topology, and it may be useful to do so:

Each Grothendieck topology $J$ specifies a notion of covering, and the first key point is to notice the similarity between the properties of the statements ``{\em a sieve $\mathcal{U}$ is a $J$-cover of an object $S$}'' and ``{\em one can solve a problem $\mathcal{U}$ using an algorithm $S$ with the help of an oracle $J$}''\footnote{One attractive interpretation of the statement ``{\em $\mathcal{U}$ $J$-covers $S$}'' in the context of logic could be ``{\em under a theory (or logic) $J$, one can prove a formula $S$ using an additional set $\mathcal{U}$ of assumptions}''. It is then tempting to think that $\mathcal{U}$ corresponds to an algorithm and $S$ corresponds to a problem to be solved. However, the Kripke-Joyal semantics leads us to a somewhat dual interpretation ``{\em under $J$, every sentence provable in any $\varphi\in\mathcal{U}$ is also provable in $S$}'' (e.g., \cite{CSSS00}). The latter seems more compatible with our point of view.}.
For these statements, it is not possible to give an exact correspondence in terms of Grothendieck topology, but we will see that it is possible in a sense to give an exact correspondence in terms of Lawvere-Tierney topology on the effective topos or its generalizations.

\subsubsection{Degree theory encounters with topos theory}
There are several precedents for the study of the relationship between oracles and Lawvere-Tierney topologies on the effective topos.
In fact, in the paper \cite{Hey82} in which Hyland first introduced the effective topos, it has already been shown that the poset of the Turing degrees can be embedded in the poset of the Lawvere-Tierney topologies on effective topos; see also \cite{Pit85,Ph89}.
In a related vein, computability relative to partial Turing oracles in the context of realizability has been studied, e.g.~in \cite{vO06,FavO14}.

However, these results are only relativization to partial Turing oracles, and as we mentioned above, there are various other reducibility notions in computability theory.
As an example of an encounter with another reducibility notion in realizability theory, Bauer and Yoshimura analyzed in detail a typical derivation method of implication in constructive reverse mathematics, and found that an extension of Weihrauch reducibility emerges when it is analyzed with a relative realizability topos such as the effective topos and the Kleene-Vesley topos.
Bauer \cite{Bau21} formalized it as {\em extended Weihrauch reducibility}, and Kihara \cite{Kih21} pointed out that this notion can be understood as incorporating non-uniform computations into Weihrauch reductions, like advice strings in computational complexity theory, and presented a computation with a random oracle as such an example.
For the details of extended Weihrauch reducibility, see also Section \ref{sec:secret-input}.

The relationship between this reducibility notion and Lawvere-Tierney topology is not obvious, but a concrete presentation of a Lawvere-Tierney topology on the effective topos by Lee-van Oosten \cite{LvO} serves as a tool to bridge them.
Using this concrete presentation, Kihara \cite{Kih21} introduced the notion of LT-reducibility, a common extension of Turing-Weihrauch reducibility and extended Weihrauch reducibility, and in this sense revealed that there is a one-to-one correspondence between extended oracles and Lawvere-Tierney topologies on the effective topos.

One benefit of this correspondence is that the topos-theoretic aspect of Lawvere-Tiernery topology, that is, interpreting intuitionistic higher-order logic, has led to a more direct connection between Weihrauch degree theory (Weihrauch-style reverse mathematics) and constructive reverse mathematics.
Another (and most) important point is that this correspondence has led to novel concrete objects of study and results in classical computability theory on the natural numbers \cite{Kih21}.
One of the goals of this article is to explore the details of this one-to-one correspondence and to better clarify the meaning of this correspondence.

\subsection{Overview}

%In Section \ref{sec:synthetic-dst}, we describe the second perspective of oracle, i.e. how to interpret the idea of synthetic descriptive set theory as the relativization of spaces by oracles.

In Section \ref{sec:synthetic-dst}, we give an abstract definition of the notion of oracle computation, based on our second perspective of oracle (inspired by synthetic descriptive set theory).
Formally, we analyze various properties on multivalued functions, e.g., being 
\begin{center}
{\em computably transparent}, {\em inflationary}, and {\em idempotent},
\end{center}
and link these properties to the way we access oracles.
In computability theory, there are various notions of oracle computation, each of which is determined by fixing a reducibility notion that specifies how to access oracles.
Well-known reducibility notions include Turing reducibility and many-one reducibillity \cite{Cooper,RogBook,SoareBook}, but in this article we also deal with Medvedev reducibility \cite{Medvedev}, Weihrauch reducibility \cite{CCA}, etc.
Then, for example, we show that being computably transparent, inflationary, and idempotent correspond to the notions of universal oracle computation with exactly one query (Weihrauch reducibility), at most one query (pointed Weihrauch reducibility), and finitely many queries (Turing-Weihrauch reducibility), respectively.

In Section \ref{sec:realizability}, we revisit this idea in the context of operations on the object $\Omega$ of truth-values.
There are various additional properties of such operations, such as being
\begin{center}
{\em monotone}, {\em inflationary}, and {\em idempotent},
\end{center}
and various other preservation properties.
Monotone, inflationary, and idempotent operations correspond to what is known as Lawvere-Tierney topologies (also known as local operators, geometric modalities, and internal nuclei) \cite{elephant,SGL,LvO}.
We link these properties to the above notions of multivalued functions.
Through that link, we show, for example, the following correspondences:
\begin{itemize}
\item Many-one degrees (for partial single-valued functions) correspond to $\bigcup,\bigcap$-preserving operations on $\Omega$ in the effective topos.
\item Turing degrees (for partial functions) correspond to $\bigcup,\bigcap$-preserving Lawvere Tierney topologies on the effective topos.
\item Medvedev degrees correspond to open Lawvere-Tierney topologies on the Kleene-Vesley topos.
\item Weihrauch degrees correspond to $\bigcap$-preserving monotone operations on $\Omega$ in the Kleene-Vesley topos.
\item Pointed Weihrauch degrees correspond to $\bigcap$-preserving monotone, inflationary, operations on $\Omega$ in the Kleene-Vesley topos.
\item Turing-Weihrauch degrees correspond to $\bigcap$-preserving Lawvere-Tierney topologies on the Kleene-Vesley topos.
\end{itemize}

%As a side result, we also observe that Medvedev degrees correspond to open Lawvere-Tierney topologies (i.e., topologies for which internal nuclei yield open sublocales of $\Omega$; see e.g.~\cite[Section A.4.5]{elephant}) 
Here, both the effective topos and the Kleene-Vesley topos can be treated as worlds of computable mathematics, but the former uses coding by natural numbers, while the latter uses coding by streams (potentially infinite sequences of symbols in a fixed alphabet); see e.g.~\cite{vOBook}.
Also, if we consider the realizability topos obtained from Kleene's second algebra (i.e., the world of continuous mathematics), Wadge degrees (for functions) correspond to $\bigcup,\bigcap$-preserving operations on $\Omega$, continuous Weihrauch degrees correspond to $\bigcap$-preserving monotone operations, and so on.

In summary, various notions of oracle computation (or reducibility) which have been studied in depth in computability theory and related areas can thus be characterized surprisingly neatly in the context of operations on truth-values.

%The oracle-computability notions mentioned so far are those that have been studied in the past in computability theory and related fields, but such oracle-computability notions do not give operations on $\Omega$ that do not preserve intersection.
By recasting the oracle-computability notions in terms of operations on the truth values, it now becomes clear that there are hidden oracle-computability notions corresponding to operations that do not preserve $\bigcap$, which have not been treated in classical computability theory.
These hidden oracle-computability notions were first dealt with in Bauer \cite{Bau21} and Kihara \cite{Kih21}, and this paper explains them more closely to the traditional computability theory setting.

To do so, Section \ref{sec:rep-sp-main} introduces the notion of (multi-)represented space, which plays a central role in computable analysis and realizability theory.
Section \ref{sec:rep-sp-main} also provides a systematic study of ``change of coding (jump of representation)'' and links this to the notion of universal closure operator (a notion related to Lawvere-Tierney topology).

In Section \ref{sec:mmmap}, we show that reducibility notions on multi-represented spaces allow us to correspond oracle-computability notions to operations on truth values precisely.
For example, the following correspondences are shown.
\begin{itemize}
\item Weihrauch degrees on multi-represented spaces correspond to monotone operations on $\Omega$ in the Kleene-Vesley topos.
\item Pointed Weihrauch degrees on multi-represented spaces correspond to monotone, inflationary, operations on $\Omega$ in the Kleene-Vesley topos.
\item Turing-Weihrauch degrees on multi-represented spaces correspond to Lawvere-Tierney topologies on the Kleene-Vesley topos.
\end{itemize}

%In Section 4, we describe the relativization of a partial combinatory algebra by an oracle.
%For instance, we observe that the notion of $\Sigma^\ast$-pointclass in descriptive set theory \cite{MosBook} yields a partial combinatory algebra; hence a realizability topos.

In Section \ref{sec:DST}, we give results suggesting the existence of a wealth of concrete examples.
In particular, we present various natural examples of (type-two) oracles and relative partial combinatory algebras (PCAs) from descriptive set theory.

\subsection{Preliminaries}\label{sec:notations}

In this article, we assume that the reader is familiar with elementary facts about computability theory.
We refer the reader to \cite{Cooper,OdiBook,RogBook,SoareBook} for the basics of computability theory; \cite{MosBook} for descriptive set theory; \cite{Handbook-CCA} for computable analysis; \cite{vOBook} for realizability theory; and \cite{Bau00} for the relationship between computable analysis and reallizability.

We denote a partial function from $X$ to $Y$ as $f\pcolon X\to Y$.
The image of $A\subseteq X$ under $f$ is written as $f[A]$ whenever ${\rm dom}(f)\subseteq A$. 
We use the symbol $\mathcal{P}(Y)$ to denote the power set of a set $Y$.
In this article, a partial function $f\pcolon X\to\mathcal{P}(Y)$ is often called a {\em partial multi-valued function} (abbreviated as a {\em multifunction} or a {\em multimap}), and written as $f\pcolon X\tto Y$.
In computable mathematics, we often view a $\forall\exists$-formula $S$ as a partial multifunction.
Informally speaking, a (possibly false) statement $S\equiv\forall x^X[Q(x)\rightarrow\exists y^YP(x,y)]$ is transformed into a partial multifunction $f_S\pcolon X\rightrightarrows Y$ such that ${\rm dom}(f_S)=\{x\in X:Q(x)\}$ and $f_S(x)=\{y\in Y:P(x,y)\}$.
Here, we consider formulas as partial multifunctions rather than relations in order to distinguish a {\em hardest} instance $f_S(x)=\emptyset$ (corresponding to a {\em false sentence}) and an {\em easiest} instance $x\in X\setminus{\rm dom}(f_S)$ (corresponding to a {\em vacuous truth}).
In this sense, a relation does not correspond to a partial multifunction.
More to the point, the category of partial multifunctions and that of relations have different morphism compositions \cite{Pau17}.
For partial multifunctions $f\pcolon X\tto Y$ and $g\pcolon Y\tto Z$, the composition $g\circ f\pcolon X\tto Z$ is defined as follows:
$x\in{\rm dom}(g\circ f)$ if $x\in{\rm dom}(f)$ and $y\in{\rm dom}(g)$ for any $y\in f(x)$; and then $g\circ f(x)=\bigcup\{g(y):y\in f(x)\}$.

\section{Abstract characterization of universal oracle computation}\label{sec:synthetic-dst}

\begin{notation}
We use the symbol $\tpbf{N}$ to denote either $\N$ or $\N^\N$.
To be precise, we consider one of the following systems:
\begin{enumerate}
\item First system ${\sf K}_1$: The basic object is $\tpbf{N}=\N$, and
\begin{itemize}
\item computable functions on $\tpbf{N}$ $=$ computable functions on $\N$.
\item continuous functions on $\tpbf{N}$ $=$ computable functions on $\N$.
\end{itemize}
\item Second system ${\sf K}^{\sf eff}_2$: The basic object is $\tpbf{N}=\N^\N$, and
\begin{itemize}
\item computable functions on $\tpbf{N}$ $=$ computable functions on $\N^\N$.
\item continuous functions on $\tpbf{N}$ $=$ topologically continuous functions on $\N^\N$.
\end{itemize}
\end{enumerate}

If the reader is only interested in computability on natural numbers, just consider the system ${\sf K}_1$; that is, the reader may proceed by replacing all $\tpbf{N}$ in the following with $\N$ and all ``continuous'' with ``computable'' (no distinction in made between ``continuous'' and ``computable'' in this case).

Modern computability theorists would also like to consider computations on the type {\tt nat\,->\,nat}.
In this case, consider the system ${\sf K}_2^{\sf eff}$:
The set $\N^\N$ of all functions on the natural numbers is equipped with the usual Baire topology, i.e., the product of the discrete topology on $\N$.
A partial continuous function $f$ on $\N^\N$ is determined by a monotone function on $\N^\ast$ (the set of finite strings of natural numbers) and any function on $\N^\ast$ can be coded as an element $\alpha_f$ in $\tpbf{N}=\N^\N$ (since a finite string can be coded as a natural number).

We call such an $\alpha_f$ a code of $f$, and let $\varphi_\alpha$ denote the partial continuous function on $\tpbf{N}$ coded by $\alpha\in\tpbf{N}$.
%If $\alpha\in\tplf{N}$, $\varphi_\alpha$ is a partial computable function on $\tpbf{N}$.
\end{notation}
%We also use $\tplf{N}\subseteq\tpbf{N}$ to denote the set of all computable functions on $\N$.
%There are three typical coding systems:

\subsection{Universal machine}\label{sec:sdst-notation}

%First, let us look at de Brecht-Pauly's approach \cite{dB14,PaBr15} to oracle computation.
Let us begin by reinterpreting de Brecht-Pauly's approach \cite{dB14,PaBr15} as an abstraction of oracle computation.
In order to do so, we need to clarify what a universal machine is.
Let $U^\alpha$ be a universal machine relative to an oracle $\alpha$.
Universality means that the $\alpha$-relative computation of any machine can be simulated by taking its code as input; that is, for any $\alpha$-computable function $f^\alpha$ there exists a code $e_f$ such that $U^\alpha(e_f,n)\simeq f^\alpha(n)$ for all $n$.
Here, note that $\alpha$-computability is equivalent to computability relative to $U^\alpha$, so $f^\alpha$ can be replaced with $f\circ U^\alpha$ for a computable function $f$.
By putting $F(n)=(e_f,n)$, the last equality becomes {$U^\alpha\circ F=f\circ U^\alpha$}.
In conclusion:
\begin{itemize}
\item[] if $U$ is a {\em universal oracle computation} then for any computable function $f\pcolon\N\to\N$ one can find a computable function $F\colon\N\to\N$ such that $U\circ F=f\circ U$.
\end{itemize}

In the case of a universal machine constructed in the standard manner, $f\mapsto F$ is in fact computable (i.e., given a code of $f$ one can effectively find a code of $F$).

In \cite{dB14}, de Brecht introduced its topological version (i.e., a type-two version).
%Before explaining this, we introduce some notation.
%
The following is a topological version of the above property of universal computation.
%, but it is safe to say that it is a generalization, via the argument in Section \ref{sec:realizability-PCA}.

\begin{definition}[de Brecht \cite{dB14}]\label{def:basic-computably-transparent}
A nonempty partial map $U\pcolon\tpbf{N}\to\tpbf{N}$ is {\it transparent} if for any continuous function $f\pcolon\tpbf{N}\to\tpbf{N}$ there exists a continuous function $F\pcolon\tpbf{N}\to\tpbf{N}$ such that $U\circ F=f\circ U$.
If such $f\mapsto F$ is computable (i.e., given a code of $f$ one can effectively find a code of $F$), we say that $U$ is {\it computably transparent}.

For such a $U$, a partial map $g\pcolon \tpbf{N}\to \tpbf{N}$ is {\it $U$-continuous} if there exists a continuous function $G\pcolon \tpbf{N}\to \tpbf{N}$ such that $g=U\circ G$.
If such $G$ is computable, we say that $g$ is {\it $U$-computable}.
\begin{align*}
\xymatrix{
\tpbfxy{N}\ar[rr]^f & & \tpbfxy{N}\\
\tpbfxy{N}\ar[rr]_{F} \ar[u]^(.4)U & & \tpbfxy{N}\ar[u]_(.4)U
}
& &
\xymatrix{
\tpbfxy{N}\ar[rr]^g \ar[drr]_{G} & & \tpbfxy{N}\\
& & \tpbfxy{N} \ar[u]_(.4)U
}
\end{align*}
\end{definition}

Note that $g$ is $U$-continuous if and only if $g$ is Wadge reducible (i.e., continuously many-one reducible) to $U$.
Similarly, $g$ is $U$-computable if and only if $g$ is many-one reducible to $U$ (in the sense of Section \ref{sec:background}).

\begin{example}[when $\tpbf{N}=\N^\N$]
The limit operator $\lim\pcolon\tpbf{N}^\N\to\tpbf{N}$ is computably transparent (via the identification $\tpbf{N}^\N\simeq\tpbf{N}$).
Moreover, a map $g\pcolon\tpbf{N}\to\tpbf{N}$ is $\lim$-continuous if and only if $g$ is $\tpbf{\Sigma}^0_2$-measurable \cite{dB14}.
\end{example}

\begin{remark}
A transparent map is always surjective; consider a constant function $f$ that takes a given value.
\end{remark}

\begin{remark}
In the original terminology by de Brecht \cite{dB14}, a transparent map is referred to as a {\em jump operator}.
The term transparent was coined in \cite{BGM12}.
\end{remark}

So the notion of universal oracle computation can be neatly abstracted in this way, but the notion becomes more interesting when we consider multivalued oracles.
Transparency for multimaps is defined in the same way as for single-valued maps.
We use the refinement relation $\preceq$ introduced in Section \ref{sec:background}.

\begin{definition}\label{def:pmultimap-transparent}
A nonempty partial multimap $U\pcolon\tpbf{N}\tto\tpbf{N}$ is {\em transparent} if for any continuous function $f\pcolon\tpbf{N}\to\tpbf{N}$ there exists a continuous function $F\pcolon\tpbf{N}\to\tpbf{N}$ such that $U\circ F\preceq f\circ U$.
If such $f\mapsto F$ is computable, we say that $U$ is {\it computably transparent}.

For such a $U$, a partial multimap $g\pcolon \tpbf{N}\tto \tpbf{N}$ is {\it $U$-continuous} if there exists a continuous function $G\pcolon \tpbf{N}\to \tpbf{N}$ such that $U\circ G\preceq g$.
If such $G$ is computable, we say that $g$ is {\it $U$-computable}.
\end{definition}

Note that the only difference between these definitions for multivalued maps and those for single-valued maps is that the equality relation is replaced by the refinement relation $\preceq$.
To be explicit, $U$ is transparent iff for any continuous $f$ there exists a continuous $F$ such that $U(F(x))\subseteq f[U(x)]$ for any $x\in{\rm dom}(U)$, and $g$ is $U$-continuous iff there exists a continuous $G$ such that $G(x)\in{\rm dom}(U)$ and $U(G(x))\subseteq g(x)$ for any $x\in{\rm dom}(g)$.
%Using the refinement relation $\preceq$ (introduced in Section \ref{sec:background}), we can write $U\circ F\preceq f\circ U$ for the innermost formula in the definition of transparency, and $U\circ G\preceq g$ for $U$-continuity.
As before, $g$ is $U$-computable if and only if $g$ is many-one reducible to $U$ (in the sense of Section \ref{sec:background}).
%Note also that $U$ is transparent if and only if $f\circ U\leq_mU$ for any continuous function $f$.

In computability theory, multifunctions are used to represent problems that can have many solutions.
An example of a tool that can be used to measure the complexity of a problem with many possible solutions is the notion of Medvedev reducibility \cite{Medvedev}; see also Section \ref{sec:intro-mass-problems}.
For $P,Q\subseteq\tpbf{N}$ (which are considered to be solution sets of some problems), recall that $P$ is Medvedev reducible to $Q$ (written $P\leq_MQ$) if there exists a computable function which, given a solution of $Q$, returns a solution of $P$; that is, there exists a partial computable function $\varphi$ on $\tpbf{N}$ such that $Q\subseteq{\rm dom}(\varphi)$ and $\varphi[Q]\subseteq P$.

\begin{example}[Medvedev oracle]\label{exa:oracle-Medvedev-reducibility}
For a set $Q\subseteq\tpbf{N}$, the universal computation relative to the Medvedev oracle $Q$, ${\tt Med}_Q\pcolon\tpbf{N}\tto\tpbf{N}$, is defined as follows:
\begin{align*}
{\rm dom}({\tt Med}_Q)&=\{\tau\in\tpbf{N}:Q\subseteq{\rm dom}(\varphi_\tau)\}\\
{\tt Med}_Q(\tau)&=\varphi_\tau[Q]
%{\rm dom}({\tt Med}_Q)&:=\{\tau\in\N^\N:(\forall x\in Q)\;\tau\ast x\downarrow\}=\{\tau\in\N^\N:Q\subseteq{\rm dom}(\varphi_\tau)\}\\
%{\tt Med}_Q(\tau)&:=\{\tau\ast x:x\in Q\}=\varphi_\tau[Q]
\end{align*}

Then, ${\tt Med}_Q$ is computably transparent.
To see this, note that, given codes $\sigma,\tau$ of continuous functions, one can effectively find a code $b_{\sigma,\tau}$ of the composition $\varphi_\sigma\circ\varphi_\tau$.
Then we have $\varphi_\sigma[{\tt Med}_Q(\tau)]=\varphi_\sigma\circ\varphi_\tau[Q]={\tt Med}_Q(b_{\sigma,\tau})$.

Moreover, the constant multimap $x\mapsto P\subseteq\tpbf{N}$ is ${\tt Med}_Q$-computable if and only if $P\leq_MQ$.
This is because $x\mapsto P$ is ${\tt Med}_Q$-computable if and only if ${\tt Med}_Q(\tau)\subseteq P$ for some computable $\tau\in\tplf{N}$ if and only if $P\leq_MQ$ via $\varphi_\tau$.
\end{example}

As a useful tool for comparing multifunctions, the most used in modern computable analysis is the notion of Weihrauch reducibility \cite{CCA}; see also Section \ref{sec:intro-search-problems}.
For partial multifunctions $f,g\pcolon\tpbf{N}\to\tpbf{N}$, recall that {$f$ is Weihrauch reducible to $g$} (written $f\leq_Wg$) if there exist computable functions $h,k\pcolon\tpbf{N}\to\tpbf{N}$ such that, for any $x\in{\rm dom}(f)$, $h(x)\in{\rm dom}(g)$ and for any solution $y\in g(h(x))$, $k(x,y)\in f(x)$.
%This notion seems a bit artificial, but from a certain point of view, it is computationally quite natural.
Simply put, it states that, during a process solving $f(x)$, values of the oracle $g$ can be accessed only once.

\begin{example}[Weihrauch oracle]\label{exa:oracle-Weihrauch-reducibility}
For a partial multifunction $g\pcolon\tpbf{N}\tto\tpbf{N}$, the universal computation relative to the Weihrauch oracle $g$, ${\tt Weih}_g\pcolon\tpbf{N}\tto\tpbf{N}$, is defined as follows:
\begin{align*}
{\rm dom}({\tt Weih}_g)&=\{(h,k,x)\in\tpbf{N}:\varphi_h(x)\in{\rm dom}(g)\\
&\qquad\mbox{ and }\varphi_k(x,y)\downarrow\mbox{ for all $y\in g(\varphi_h(x))$}\}\\
{\tt Weih}_g(h,k,x)&=\{\varphi_k(x,y):y\in g(\varphi_h(x))\}
%{\rm dom}({\tt Med}_Q)&:=\{\tau\in\N^\N:(\forall x\in Q)\;\tau\ast x\downarrow\}=\{\tau\in\N^\N:Q\subseteq{\rm dom}(\varphi_\tau)\}\\
%{\tt Med}_Q(\tau)&:=\{\tau\ast x:x\in Q\}=\varphi_\tau[Q]
\end{align*}

Then, ${\tt Weih}_g$ is computably transparent.
This is because we have $\varphi_\ell[{\tt Weih}_g(h,k,x)]=\{\varphi_\ell\circ\varphi_k(x,y):y\in g(\varphi_h(x))\}={\tt Weih}_g(h,b_{\ell,k},x)$, where $b_{\ell,k}$ is a code of $\varphi_\ell\circ\varphi_k$.

Moreover, $f\pcolon\tpbf{N}\tto\tpbf{N}$ is ${\tt Weih}_g$-computable if and only if $f\leq_Wg$.
This is because $f\leq_Wg$ via $h,k$ if and only if, for any $x\in{\rm dom}(f)$, $(h,k,x)\in{\rm dom}({\tt Weih}_g)$ and ${\tt Weih}_g(h,k,x)\subseteq f(x)$ if and only if $f$ is ${\tt Weih}_g$-computable via $x\mapsto(h,k,x)$.
Conversely, if $f$ is ${\tt Weih}_g$-computable via $x\mapsto (h(x),k(x),z(x))$ then $f\leq_Wg$ via $x\mapsto\varphi_{h(x)}(z(x))$ and $(x,y)\mapsto\varphi_{k(x)}(z(x),y)$.
\end{example}

\begin{remark}
One might think that the universal computation relative to a strong Weihrauch oracle could be defined in a similar way.
However, for ${\tt SWeih}_g$ defined in a straightforward manner, it is not possible to show the equivalence of ${\tt SWeih}_g$-computability and strong Weihrauch reducibility to $g$.
The computational understanding of strong Weihrauch reduction (generalized Galois-Tukey connection) leads to the peculiar situation that certain information  (input data) once obtained is lost in the middle of the computation, and is not accessible at the stage of computing an output; so it seems difficult to treat it as a type of oracle computation in our framework.
\end{remark}

With respect to Weihrauch reducibility, in general, statements such as ${\rm id}\leq_Wg$ or $g\circ g\leq_Wg$ are not always true.
It is convenient to give special names to oracle computations that have these properties.
%First, to make the idea easier to understand, we introduce the following notion for single-valued functions.

\begin{definition}\label{def:pmultimap-transparent2}
A transparent multimap $U\pcolon\tpbf{N}\tto\tpbf{N}$ is {\em inflationary} if the identity map is $U$-computable, and $U$ is {\em idempotent} if $U\circ U$ is $U$-computable.
\end{definition}

Observe that if $U$ is an inflationary computably transparent map, then every continuous function is $U$-continuous.
Moreover, if $U$ is idempotent, then the composition of two $U$-continuous functions is also $U$-continuous.
For instance, the limit operation, $\lim$, is inflationary, but not idempotent.
This is because $\lim\circ\lim$-continuity is known to be equivalent to $\tpbf{\Sigma}^0_3$-measurability \cite{dB14} while $\lim$-continuity is equivalent to $\tpbf{\Sigma}^0_2$-measurability as mentioned above.

How these notion are used in synthetic descriptive set theory and how they relate to the second perspective on oracles (i.e., ``oracles change the way we access spaces'') will be explained in Section \ref{sec:represented-space}.

%These notions for multmaps are defined in the same way.

\begin{remark}[Lifschitz realizability]
Aside from synthetic descriptive set theory, Kihara \cite{Kih20}, for example, also focuses on these three basic concepts (computable transparency, being inflationary, and idempotence) for multifunctions.
The reason for focusing on these concepts was to extract the conditions for why Lifschitz realizability (realizability with bounded $\Pi^0_1$ sets; \cite{Lif,vO,ChRa12}) works.
To be more precise, computable transparency corresponds to \cite[Lemma 4]{Lif}, \cite[Lemma 5.7]{vO}, and \cite[Lemma 4.4]{ChRa12}, being inflationary corresponds to \cite[Lemma 2]{Lif}, \cite[Lemma 5.4]{vO}, and \cite[Lemma 4.2]{ChRa12}, and idempotence corresponds to \cite[Lemma 3]{Lif}, \cite[Lemma 5.6]{vO}, and \cite[Lemma 4.5]{ChRa12}.
% and to extend Lifschitz realizability to realizability relative to more general multivalued oracles;
For Lifschitz realizability, see also Section \ref{sec:realizability-relative-to-oracle}.
\end{remark}

\subsection{Relative partial combinatory algebra}\label{sec:realizability-PCA}

Let us organize the current setup.
We have a basic object $\tpbf{N}$.
%a pair $(\tplf{N},\tpbf{N})$ of sets such that $\tplf{N}\subseteq\tpbf{N}$.
Moreover,  we have an indexed family $(\varphi_x:x\in\tpbf{N})$ of partial continuous functions on $\tpbf{N}$.
In such a case, one may define a partial binary operation $\ast\pcolon\tpbf{N}^2\to\tpbf{N}$ by $x\ast y=\varphi_x(y)$ for any $y\in{\rm dom}(\varphi_x)$.
For simplicity, we also use the notation $x\ast y\ast z$ to denote $(x\ast y)\ast z$, and $x\ast A$ to denote $\varphi_x[A]$ for $A\subseteq\tpbf{N}$.

Moreover, continuous functions are encoded by elements of $\tpbf{N}$, but computable functions may be encoded by only elements of a small part ${\sf N}\subseteq\tpbf{N}$.
For instance, one may consider the following three systems:
\begin{enumerate}
\item The system ${\sf K}_1$: The basic objects are $\tplf{N}=\tpbf{N}=\N$, and $(\varphi_e)_{e\in\N}$ is the standard indexing of all partial computable functions on $\N$.
\item The system ${\sf K}_2$: The basic objects are $\tplf{N}=\tpbf{N}=\N^\N$, and $(\varphi_e)_{e\in\N^\N}$ is the standard indexing of all partial continuous functions on $\N^\N$.
\item The system ${\sf K}_2^{\sf eff}$: The baldface object $\tpbf{N}$ is the set $\N^\N$, the lightface object $\tplf{N}$ is the set of all computable elements in $\N^\N$, and $(\varphi_x)_{x\in\N^\N}$ is the standard indexing of all partial continuous functions on the Baire space $\N^\N$.
In this case, $(\varphi_x)_{x\in{\sf N}}$ provides an indexing of all partial computable functions on $\N^\N$.
\end{enumerate}

From top to bottom, the triple $(\tplf{N},\tpbf{N},\ast)$ is known as the {\em Kleene first algebra}, the {\em Kleene second algebra}, and the {\em Kleene-Vesley algebra}.
Our definition can be generalized as follows:

\begin{definition}
Given a triple $(\tplf{N},\tpbf{N},\ast)$, one may define $\varphi_x\pcolon\tpbf{N}\to\tpbf{N}$ by $\varphi_x(y)=x\ast y$ if $x\ast y$ is defined (written $x\ast y\downarrow$).
A partial function $f\pcolon\tpbf{N}\to\tpbf{N}$ is said to be {\em continuous} if there exists $e\in\tpbf{N}$ such that $f(x)=\varphi_e(x)$ for any $x\in{\rm dom}(f)$; and $f$ is {\em computable} if such an $e$ can be chosen from the lightface part $\tplf{N}$.
\end{definition}

\begin{example}
In the Kleene first algebra, continuity and computability are equivalent, and both coincide with computability in the standard sense.
In the Kleene second algebra, continuity and computability are equivalent, and both coincide with topological continuity (on Baire space).
In the Kleene-Vesley algebra, continuity coincides with topological continuity, and computability coincides with computability in the standard sense.
\end{example}

Most of the results in this article hold if a triple $(\tplf{N},\tpbf{N},\ast)$ satisfies the condition known as a relative partial combinatory algebra (see \cite{vOBook} and also Definition \ref{def:partial-combinatory-algebra-int}).
However, the three algebras mentioned above are the main ones used in applications in this section (although many natural and important examples of relative PCAs are known; see e.g.~\cite{vOBook} and also Section \ref{sec:pointclass-partial-combinatory-algebra}).
To lower the barrier to entry into our research, we emphasize ease of understanding over generality.
Once understood, generalization is easy.
For now, let us proceed with the above three algebras in mind.

First, it is useful to rephrase Definitions \ref{def:pmultimap-transparent} and \ref{def:pmultimap-transparent2} using these terms.

\begin{obs}\label{obs:transparent-multmap}
Let $U\pcolon\tpbf{N}\tto\tpbf{N}$ be a partial multimap.
\begin{itemize}
\item $U$ is computably transparent if and only if
there exists ${\tt u}\in\tplf{N}$ such that for all ${\tt f},{\tt x}\in \tpbf{N}$, whenever ${\tt f}\ast U({\tt x})$ is defined (i.e., ${\tt x}\in{\rm dom}(U)$ and ${\tt f}\ast{\tt y}$ is defined for any ${\tt y}\in U({\tt x})$),
\[{\tt u\ast f\ast x}\in{\rm dom}(U)\mbox{ and }U({\tt u\ast f\ast x})\subseteq{\tt f}\ast U({\tt x}).\]

\item $U$ is {inflationary} if and only if
there exists $\eta\in\tplf{N}$ such that for all ${\tt x}\in \tpbf{N}$,
\[\eta\ast {\tt x}\in{\rm dom}(U)\mbox{ and }U(\eta\ast {\tt x})\subseteq\{{\tt x}\}.\]

\item $U$ is {idempotent} if and only if there exists $\mu\in\tplf{N}$ such that for all ${\tt x}\in{\rm dom}(U\circ U)$,
\[\mu\ast{\tt x}\in{\rm dom}(U)\mbox{ and }U(\mu\ast {\tt x})\subseteq U\circ U({\tt x}).\]
\end{itemize}
\end{obs}

\subsection{Reflective subcategory}\label{sec:left-adjoint}

From this point forward, we will discuss the exact relationship between computable transparency and universal computation.
For $f,g\pcolon\tpbf{N}\tto\tpbf{N}$, recall that $f$ is many-one reducible to $g$ (written $f\leq_mg$) if there exists a computable function $\varphi$ such that $g\circ\varphi$ refines $f$; that is, for any $x\in{\rm dom}(f)$, $\varphi(x)$ is defined and $g(\varphi(x))\subseteq f(x)$.

\begin{definition}
Let $\mathcal{MR}ed$ be the set of all partial multimaps on $\tpbf{N}$ preordered by many-one reducibility $\leq_m$.
The restriction of this preordered set to those that are computably transparent, inflationary, and idempotent, respectively, is expressed by decorating it with superscripts ${\rm ct}$, $\eta$, and $\mu$, respectively.
\end{definition}

Recall the notion of universal computation ${\tt Weih}_f$ relative to a Weihrauch oracle $f$ introduced in Example \ref{exa:oracle-Weihrauch-reducibility}.
As we have seen, the construction ${\tt Weih}\colon g\mapsto{\tt Weih}_g$ yields a computably transparent map from a given multimap.
Moreover, it is easy to see that $f\leq_mg$ implies ${\tt Weih}_f\leq_m{\tt Weih}_g$.
Hence, ${\tt Weih}$ can be viewed as an order-preserving map from $\mathcal{MR}ed$ to $\mathcal{MR}ed^{\rm ct}$.
The following guarantees that ${\tt Weih}(f)={\tt Weih}_f$ is the $\leq_m$-least computably transparent map which can compute $f$.

\begin{prop}\label{thm:computably-transparent-many-one-Weihx}
The order-preserving map ${\tt Weih}\colon\mathcal{MR}ed\to\mathcal{MR}ed^{\rm ct}$ is left adjoint to the inclusion $i\colon\mathcal{MR}ed^{\rm ct}\mono\mathcal{MR}ed$.
In other words, for any $f,g\pcolon\tpbf{N}\tto\tpbf{N}$, if $g$ is computably transparent, then $f\leq_mg$ if and only if ${\tt Weih}(f)\leq_mg$.
\end{prop}

\begin{proof}
Obviously, ${\tt Weih}(f)\leq_mg$ implies $f\leq_mg$.
For the converse direction, if $f\leq_mg$ then ${\tt Weih}(f)\leq_m{\tt Weih}(g)$ by monotonicity; hence it suffices to show that ${\tt Weih}(g)\leq_mg$.
Given ${\tt h},{\tt k},{\tt x}\in \tpbf{N}$, note that ${\tt y}\in g({\tt h}\ast{\tt x})$ implies ${\tt k}\ast\langle{\tt x},{\tt y}\rangle\in {\tt Weih}_g({\tt h},{\tt k},{\tt x})$.
Put ${\tt k'}=\lambda xy.{\tt k}\ast\langle x,y\rangle$.
As $g$ is computably transparent, there exists ${\tt u}\in N$ such that $g({\tt u}\ast ({\tt k'\ast x})\ast ({\tt h\ast x}))\subseteq ({\tt k'\ast x})\ast g({\tt h\ast x})$.
Note that any element of $({\tt k'\ast x})\ast g({\tt h\ast x})$ is of the form ${\tt k'\ast x\ast y}$ for some ${\tt y}\in g({\tt h\ast x})$.
We also have ${\tt k'\ast x\ast y}={\tt k}\ast \langle{\tt x},{\tt y}\rangle\in {\tt Weih}_g({\tt h},{\tt k},{\tt x})$.
Hence, we get $g({\tt u}\ast ({\tt k'\ast x})\ast ({\tt h\ast x}))\subseteq {\tt Weih}_g({\tt h},{\tt k},{\tt x})$, so the term $\lambda hkx.{\tt u}\ast (({\tt a}\ast k)\ast x)\ast (h\ast x)\in\tplf{N}$ witnesses ${\tt Weih}_g\leq_mg$, where ${\tt a}\in\tplf{N}$ is a computable element such that ${\tt a}\ast{\tt k}={\tt k}'$.
\end{proof}

This implies that the notions of many-one reducibility and Weihrauch reducibility coincide on computably transparent multifunctions.

%For a preorder, its quotient by the induced equivalence relation is called the poset reflection.

\begin{cor}\label{prop:computably-transparent-many-one-Weih}
%The poset reflection of $\mathcal{MR}ed^{\rm ct}$ (i.e.,
The poset of many-one degrees of computably transparent, partial multimaps on $\tpbf{N}$ is isomorphic to the Weihrauch degrees.
\end{cor}

\begin{proof}
We first claim that if $g$ is computably transparent, then $f\leq_mg$ if and only if $f\leq_Wg$.
The forward direction is obvious.
For the backward direction, first note that $f\leq_Wg$ if and only if $f\leq_m{\tt Weih}_g$ as seen in Example \ref{exa:oracle-Weihrauch-reducibility}.
By Proposition \ref{thm:computably-transparent-many-one-Weihx}, ${\tt Weih}_g\leq_mg$ since $g$ is computably transparent.
Hence, $f\leq_mg$.
This claim ensures that the identity map is an embedding from the poset reflection of $\mathcal{MR}ed^{\rm ct}$ into the Weihrauch degrees.
For surjectivity, we have $f\equiv_W{\tt Weih}_f$ as seen in Example \ref{exa:oracle-Weihrauch-reducibility}, and ${\tt Weih}_f$ is computably transparent.
\end{proof}

\begin{remark}
By Proposition \ref{thm:computably-transparent-many-one-Weihx}, the notion of Weihrauch reducibility automatically appears when we have the notions of computable transparency and many-one reducibility, so in fact, there is no need to define Weihrauch reducibility.
More precisely, any left-adjoint ${\tt W}$ of the inclusion $i\colon\mathcal{MR}ed^{\rm ct}\mono\mathcal{MR}ed$ can be seen as a universal Weihrauch machine since ${\tt W}(g)\equiv_m{\tt Weih}(g)$, so $f\leq_Wg$ if and only if $f\leq_m{\tt W}(g)$.
The point is that ${\tt W}$ has no meaning a priori; nevertheless ${\tt W}$ somehow automatically recovers the notion of Weihrauch reducibility.
\end{remark}

The question then naturally arises as to what corresponds to being inflationary or idempotent.
First, let us consider the notion corresponding to being inflationary.
Recall that the notion of Weihrauch reduction requires that {\em exactly one} query to an oracle must be made during the computation.
Let us weaken this condition so that it is not necessary to make a query to an oracle.
That is, it is a relative computation with {\em at most one} query to an oracle.
When no query is made to an oracle, it is a bare computation without an oracle.
Formally, if an oracle $g$ is given, this relative computability notion would be equivalent to making exactly one query to the oracle ${\rm id}\sqcup g$.
Here,
\begin{align*}
{\rm dom}(f\sqcup g)&=\{\langle\underline{\tt 0},x\rangle:x\in{\rm dom}(f)\}\cup\{\langle\underline{\tt 1},x\rangle:x\in{\rm dom}(g)\}\\
(f\sqcup g)(i,x)&=
\begin{cases}
f(x)&\mbox{ if }i=\underline{\tt 0},\\
g(x)&\mbox{ if }i=\underline{\tt 1}.
\end{cases}
\end{align*}

For $f,g\pcolon\tpbf{N}\tto\tpbf{N}$, we say that $f$ is {\em pointed Weihrauch reducible to $g$} (written $f\leq_{pW}g$) if $f\leq_W{\rm id}\sqcup g$; see also \cite{HiPa}.
Moreover, the universal computation relative to the pointed Weihrauch oracle $g$ is defined by ${\tt pWeih}(g):={\tt Weih}({\rm id}\sqcup g)$.
Note that ${\tt pWeih}(g)$ is inflationary, so ${\tt pWeih}$ can be viewed as an order-preserving map from $\mathcal{MR}ed$ to $\mathcal{MR}ed^{{\rm ct},\eta}$.
The following guarantees that ${\tt pWeih}(f)$ is the $\leq_m$-least computably transparent, inflationary, map which can compute $f$.

\begin{prop}\label{thm:computably-transparent-many-one-pWeihx}
The order-preserving map ${\tt pWeih}\colon\mathcal{MR}ed\to\mathcal{MR}ed^{{\rm ct},\eta}$ is left adjoint to the inclusion $i\colon\mathcal{MR}ed^{{\rm ct},\eta}\mono\mathcal{MR}ed$.
In other words, for any $f,g\pcolon\tpbf{N}\tto\tpbf{N}$, if $g$ is computably transparent and inflationary, then $f\leq_mg$ if and only if ${\tt pWeih}(f)\leq_mg$.
\end{prop}

\begin{proof}
As in the proof of Proposition \ref{thm:computably-transparent-many-one-Weihx}, it suffices to show that ${\tt pWeih}(g)\leq_mg$ for any computably transparent, inflationary, multimap $g$.
This is equivalent to say that ${\tt Weih}({\rm id}\sqcup g)\leq_mg$.
By Proposition \ref{thm:computably-transparent-many-one-Weihx}, we have already shown that ${\tt Weih}({\rm id}\sqcup g)\leq_m{\rm id}\sqcup g$.
Thus, it suffices to show that ${\rm id}\sqcup g\leq_mg$.
The following process gives such a reduction:
Given $(i,x)$, if $i=\underline{\tt 0}$ then return $\eta\ast x$, where $\eta$ is a witness that $g$ is inflationary.
If $i=\underline{\tt 1}$ then return $x$.
This process is computable, and witnesses ${\rm id}\sqcup g\leq_mg$.
\end{proof}

\begin{cor}\label{prop:computably-transparent-many-one-pWeih}
%The poset reflection of $\mathcal{MR}ed^{\rm ct}$ (i.e.,
The poset of many-one degrees of computably transparent, inflationary, partial multimaps on $\tpbf{N}$ is isomorphic to the pointed Weihrauch degrees.
\end{cor}

\begin{proof}
For $g\in\mathcal{MR}ed^{{\rm ct},\eta}$, we claim that $f\leq_{m}g$ if and only if $f\leq_{pW}g$.
This is because the latter is equivalent to $f\leq_{W}{\rm id}\sqcup g$, and by Proposition \ref{thm:computably-transparent-many-one-Weihx}, this is equivalent to $f\leq_m{\tt Weih}({\rm id}\sqcup g)={\tt pWeih}(g)\leq_mg$, where the last inequality follows from Proposition \ref{thm:computably-transparent-many-one-pWeihx} since $g$ is inflationary.
This claim ensures that the identity map is an embedding.
The surjectivity of the identity map follows from the fact that ${\tt pWeih}(g)$ is computably transparent and inflationary, and ${\tt pWeih}(g)\equiv_{pW}g$.
\end{proof}

Next, let us discuss the idempotent version of Weihrauch reducibility, which is a reducibility notion that, like Turing reducibility, allows multiple access to the oracle.
This is usually defined by what is called a {\em reduction game}.

\begin{definition}[see e.g.~{\cite[Definitions 4.1 and 4.3]{HiJo16}}]\label{def:game-Turing-Weihrauch-reducibility}
For multifunctions $f$ and $g$, let us consider the following perfect information two-player game $\mathfrak{G}(f,g)$:
%\[
%\begin{array}{rcccccccccc}
%\Me\colon	& (x_0\mid c_0)	&		&		& x_1	&		&		& x_2	&		& 		& \dots \\
%\Ar\colon	&				& y_0	&		&		& y_1	& 		&		& y_2	& 		& \dots \\
%\Ni\colon		&				&		& z_0	&		&		& z_1	& 		&		& z_2	& \dots
%\end{array}
%\]
\[
\begin{array}{|cc|}\hline 
 {\tt Oracle} & {\tt Computer} \\[0.5em]
z_0\in{\rm dom}(f)	& \\
				& {\tt Query}\colon \;x_0\in{\rm dom}(g)	\\
z_1\in g(x_0)	& \\	 
				& {\tt Query}\colon \;x_1\in{\rm dom}(g)	\\
z_2\in g(x_1)	& \\	 
			\vdots	& \vdots 	\\
%\vdots	& & \\	 
				& {\tt Query}\colon \;x_n\in{\rm dom}(g)	\\
z_{n+1}\in g(x_n)	& \\	 
 		& {{\tt Halt}}\colon\; x_{n+1}\in f(z_0)  \\[0.5em] \hline
\end{array}
\]

%Each player chooses a natural number at each round.
%Hereafter, we use $[p]$ to denote the partial continuous function on $\N ^\N $ coded by $p$.

\noindent
{\it Game rules:}
Here, the players need to obey the following rules.
\begin{itemize}
\item First, {\tt Oracle} chooses $z_0\in{\rm dom}(f)$.
\item At the $n$th round, {\tt Computer} reacts with $y_n=\langle {\tt A}_n,x_n\rangle$.
\begin{itemize}
\item The choice ${\tt A}_n={\tt Query}$ (coded by $\underline{\tt 0}$) indicates that {\tt Computer} makes a new query $x_n$ to $g$.
\item The choice ${\tt A}_n={\tt Halt}$ (coded by $\underline{\tt 1}$) indicates that {\tt Computer} declares termination of the game with $x_n$.
\end{itemize}
\item At the $(n+1)$st round, {\tt Oracle} responds to the query made by {\tt Computer} at the previous stage.
This means that $z_{n+1}\in g(x_n)$.
\end{itemize}

Then, {\em {\tt Computer} wins the game $\mathfrak{G}(f,g)$} if either {\tt Oracle} violates the rule before {\tt Computer} violates the rule, or {\tt Computer} obeys the rule and declares termination with $x_n\in f(z_0)$.

\medskip

\noindent
{\it Strategies:}
{\tt Computer}'s strategy is a partial map $\tau\pcolon \tpbf{N}^{<\om}\to 2\times\tpbf{N}$.
{\tt Computer}'s strategy is always computable or continuous, coded in the lightface part $\tplf{N}$ for the former case and in $\tpbf{N}$ for the latter.
%Moreover, we require that \Art's moves are chosen in a computable manner.
%In other words, \Art's strategy is a code $\tau$ of a partial {\em computable} function $h_\tau\pcolon\N ^{<\N }\to\N$, which reads \Mer's moves $x_0,\dots,x_n$ and then returns $y_n$.
On the other hand, {\tt Oracle}'s strategies are any partial functions (which are not necessarily computable).

If $\sigma$ and $\tau$ are strategies of {\tt Oracle} and {\tt Computer}, respectively, then the play that follow these strategies are defined as follows:
{\tt Oracle}'s first move is $z_0:=\sigma()$, and ($n+1$)th move is $z_{n+1}:=\sigma(x_0,\dots,x_n)$.
{\tt Computer}'s $n$th move is $\langle{\tt A}_n,x_n\rangle:=\tau(z_0,\dots,z_n)$.
Here, ``{\em undefined}'' counts as a rule violation.

An {\tt Computer}'s strategy $\tau$ is {\em winning} if, as long as {\tt Computer} follows the strategy $\tau$, {\tt Computer} wins the game, no matter what {\tt Oracle}'s strategy $\sigma$ is.
\end{definition}

\begin{definition}
Let $f$ and $g$ partial multifunctions.
We say that {\em $f$ is Turing-Weihrauch reducible to $g$} (written $f\leq_{TW} g$) if {\tt Computer} has a computable winning strategy for the game $\mathfrak{G}(f,g)$.
\end{definition}

\begin{remark}
Note that $f$ is Weihrauch reducible to $g$ if and only if {\tt Computer} has a computable ``exactly one-query'' strategy for the game $\mathfrak{G}(f,g)$; that is, ${\tt A}_1={\tt Halt}$.
Similarly, $f$ is pointed Weihrauch reducible to $g$ if and only if {\tt Computer} has a computable ``at most one-query'' strategy for the game $\mathfrak{G}(f,g)$; that is, ${\tt A}_i={\tt Halt}$ for some $i<2$.
\end{remark}

\begin{remark}
This idempotent version of Weihrauch reducibility is usually called generalized Weihrauch reducibility \cite{HiJo16}.
%The notion of degree is always assigned to reducibility; therefore, it is necessary to name the notion of degree corresponding to generalized Weihrauch reducibility, but it would be a disaster to name it ``generalized Weihrauch degree'' (as it is not a generalization of Weihrauch degree).
As explained in Section \ref{sec:background}, this reducibility can be viewed as a fusion of Weihrauch reducibility and Turing reducibility.
Therefore, from this perspective, this article will refer to this reducibility as {\em Turing-Weihrauch reducibility} (and use the symbol $\leq_{TW}$) rather than generalized Weihrauch reducibility.
\end{remark}

The idea of describing the process of Turing-like reduction using a game has some precedents, e.g., a {\em dialogue} \cite{vO06,FavO16}.
A notion essentially almost equivalent to Turing-Weihrauch reducibility has been studied by Lee-van Oosten \cite{LvO} (see also \cite[Remark 2.15]{Kih21}) in an intricate tree form (dealing with a more general setting) rather than in game form.
It was Hirschfeldt-Jockusch \cite[Definitions 4.1 and 4.3]{HiJo16} who gave a very clear and intuitive formulation of this notion in game form.
One of their aims was to give a clear computability-theoretic characterization of $\om$-model separation of $\Pi^1_2$ principles in second-order arithmetic \cite[Proposition 4.2]{HiJo16}.
Neumann-Pauly \cite{NePa18} gave another formulation using register machines.

%In a similar way as above, we can give a characterization of Turing-Weihrauch reducibility.
%Recall that a universal computation for a Weihrauch oracle is a computational model that accesses the oracle exactly once.
%In the case of a Turing-Weihrauch oracle, the number of times the oracle is accessed can be any number of times or even zero.
%For the basics of Turing-Weihrauch reducibility (i.e., generalized Weihrauch reducibility) $\leq_{TW}$, see \cite{HiJo16,Kih20}.
How to introduce a {\em universal} computation for Turing-Weihrauch reducibility has been also studied in depth:
The universal computation relative to a partial Turing oracle (in a more general setting) is formulated as a dialogue \cite{vO06,FavO16}, and the universal computation relative to a Turing-Weihrauch oracle $g$ has a machine formulation \cite{NePa18,Wes21}, denoted $g^{\diamondsuit}$, and a game formulation \cite{Kih20,Kih21}, denoted $g^\Game$, which are essentially equivalent.

Note that the rule of the game $\mathfrak{G}(f,g)$ does not mention $f$ except for Player I's first move.
Hence, if we skip Player I's first move, we can judge if a given play follows the rule without specifying $f$.
Such a restricted game is denoted by $\mathfrak{G}(g)$.
% {\em \Ar and \Ni win the game $\mathfrak{G}(g)$} if either \Me violates the rule before \Ar or \Ni violates the rule, or both \Ar and \Ni obey the rule and \Ar declares termination.

\begin{definition}[{see e.g.~\cite[Definition 2.20]{Kih21}}]
Given a partial multifunction $h$, let us define the new partial multifunction $g^\Game$ as follows:
An input for $g^\Game$ is a {\tt Computer}'s continuous strategy $\tau$.
\begin{itemize}
\item $h^\Game(\tau)$ is defined only if, along any play in $\mathfrak{G}(g)$ following the strategy $\tau$, either {\tt Oracle} violates the rule before {\tt Computer} violates the rule, or {\tt Computer} obeys the rule and declares termination, whatever {\tt Oracle}'s strategy is.
\item $u\in g^\Game(\tau)$ if and only if there is a play in $\mathfrak{G}(g)$ that follows the strategy $\tau$ such that {\tt Computer} declares termination with $u$ at some round, where all players obey the rule.
\end{itemize}
\end{definition}

The idea of the definition of $g^\diamondsuit$ is essentially the same.
Note, however, that in Westrick's definition \cite{Wes21}, computable transparency is lost because all the information in a run of computation (or history of a play in the game) is taken as output.
Replacing $g^\diamondsuit$ with ${\tt Weih}(g^\diamondsuit)$ solves this problem.
Other definitions \cite{NePa18,Kih21} do not cause this problem.
In the following, we alway adopt a computably transparent version of $\diamondsuit$ (e.g., take $\diamondsuit:=\Game$ as above).

The following observations are only outlined here, as we will give a proof for the more general case in Section \ref{sec:universal-bimap}.
For the basic property of the diamond operator, the definition in \cite{NePa18} (or in \cite{Kih20,Kih21}) ensures that $f\leq_{TW}g$ if and only if $f\leq_mg^\diamondsuit$ (see \cite{Kih20}).
Moreover, as observed in \cite{Kih20}, $g^\diamondsuit$ is computably transparent, inflationary, and idempotent, and $g^\diamondsuit\equiv_{TW}g$ (this is a Turing-Weihrauch version of the property mentioned in Example \ref{exa:oracle-Weihrauch-reducibility}).
In particular, the diamond operator $\diamondsuit\colon g\mapsto g^\diamondsuit$ can be viewed as an order-preserving map from $\mathcal{MR}ed$ to $\mathcal{MR}ed^{{\rm ct},\eta,\mu}$.

Conversely, the result by Westrick \cite[Theorem 1]{Wes21} implies that if $g$ is computably transparent, inflationary, and idempotent, then $g^\diamondsuit\leq_Wg$.
By Proposition \ref{prop:computably-transparent-many-one-Weih}, this is equivalent to $g^\diamondsuit\leq_mg$.
Therefore:

\begin{fact}[essentially by Westrick \cite{Wes21}]\label{fact:Westrick}
The diamond operator $\diamondsuit\colon\mathcal{MR}ed\to\mathcal{MR}ed^{{\rm ct},\eta,\mu}$ is left adjoint to the inclusion $i\colon\mathcal{MR}ed^{{\rm ct},\eta,\mu}\mono\mathcal{MR}ed$.
\qed
\end{fact}

\begin{cor}\label{prop:computably-transparent-many-one-TWeih}
The poset of many-one degrees of computably transparent, inflationary, idempotent, partial multimaps on $\tpbf{N}$ is isomorphic to the Turing-Weihrauch degrees.
\end{cor}

\begin{proof}
For $g\in\mathcal{MR}ed^{{\rm ct},\eta,\mu}$, we claim that $f\leq_{m}g$ if and only if $f\leq_{TW}g$.
This is because we have $f\leq_{TW}g$ if and only if $f\leq_mg^{\diamondsuit}\leq_mg$ by Fact \ref{fact:Westrick}.
%By Proposition \ref{prop:computably-transparent-many-one-Weih}, the latter is equivalent to $f\leq_mg$ since $g$ is computably transparent.
This claim ensures that the identity map is an embedding.
The surjectivity of the identity map follows from the fact that $g^\diamondsuit$ is computably transparent, inflationary, and idempotent, and $g^\diamondsuit\equiv_{TW}g$ as mentioned above.
\end{proof}

In this way, various oracle-computability notions (degree notions) can be recast as properties of multimaps.
%\begin{table*}[t]\small
 % \caption{Correspondences between operations on $\Omega$ and degree-theoretic notions}
%  \label{table}
%\end{table*}
\begin{align*}
\mathcal{MR}ed^{\rm ct}&\simeq\mbox{the Weihrauch degrees}\\
&\quad \mbox{(oracle computability with exactly one query)}\\
\mathcal{MR}ed^{{\rm ct},\eta}&\simeq\mbox{the pointed Weihrauch degrees}\\
&\quad \mbox{(oracle computability with at most one query)}\\
\mathcal{MR}ed^{{\rm ct},\eta,\mu}&\simeq\mbox{the Turing-Weihrauch degrees}\\
&\quad \mbox{(oracle computability with finitely many queries)}
\end{align*}

Since a preordered set can be thought of as a category, one can organize the above results using the language of category theory.
If an inclusion functor $\mathcal{\iota}\colon\mathcal{C}\embed\mathcal{D}$ between categories has a left adjoint, then $\mathcal{C}\embed\mathcal{D}$ is called a {\em reflective subcategory}, and such a left adjoint is called a {\em reflector}.
That is, $\mathcal{MR}ed^{{\rm ct},\eta,\mu}\embed\mathcal{MR}ed^{{\rm ct},\eta}\embed\mathcal{MR}ed^{{\rm ct}}\embed\mathcal{MR}ed$ are reflective subcategories, and corresponding universal oracle relativization functors are reflectors.

\begin{center}
%\begin{tikzcd}
%\mathcal{MR}ed
%\arrow[rr, "{\tt Weih}"{name=A}, bend left=35] & &
%\mathcal{MR}ed^{\rm ct}
%\arrow[ll, "\iota"{name=B}, bend left=35]
%\arrow[rr, "{\tt pWeih}"{name=C}, bend left=35] & &
%\mathcal{MR}ed^{{\rm ct},\eta}
%\arrow[ll, "\iota"{name=D}, bend left=35]
%\arrow[rr, "\diamondsuit"{name=E}, bend left=35] & &
%\mathcal{MR}ed^{{\rm ct},\eta,\mu}
%\arrow[ll, "\iota"{name=F}, bend left=35]
%--- Adjunction Symbol
%\arrow[phantom, from=A, to=B, "\dashv" rotate=-90, no line]
%\arrow[phantom, from=C, to=D, "\dashv" rotate=-90, no line]
%\arrow[phantom, from=E, to=F, "\dashv" rotate=-90, no line]
%\end{tikzcd}
%\includegraphics[width=12cm]{adjoint1.png}
\includegraphics[width=15cm]{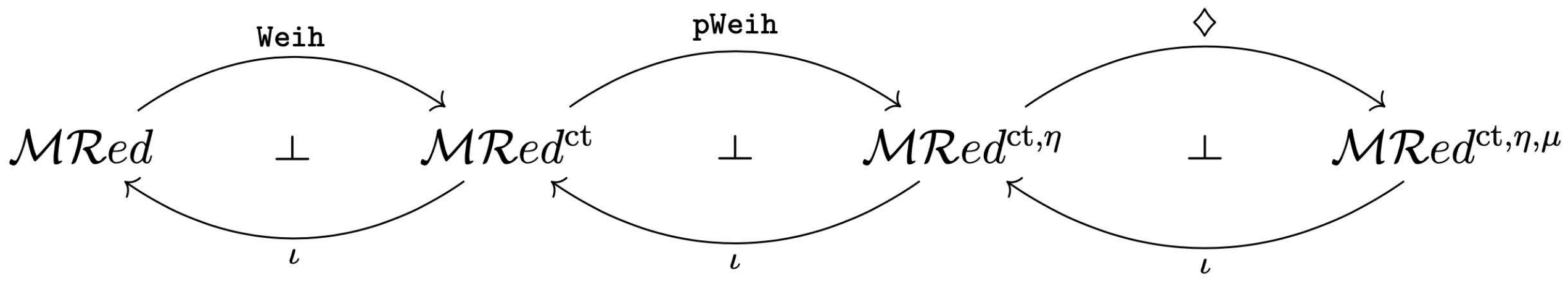}
\end{center}

Similar results hold when restricted to single-valued functions.
Here, the restriction of Turing-Weihrauch reducibility to total single-valued functions in the Kleene first algebra corresponds to ordinary Turing reducibility by definition.
Hereafter, we refer to the restriction of Turing-Weihrauch reducibility to partial single-valued functions as {\em subTuring reducibility}.
Note that, as mentioned in Section \ref{sec:background}, our subTuring reducibility for partial functions is Turing sub-reducibility in the sense of Sasso \cite{Sas75} and Madore \cite{Mad12}.
Under this terminology, it is straightforward to see the following:

\begin{cor}\label{cor:Turing-degrees-mna}
The poset of many-one degrees of computably transparent, inflationary, idempotent, partial single-valued maps on the Kleene first algebra is isomorphic to the subTuring degrees of partial functions.
\qed
\end{cor}

One might think that considering total single-valued functions, one would obtain a similar result for ordinary Turing reducibility; however usually a universal machine (even relative to a total oracle) cannot be total.
In other words, in general, computable transparency and totality seem incompatible.
Whether or not the characterization of ordinary Turing reducibility can be made is a subject for future work (see also Question \ref{question-ot} below).

We will see in Section \ref{sec:realizability} that via these translations of oracle-computability notions into multimap properties, it is possible to understand oracle-computations as operations on truth values.

\section{Oracles as operations on truth-values}\label{sec:realizability}

Next, let us explain the third point of view, which is to think of oracles as operations on truth-values.
In this section, we analyze this third view by comparing it to the second view introduced in Section \ref{sec:synthetic-dst}.
For this purpose, we translate the notions around transparent maps into the language regarding truth-values.

\subsection{Realizability and modality}\label{sec:modest-modality}

Before we begin, let us discuss a factor that changes truth values.
As described in Section \ref{sec:intro-oracle}, a factor that causes a change in Kripke semantics (see also \cite{GuHo19}) is a {\em coverage}.
In the usual intuitionistic Kripke semantics, in order to claim that $\varphi\lor\psi$ is valid at a state $p$, one must determine whether $\varphi$ or $\psi$ is valid at the state $p$, as in the usual BHK interpretation.
However, it is also useful to have a semantics that does not immediately decide which is valid, but rather postpones the decision of which is valid.
The notion of (abstract) coverage is used to realize various types of ``postpone''.
If there is a notion of covering of a state $p$, one can consider a modified semantics of the form that $\varphi\lor\psi$ is valid at a state $p$ if there exists a cover $\mathcal{U}$ of $p$ such that either $\varphi$ or $\psi$ holds locally in each state $q\in\mathcal{U}$ (which may be different for each $q$).
In this way, one can use a notion of covering to change the interpretation of the Kripke semantics.

\medskip
\noindent
{\it Abstract covering relation}:
Here are some typical properties of the covering relation $A\subseteq\bigcup\mathcal{U}$ (read that $A$ is covered by $\mathcal{U}$) for a set $A$ and family $\mathcal{U}$ of sets:
({\sf monotone}) Whenever $\mathcal{U}\subseteq\mathcal{V}$, if $A$ is covered by $\mathcal{U}$ then $A$ is covered by $\mathcal{V}$;
({\sf inflationary}) If $A\in\mathcal{U}$ then $A$ is covered by $\mathcal{U}$;
({\sf idempotent}) If $A$ is covered by $\mathcal{U}$ and all $B\in\mathcal{U}$ is covered by $\mathcal{V}$ then $A$ is covered by $\mathcal{V}$;
({\sf local}) $A$ is covered by $\mathcal{U}$ if and only if $A$ is covered by $\mathcal{U}|_A:=\{U\cap A:U\in\mathcal{U}\}$.
%More generally, if $A\in\mathcal{V}$, $A$ is covered by $\mathcal{U}$ if and only if $A$ is covered by $\{U\cap V:U\in\mathcal{U},V\in\mathcal{V}\}$.

Letting $j(\mathcal{U})$ be the set of all $A$ covered by $\mathcal{U}$, the above property can be rewritten as follows:
({\sf monotone}) $\mathcal{U}\subseteq\mathcal{V}$ implies $j(\mathcal{U})\subseteq j(\mathcal{V})$;
({\sf inflationary}) $\mathcal{U}\subseteq j(\mathcal{U})$;
({\sf idempotent}) $j\circ j(\mathcal{U})\subseteq j(\mathcal{U})$;
({\sf local}) $A\in j(\U)$ if and only if $A\in j(\U|_A)$.

\begin{obs}\upshape
({\sf local}) is equivalent to say that $\U|_A=\V|_A$ implies $j(\U)|_A=j(\V)|_A$.
\end{obs}

\begin{proof}
If $\U|_A=\V|_A$ then $\U|_{U\cap A}=\V|_{U\cap A}$ for any $U$, so locality implies $U\cap A\in j(\U)$ iff $U\cap A\in j(\U|_{U\cap A})$ iff $U\cap A\in j(\V|_{U\cap A})$ iff $U\cap A\in j(\V)$, which means $j(\U)|_A=j(\V)|_A$ since $j(\U)$ and $j(\V)$ are $\subseteq$-downward closed.
Conversely, we have $\U|_A=\U|_A|_A$, so we get $j(\U)|_A=j(\U|_A)|_A$, which clearly implies $A\in j(\U)$ iff $A\in j(\U|_A)$.
\end{proof}

Note that the $\subseteq$-downward closed families form a complete Heyting algebra, and the above characterization of ({\sf local}) can be rewritten as $(\mathcal{U}\iffarr\mathcal{V})\subseteq (j(\mathcal{U})\iffarr j(\mathcal{V}))$, where $\arr$ is the Heyting operation on the $\subseteq$-downward closed families.

This leads us to the following definition:
Let $\Omega$ be a complete Heyting algebra.
An operation $j\colon\Omega\to\Omega$ is a {\em nucleus} \cite{PiPu12} if the following conditions hold:
({\sf monotone}) $p\leq q$ implies $j(p)\leq j(q)$;
({\sf inflationary}) $p\leq j(p)$;
({\sf idempotent}) $j\circ j(p)\leq j(p)$;
({\sf local}) 
%$j(p)\land q=j(p\land q)\land q$.
%Here, one can see that, using the Heyting operation $\arr$ on $\Omega$, the condition ({\sf local}) can be rewritten as the condition $(p\iffarr q)\leq (j(p)\iffarr j(q))$.
$(p\iffarr q)\leq (j(p)\iffarr j(q))$.
By combining monotonicity and locality, a nucleus actually satisfies the condition ({\sf local monotone}) $(p\arr q)\leq (j(p)\arr j(q))$; see e.g.~\cite[Proposition I.5.3.1]{PiPu12}.
\medskip

As the discussion above suggests, this notion of coverage/nucleus provides a factor that causes a change in Kripke semantics \cite{GuHo19} (that is, Kripke-Joyal semantics \cite[Theorem VI.7.1]{SGL}).
The semantics of intuitionistic logic that is as important as Kripke semantics is Kleene's realizability interpretation \cite{Tro98,vO02}.
Obviously, an oracle can be a factor that causes a change in Kleene realizability, but what does this have to do with coverage/nucleus?
To investigate this relationship, let us first bring the notion of nucleus into the context of realizability.

According to the realizability interpretation, the set of truth values on the stage we are considering now is $\Omega=\mathcal{P}(\tpbf{N})$.
Indeed, the morphism tracked by $true\colon\mathbf{1}\to\mathcal{P}(\tpbf{N})$, where $true(\ast)=\tpbf{N}$, is a subobject classifier in the corresponding realizability topos (see also \cite{vOBook}).
%For operations on truth-values, we consider the following properties.

\begin{definition}[Kleene realizability \cite{Tro98,vO02}]\label{def:Kleene-realizability}
For $\Omega=\mathcal{P}(\tpbf{N})$, we introduce binary operations $\land,\lor,\arr\colon\Omega^2\to\Omega$ as follows:
\begin{align*}
p\land q&:=\{\langle{\tt a},{\tt b}\rangle:{\tt a}\in p\mbox{ and }{\tt b}\in q\},\\
p\lor q&:=\{\langle {\tt 0},{\tt a}\rangle:{\tt a}\in p\}\cup \{\langle {\tt 1},{\tt b}\rangle:{\tt b}\in q\},\\
p\arr q&:=\{{\tt x}\in\tpbf{N}:p\subseteq{\rm dom}(\varphi_{\tt x})\mbox{ and }\varphi_{\tt x}[p]\subseteq q\}.
\end{align*}
For ${\tt e}\in\tpbf{N}$ and $p\in\Omega$, we say that {\em ${\tt e}$ realizes $p$} if ${\tt e}\in p$.
Moreover, we say that {\em $p$ is realizable} if such an ${\tt e}$ can be chosen from the lightface part $\tplf{N}$ of the algebra $(\tplf{N},\tpbf{N},\ast)$.
For a collection $(p(x))_{x\in I}$, we also say that {\em $\forall x\in I.\ p(x)$ is realizable} if there exists ${\tt e}$ such that ${\tt e}$ realizes $p(x)$ for all $x\in I$.
\end{definition}

\begin{definition}[see also \cite{LvO}]
For an operation $j\colon\Omega\to\Omega$, consider the following formulas:
\begin{enumerate}
\item ({\sf local monotone}) $\forall p,q\in\Omega\;[(p\arr q)\arr(j(p)\arr j(q))]$,
\item ({\sf inflationary}) $\forall p\in\Omega\;[p\arr j(p)]$,
\item ({\sf idempotent}) $\forall p\in\Omega\;[j(j(p))\arr j(p)]$.
\end{enumerate}

If (1) is realizable, $j$ is called {\it (computably) monotone}; if (2) is realizable, $j$ is called {\it inflationary}; and if (3) is realizable, $j$ is called {\it idempotent}.
\end{definition}

\begin{remark}
Consider (1$^\prime$) $\top\arr j(\top)$ and (2$^\prime$) $\forall p,q\in\Omega\;[j(p\land q)\iffarr j(p)\land j(q)]$.
Then, one can check that, for an operation $j\colon\Omega\to\Omega$, (1), (2) and (3) are realizable if and only if (1$^\prime$), (2$^\prime$) and (3) are realizable; see \cite{LvO}.
\end{remark}

In topos-theoretic language, an operation $j$ such that (1), (2) and (3) are realizable is called a {\em Lawvere-Tierney topology} \cite{SGL} or a {\em local operator} \cite{elephant,vOBook} (on the corresponding realizability topos).
In the context of frames and locales \cite{PiPu12}, it is an {\it (internal) nucleus}.
As it is an operation on truth-values, some people think that such a notion is a kind of modal operator.
For this reason, this notion is sometimes referred to as a {\em (geometric) modality}.

%Some people prefer the terms coverage or local operator, claiming that since this concept is different from topology in the classical sense, that term is inappropriate.
Some prefer the terms coverage or local operator, arguing that the term Lawvere-Tierney topology is inappropriate because this notion has nothing to do with topology in the classical sense.
However, these terms would still be inappropriate anyway, since this notion, when considered in a realizability topos, has no more to do with cover or locality than topology.
It appears that modality is the most appropriate term.

\begin{example}[see also \cite{PiPu12}]\label{exa:LT-topologies-nucleus}
Define $\neg p=p\arr\emptyset$ for $p\in\Omega$.
It is easy to see that $\neg\neg p=\tpbf{N}$ if $p\not=\emptyset$; otherwise $\neg\neg p=\emptyset$.
Then $\neg\neg\colon\Omega\to\Omega$ is a Lawvere-Tierney topology, which is known as the {double negation topology}.

For $p,q\in\Omega$, define $j^q(p)=q\arr p$.
Then $j^q\colon\Omega\to\Omega$ is a Lawvere-Tierney topology.
Such a topology determines an open sublocale of $\Omega$ (internally).

For $p,q\in\Omega$, define $j_q(p)=p\lor q$.
Then $j_q\colon\Omega\to\Omega$ is a Lawvere-Tierney topology.
Such a topology determines a closed sublocale of $\Omega$ (internally).
\end{example}

We now explain how transparent maps can be regarded as operations on truth-values.

\begin{definition}\label{def:transform-mult-op-truth}
For a multifunction $U\pcolon\tpbf{N}\tto\tpbf{N}$, define an operation $j_U\colon\Omega\to\Omega$ as follows:
\[j_U(p)=\{x\in{\rm dom}(U):U(x)\subseteq p\}.\]
\end{definition}

Note that if $U$ is single-valued, then $j_U(p)=U^{-1}[p]$.
In the context of coverage, one sometimes interprets $x\in j(p)$ as ``$p$ is a $j$-cover of $x$'', but the idea here is to read the expression $x\in j_U(p)$ as ``To solve the problem $p$, just run the algorithm $x$ with the oracle $U$ (more precisely, just enter the code $x$ into the universal oracle computation $U$ and run it)''.

\begin{theorem}\label{prop:transparent-to-modality}
Let $U\pcolon\tpbf{N}\tto\tpbf{N}$ be computably transparent.
Then, $j_U$ is monotone.
Moreover, if $U$ is inflationary, so is $j_U$; and if $U$ is idempotent, so is $j_U$.
\end{theorem}

\begin{proof}
Let $U$ be computably transparent.
To see that $j_U$ is monotone, assume that ${\tt a}$ realizes $p\arr q$.
Given ${\tt x}\in j_U(p)$, by the definition of $j_U$, we have $U({\tt x})\subseteq p$.
Hence, ${\tt a}\ast U({\tt x})\subseteq q$.
By computable transparence, $U({\tt u}\ast{\tt a}\ast{\tt x})\subseteq {\tt a}\ast U({\tt x})\subseteq q$.
Again, by the definition of $j_U$, we have ${\tt u}\ast{\tt a}\ast{\tt x}\in j_U(q)$.
Hence, $\lambda ax.{\tt u}\ast a\ast x\in \tplf{N}$ realizes $(p\arr q)\arr (j_U(p)\arr j_U(q))$.

Let $U$ be inflationary.
In this case, ${\tt x}\in p$ implies $U(\eta\ast{\tt x})\subseteq\{{\tt x}\}\subseteq p$.
By the definition of $j_U$, we have $\eta \ast{\tt x}\in j_U(p)$.
Therefore, $\lambda x. \eta\ast x\in \tplf{N}$ realizes $p\arr U(p)$.

Let $U$ be idempotent.
If ${\tt x}\in j_U(j_U(p))$ then $U({\tt x})\subseteq j_U(p)$ by the definition of $j_U$.
Thus, for any ${\tt y}\in U({\tt x})$, we have ${\tt y}\in j_U(p)$, which implies that ${\tt y}\in{\rm dom}(U)$ and $U({\tt y})\subseteq p$ by definition.
By the former property, we have $U({\tt x})\subseteq{\rm dom}(U)$, so ${\tt x}\in{\rm dom}(U\circ U)$.
Hence, by the latter property, we have $U\circ U({\tt x})=\bigcup\{U({\tt y}):{\tt y}\in U({\tt x})\}\subseteq p$.
As $U$ is idempotent, we get $U(\mu\ast{\tt x})\subseteq U\circ U({\tt x})\subseteq p$.
This implies $\mu\ast{\tt x}\in j_U(p)$ by definition.
Hence, $\mu\in \tplf{N}$ realizes $j_U(j_U(p))\arr j_U(p)$.
\end{proof}

A natural question to ask is whether a converse of Theorem \ref{prop:transparent-to-modality} holds in some sense.
However, one can observe that there is a rather strong restriction on operations on $\Omega$ obtained from transparent maps, as follows.

\begin{prop}\label{prop:oracle-to-modality-preservation-basic}
For any partial multimap $U\pcolon \tpbf{N}\tto \tpbf{N}$, the operation $j_U\colon\Omega\to\Omega$ preserves arbitrary intersection:
\[
j_U\left(\bigcap_{i\in I}p_i\right)=\bigcap_{i\in I}j_U(p_i).
\]

Moreover, if $U$ is single-valued, $j_U$ preserves arbitrary union:
\[
j_U\left(\bigcup_{i\in I}p_i\right)=\bigcup_{i\in I}j_U(p_i).
\]
\end{prop}

\begin{proof}
By the definition of $j_U$, we have that $x\in j_U\left(\bigcap_{i\in I}p_i\right)$ if and only if $U(x)\subseteq \bigcap_{i\in I}p_i$.
Clearly, the latter condition is equivalent to that $U(x)\subseteq p_i$ for any $i\in I$.
By the definition of $j_U$ again, this is equivalent to that $x\in\bigcap_{i\in I}j_U(p_i)$.
If $U$ is single-valued, as $j_U(p)=U^{-1}[p]$, it is clear that $j_U$ preserves arbitrary union.
\end{proof}

\begin{remark}
This is actually a strong restriction.
There are many operations on $\Omega$ that do not preserve arbitrary intersection in any sense, for example, consider the double negation topology $\neg\neg\colon\Omega\to\Omega$.
There are also natural operations on $\Omega$ that preserve arbitrary intersection but not union.
For example, one can easily see that the topology $j^q\colon\Omega\to\Omega$ for an open sublocale introduced in Example \ref{exa:LT-topologies-nucleus} is such an operation.
\end{remark}

\begin{remark}
Note that if $j\colon\Omega\to\Omega$ preserves arbitrary union, then $j$ preserves the bottom element, i.e., $j(\emptyset)=\emptyset$.
Indeed, $U$ is $\neg\neg$-dense (see \cite{Bau21} and Section \ref{sec:universal-closure-operator-intro} below) if and only if $U(p)\not=\emptyset$ for any $p\in{\rm dom}(U)$ if and only if $j_U(\emptyset)=\emptyset$.
Note also that $j_U(\tpbf{N})={\rm dom}(U)$ for a partial multimap $U$.
Thus, if $U$ is total, then $j_U$ preserves the top element, i.e., $j_U(\tpbf{N})=\tpbf{N}$.
\end{remark}

Now, let us consider the possibility of a certain backwards assertion of Theorem \ref{prop:transparent-to-modality}.
In other words, let us consider how to obtain a partial multimap on $\tpbf{N}$ from an operation on $\Omega$.

\begin{definition}\label{def:operation-to-pamulti}
For an operation $j\colon\Omega\to\Omega$, define a partial multifunction $U_j\pcolon\tpbf{N}\tto\tpbf{N}$ as follows:
\begin{align*}
{\rm dom}(U_j)=\bigcup\{j(p):p\in\Omega\},& &
U_j(x)=\bigcap\{p\in\Omega:x\in j(p)\}.
\end{align*}
\end{definition}

\begin{example}\label{example:open-sublocales-Medvedev}
For $q\in\Omega$, let $j^q\colon\Omega\to\Omega$ be the topology for an open sublocale introduced in Example \ref{exa:LT-topologies-nucleus}.
Note that ${\rm dom}(U_{j^q})=\bigcup\{q\arr p:p\in\Omega\}=q\arr\tpbf{N}$, which is the set of all $e\in\tpbf{N}$ such that $q\subseteq{\rm dom}(\varphi_e)$.
Moreover, $U_{j^q}(x)=\bigcap\{p:x\in q\arr p\}=x\ast q=\varphi_x[q]$.
Observe that this is exactly the universal computation ${\tt Med}_q$ relative to the Medvedev oracle $q$ introduced in Example \ref{exa:oracle-Medvedev-reducibility}, i.e., $U_{j^q}={\tt Med}_q$.
\end{example}

The transformation $j\mapsto U_j$ is the inverse of the transformation $U\mapsto j_U$ in the following sense:

\begin{theorem}\label{thm:modality-to-oracle-various-preserve}
Let $j\colon\Omega\to\Omega$ be an operation preserving arbitrary intersection.
Then, $j_{U_j}=j$ and $U_{j_U}=U$.
If $j$ is monotone, $U_j$ is computably transparent; if $j$ is inflationary, so is $U_j$; and if $j$ is monotone and idempotent, so is $U_j$.
Moreover, if $j$ preserves arbitrary union, then $U_j$ is single-valued.
\end{theorem}

\begin{proof}
As $j$ preserves arbitrary intersection, we have, for any $x\in{\rm dom}(U_j)$,
\[x\in\bigcap\{j(p):x\in j(p)\}=j\left(\bigcap\{p:x\in j(p)\}\right)=j(U_j(x)).\]

Combining the above equality and the definition of $U_j$, it immediately follows that $U_j(x)$ is the least $p$ such that $x\in j(p)$.
In other words,
\begin{align}\label{equ:from-modality-to-oracle-modest-case}
x\in j(p)\iff U_j(x)\subseteq p
\end{align}

Therefore, we get $j_{U_j}(p)=\{x\in{\rm dom}(U_j):U_j(x)\subseteq p\}=j(p)\cap{\rm dom}(U_j)=j(p)$.
Here, the last equality follows from $j(p)\subseteq{\rm dom}(U_j)$.
Hence, we obtain $j_{U_j}=j$.

Next, let $U\colon\tpbf{N}\tto\tpbf{N}$ be given.
Then, we have ${\rm dom}(U_{j_U})=\bigcup_pj_U(p)={\rm dom}(U)$ by the definitions of $U_{j_U}$ and $j_U$, and the observation that $x\in{\rm dom}(U)$ implies $x\in j_U(U(x))$.
Since $j_U$ preserves arbitrary intersection by Proposition \ref{prop:oracle-to-modality-preservation-basic}, we again have the equivalence (\ref{equ:from-modality-to-oracle-modest-case}) for $j_U$, so $U_{j_U}(x)\subseteq p$ if and only if $x\in j_U(p)$, which is equivalent to $U(x)\subseteq p$ by definition.
This implies $U_{j_U}=U$.

Assume that $j$ is monotone, realized by ${\tt u}\in \tplf{N}$.
To show that $U_j$ is computably transparent, let ${\tt f},{\tt x}\in\tpbf{N}$ be given.
First note that ${\tt f}$ clearly realizes $U_j({\tt x})\arr {\tt f}\ast U_j({\tt x})$.
By monotonicity, ${\tt u}\ast{\tt f}$ realizes $j(U_j({\tt x}))\arr j({\tt f}\ast U_j({\tt x}))$.
By the definition of $U_j$, we have ${\tt x}\in j(U_j({\tt x}))$, so we get ${\tt u}\ast{\tt f}\ast{\tt x}\in j({\tt f}\ast U_j({\tt x}))$.
By the equivalence (\ref{equ:from-modality-to-oracle-modest-case}) again, this is equivalent to that $U_j({\tt u}\ast{\tt f}\ast{\tt x})\subseteq {\tt f}\ast U_j({\tt x})$, which means that $U_j$ is computably transparent.

Assume that $j$ is inflationary, realized by $\eta\in \tplf{N}$.
In particular, $\eta$ realizes $\{{\tt x}\}\arr j(\{{\tt x}\})$, so we have $\eta\ast{\tt x}\in j(\{\tt x\})$.
By the equivalence (\ref{equ:from-modality-to-oracle-modest-case}), this is equivalent to that $U_j(\eta\ast{\tt x})\subseteq\{\tt x\}$, which means that $U_j$ is inflationary.

Assume that $j$ is idempotent, realized by $\mu\in \tplf{N}$.
In particular, $\mu$ realizes $j\circ j(U_j\circ U_j({\tt x}))\arr j(U_j\circ U_j({\tt x}))$.
Note that from the definition of composition of multifunctions, if ${\tt x}\in{\rm dom}(U_j\circ U_j)$ then $U_j({\tt x})\subseteq{\rm dom}(U_j)$; that is, ${\tt z}\in U_j({\tt x})$ implies ${\tt z}\in{\rm dom}(U_j)$.
Then, for such ${\tt x}$ and ${\tt z}$ we have $U_j({\tt z})\subseteq U_j\circ U_j({\tt x})$ by the definition of composition.
This implies ${\tt z}\in j(U_j\circ U_j({\tt x}))$ by the equivalence (\ref{equ:from-modality-to-oracle-modest-case}).
Hence, $U_j({\tt x})\subseteq j(U_j\circ U_j({\tt x}))$; therefore, a code ${\tt i}$ of the identity function realizes $U_j({\tt x})\arr j(U_j\circ U_j({\tt x}))$.
Then, by monotonicity, ${\tt u}\ast{\tt i}$ realizes $j(U_j({\tt x}))\arr j\circ j(U_j\circ U_j({\tt x}))$.
By the equivalence (\ref{equ:from-modality-to-oracle-modest-case}), we have ${\tt x}\in j(U_j({\tt x}))$; therefore ${\tt u}\ast{\tt i}\ast{\tt x}\in j\circ j(U_j\circ U_j({\tt x}))$.
Hence we obtain $\mu\ast({\tt u}\ast{\tt i}\ast{\tt x})\in j(U_j\circ U_j({\tt x}))$.
By the equivalence (\ref{equ:from-modality-to-oracle-modest-case}) again, we get $U_j(\mu\ast({\tt u}\ast{\tt i}\ast{\tt x}))\subseteq U_j\circ U_j({\tt x})$.
Thus, $\lambda x.\mu\ast({\tt u}\ast{\tt i}\ast x)$ witnesses that $U_j$ is idempotent.

Finally, assume that $j$ preserves arbitrary union.
Given $x\in\tpbf{N}$, if there exists $p\in\Omega$ such that $x\in j(p)$, then let us consider an enumeration $\{y_i\}_{i\in I}$ of all elements of $p$.
As $j$ preserves arbitrary union, we have $x\in j(p)=j(\bigcup_{i\in I}\{y_i\})=\bigcup_{i\in I} j(\{y_i\})$.
Therefore, $x\in j(\{y_i\})$ for some $i\in I$.
As noted above, preservation of arbitrary union also implies $j(\emptyset)=\emptyset$, so we have $U_j(x)=\{y_i\}$ by the equivalence (\ref{equ:from-modality-to-oracle-modest-case}).
This means that $U_j$ is single-valued.
\end{proof}

By Theorems \ref{prop:transparent-to-modality} and \ref{thm:modality-to-oracle-various-preserve} and Proposition \ref{prop:oracle-to-modality-preservation-basic}, we obtain the correspondences between partial multifunctions on $\tpbf{N}$ and $\bigcap$-preserving operations on $\Omega$; and between partial (single-valued) functions on $\tpbf{N}$ and $\bigcup,\bigcap$-preserving operations on $\Omega$.
Moreover, computable transparency corresponds to monotonicity.
Being inflationary and idempotent are also preserved by this correspondence; see Table \ref{table:correspondence-truth}.
Similarly, by Example \ref{example:open-sublocales-Medvedev}, Medvedev oracles correspond to open topologies.
Here, Lawvere-Tierney topologies for which internal nuclei yield open sublocales of $\Omega$ are called open topologies; see e.g.~\cite[Section A.4.5]{elephant}.

\begin{table}[t]
\begin{center}
  \begin{tabular}{ccc}
  {multifunction} & & {operation on truth-values} \\
    \hline
    {computably transparent} &$\iff$& {monotone} \\
    {\footnotesize $(\forall f)(\exists F)\;U\circ F\preceq f\circ U$} & & {\footnotesize $(p\to q)\to(j(p)\to j(q))$}\\[0.4em]
    {inflationary} &$\iff$& {inflationary} \\
    {\footnotesize ${\rm id}\leq_mU$} & & {\footnotesize $p\to j(p)$}  \\[0.4em]
    {idempotent} &$\iff$& {idempotent} \\
    {\footnotesize $U\circ U\leq_mU$} & & {\footnotesize $j(j(p))\to j(p)$} \\
    \hline \\
  \end{tabular}
  \caption{The correspondence between oracles (multifunctions) and $\bigcap$-preserving operations on truth-values.}\label{table:correspondence-truth}
 \end{center}
 \end{table}

\subsection{Reducibility}\label{sec:truth-value-reducibility}
Let us formulate more rigorously the meaning of the correspondences in the previous section.

\begin{definition}\label{def:m-morphi-r-morphi}
Given partial multifunctions $f,g\pcolon\tpbf{N}\tto\tpbf{N}$, we say that a computable element $e\in \tplf{N}$ is an {\em $m$-morphism from $f$ to $g$} if, for any $x\in{\rm dom}(f)$, $\varphi_e(x)$ is defined and $g(\varphi_e(x))\subseteq f(x)$.
Given operations $j,k\colon\Omega\to\Omega$, we say that a computable element $e\in \tplf{N}$ is an {\em $r$-morphism from $j$ to $k$} if $e$ realizes $\forall p\in\Omega.j(p)\arr k(p)$.
\end{definition}

Note that $e$ is an $m$-morphism from $f$ to $g$ if and only if $g\circ\varphi_e$ refines $f$.
Using the terminology of Definition \ref{def:pmultimap-transparent}, for $e\in \tplf{N}$ to be an $m$-morphism from $f$ to $g$ means that $f$ is $g$-computable via $\varphi_e$, or equivalently, $\varphi_e$ is a many-one reduction witnessing $f\leq_mg$.

\begin{prop}\label{prop:two-category-isomorphic}
For any partial multifunctions $f,g\pcolon\tpbf{N}\tto\tpbf{N}$, an $m$-morphism $e\colon f\to g$ can be thought of as an $r$-morphism $e\colon j_f\to j_g$.
Conversely, for any $\bigcap$-preserving operations $j,k\colon\Omega\to\Omega$, an $r$-morphism $e\colon j\to k$ can be thought of as an $m$-morphism $e\colon U_j\to U_k$.
\end{prop}

\begin{proof}
Assume that $e$ is an $m$-morphism from $f$ to $g$.
Then $g(\varphi_e(x))\subseteq f(x)$ for any $x\in{\rm dom}(f)$.
Therefore, by definition, we have $x\in j_f(p)$ iff $f(x)\subseteq p$, which implies $g(\varphi_e(x))\subseteq p$, iff $\varphi_e(x)\in j_g(p)$
for any $p\in\Omega$.
This means that $e$ realizes $j_f(p)\arr j_g(p)$, so $e$ is an $r$-morphism from $j_f$ to $j_g$.

Conversely, let $e$ be an $r$-morphism from $j$ to $k$.
Then, for any $p\in\Omega$ and ${\tt x}\in\tpbf{N}$, ${\tt x}\in j(p)$ implies $e\ast {\tt x}\in k(p)$.
By the equivalence (\ref{equ:from-modality-to-oracle-modest-case}) in the proof of Theorem \ref{thm:modality-to-oracle-various-preserve}, this shows that $U_j({\tt x})\subseteq p$ implies $U_k(e\ast{\tt x})\subseteq p$.
In particular, by setting $p=U_j({\tt x})$, we get $U_k(e\ast {\tt x})\subseteq U_j({\tt x})$.
This means that $e$ is an $m$-morphism from $U_j$ to $U_k$.
\end{proof}

Let us organize our results in Section \ref{sec:realizability} by introducing a little terminology.
Since an $m$-morphism is essentially a many-one reduction, let us write ${\bf MRed}$ for the category of partial multifunctions on $\tpbf{N}$ and $m$-morphisms.
We also write ${\bf Op}_{\tt X}(\Omega)$ for the category of ${\tt X}$-preserving operations on $\Omega$ and $r$-morphisms.
By Proposition \ref{prop:two-category-isomorphic}, we get the following isomorphism (via $U\mapsto j_U$ and $j\mapsto U_j$):
\begin{align*}
{\bf MRed}\simeq{\bf Op}_{\bigcap}(\Omega),
\end{align*}

Similarly, we have isomorphisms between corresponding full subcategories.
The full subcategory of ${\bf MRed}$ consisting of partial single-valued functions is denoted by ${\bf Red}$.
Then Proposition \ref{prop:oracle-to-modality-preservation-basic} and Theorem \ref{thm:modality-to-oracle-various-preserve} ensure ${\bf Red}\simeq{\bf Op}_{\bigcup,\bigcap}(\Omega)$.

For the full subcategory of ${\bf MRed}$, restricted to those that are computably transparent (inflationary, and idempotent, respectively), we add the superscript ${\rm ct}$ ($\eta$ and $\mu$, respectively).
Similarly, for the full subcategory of ${\bf Op}_{\bigcap}(\Omega)$, restricted to those that are monotone (inflationary, and idempotent, respectively), we add the superscript ${\rm mon}$ ($\eta$ and $\mu$, respectively).
Using this notation, Theorems \ref{prop:transparent-to-modality} and \ref{thm:modality-to-oracle-various-preserve} give the following isomorphisms, for example:
\begin{align*}
{\bf (M)Red}^{\rm ct}\simeq{\bf Op}^{\rm mon}_{(\bigcup),\bigcap}(\Omega),
& &
{\bf (M)Red}^{\rm ct,\eta,\mu}\simeq{\bf Op}^{{\rm mon},\eta,\mu}_{(\bigcup),\bigcap}(\Omega),
\end{align*}

One of the most important parts of computability theory is to discuss the complexity of various objects.
Therefore, we shall explain these results in the context of reducibility.
We consider the following reducibility notion:

\begin{definition}
%For partial multifunctions $f,g\pcolon\tpbf{N}\tto\tpbf{N}$ we say that {\em $f$ is many-one reducible to $g$} (written $f\leq_mg$) if an $m$-morphism from $f$ to $g$ exists.
For operations $j,k\colon\Omega\to\Omega$ we say that {\em $j$ is $r$-reducible to $k$} (written $j\leq_rk$) if an $r$-morphism from $j$ to $k$ exists; that is, $\forall p\in\Omega.j(p)\arr k(p)$ is realizable.
\end{definition}

%In this context, an $m$-morphism is also called a many-one reduction.
%To use the terminology of Definition \ref{def:pmultimap-transparent}, $f\leq_mg$ means that $f$ is $g$-computable.
There also have been previous studies on $r$-reducibility for operations on $\Omega$; see e.g.~\cite{LvO,Kih21}.
For basic results on the $r$-ordering of Lawvere-Tierney topologies in a more general topos-theoretic setting, see also Johnstone \cite[Section A.4.5]{elephant}.
Each preorder obtained from a reducibility notion is often referred to as the degrees.
%so yield equivalence relations $\equiv_m$ and $\equiv_r$.
%Each equivalence class is often called a {\em degree}.
As a corollary of Theorem \ref{thm:modality-to-oracle-various-preserve} and Proposition \ref{prop:two-category-isomorphic}, we get the following:

\begin{cor}\label{cor:correspondence-m-deg-r-deg}
The $m$-degrees of partial multifunctions on $\tpbf{N}$ and the $r$-degrees of $\bigcap$-preserving operations on $\Omega$ are isomorphic.
\end{cor}

Combining the results of this section and Section \ref{sec:left-adjoint}, various degree notions can be characterized using operations on truth values.
As we have seen that $\bigcap$-preserving monotone operations on $\Omega$ correspond to computably transparent multimaps (Theorems \ref{prop:transparent-to-modality} and \ref{thm:modality-to-oracle-various-preserve}), which correspond to Weihrauch degrees (Corollary \ref{prop:computably-transparent-many-one-Weih}), we get the following:

%By combining Theorems \ref{thm:modality-to-oracle-various-preserve} and \ref{prop:transparent-to-modality} and Propositions \ref{prop:two-category-isomorphic} and \ref{prop:computably-transparent-many-one-Weih}, we get the following:

\begin{cor}
The Weihrauch lattice of partial multifunctions on $\tpbf{N}$ is isomorphic to the $r$-degrees of $\bigcap$-preserving monotone operations on $\Omega$.
\qed
\end{cor}

Similarly, we have seen that $\bigcap$-preserving monotone, inflationary, operations on $\Omega$ correspond to computably transparent, inflationary, multimaps (Theorems \ref{prop:transparent-to-modality} and \ref{thm:modality-to-oracle-various-preserve}), which correspond to pointed Weihrauch degrees (Corollary \ref{prop:computably-transparent-many-one-pWeih}), so we get the following:

\begin{cor}
The pointed Weihrauch lattice of partial multifunctions on $\tpbf{N}$ is isomorphic to the $r$-degrees of $\bigcap$-preserving, monotone, inflationary, operations on $\Omega$.
\qed
\end{cor}

We have also seen that $\bigcap$-preserving Lawvere-Tierney topologies on $\Omega$ correspond to computably transparent, inflationary, idempotent, multimaps (Theorems \ref{prop:transparent-to-modality} and \ref{thm:modality-to-oracle-various-preserve}), which correspond to Turing-Weihrauch degrees (Corollary \ref{prop:computably-transparent-many-one-TWeih}), so we get the following:

\begin{cor}\label{cor:generalized-Weihrauch}
The Turing-Weihrauch lattice of partial multifunctions on $\tpbf{N}$ is isomorphic to the $r$-degrees of $\bigcap$-preserving Lawvere-Tierney topologies.
\qed
\end{cor}

As seen in Theorems \ref{prop:transparent-to-modality} and \ref{thm:modality-to-oracle-various-preserve}, being single-valued corresponds to the $\bigcup$-preservation property, so this also gives the following characterization of the subTuring degrees (in the sense of Sasso-Madore \cite{Sas75,Mad12} as in Corollary \ref{cor:Turing-degrees-mna}) since partial Turing oracles are exactly single-valued Turing-Weihrauch oracles in the Kleene first algebra.

\begin{cor}\label{cor:Turing-degrees-mna2}
The subTuring degrees of partial functions is isomorphic to the $r$-degrees of $\bigcup,\bigcap$-preserving Lawvere-Tierney topologies (on the Kleene first algebra).
\qed
\end{cor}

\begin{remark}
It was Hyland \cite{Hey82} who first noticed that each Turing degree yields a Lawvere-Tierney topology (so a subtopos) on the effective topos.
Later, Faber-van Oosten \cite{FavO14} showed that the subTuring degrees correspond to the realizability subtoposes (i.e., the subtoposes that are realizability toposes) of the effective topos.
They also gave another characterization of a partial Turing oracle (i.e., a subTuring degree) as a Lawvere-Tierney topology preserving arbitrary union and disjointness.
\end{remark}

%Note again about Corollary \ref{cor:Turing-degrees-mna2}, as mentioned in Section \ref{sec:background}, Turing reducibility in our sense is Turing sub-reducibility in the sense of Sasso \cite{Sas75}.
Unfortunately, Sasso-Madore's Turing sub-reducibility has not become mainstream as reducibility between partial functions (see e.g.~Cooper \cite[Section 11]{Cooper}) and has not been studied in depth.
However, Corollary \ref{cor:Turing-degrees-mna2} and Faber-van Oosten's result \cite{FavO14} mentioned above suggest that Sasso-Madore's Turing sub-reducibility is more natural, at least from a topos-theoretic point of view, than other reducibilities between partial functions.

Therefore, we present the analysis of the subTuring degrees as an important problem.

\begin{question}[see also Madore \cite{Mad12}]
Analyze the structure of the subTuring degrees of partial functions.
\end{question}

This problem has recently been addressed by Kihara-Ng \cite{KN24}.

\begin{question}\label{question-ot}
Is it possible to characterize ordinary Turing reducibility (i.e., Turing reducibility for total functions) as a property of operations on truth values?
\end{question}

By Example \ref{example:open-sublocales-Medvedev} we also see that Medvedev degrees correspond to open Lawvere-Tierney topologies (i.e., topologies for which internal nuclei yield open sublocales of $\Omega$; see e.g.~Johnstone \cite[Section A.4.5]{elephant}).

\begin{cor}\label{cor:Medvedev-deg}
The Medvedev lattice is isomorphic to the $r$-degrees of open Lawvere-Tierney topologies (on the Kleene-Vesley algebra).
\qed
\end{cor}

\begin{table*}[t]\small
  \caption{Correspondences between operations on $\Omega$ and degree-theoretic notions}
  \label{table}
  \centering
  \begin{tabular}{c|ccc}
  	& no condition & monotone &  Lawvere-Tierney topology \\ \hline
      $\bigcup,\bigcap$-preserving & many-one & single-Weihrauch & Turing \\
     $\bigcap$-preserving & multi-many-one & Weihrauch & Turing-Weihrauch \\
     no condition & extended many-one & extended Weihrauch & extended Turing-Weihrauch \\
  \end{tabular}
\end{table*}
We have now characterized various notions of oracles as operations on the truth values (summarized as Table \ref{table}).
A similar characterization of operations that do not preserve $\bigcap$ is given in Section \ref{sec:mmmap}.

Note that the preservation property of union and intersection cannot be written in the language of realizability.
Indeed, these preservation properties are not closed under the $r$-equivalence $\equiv_r$, so let us introduce a slightly modified notion.
For $j\colon\Omega\to\Omega$, we say that $j$ is {\em realizably $\bigcup$-preserving} if, for any collection $(p_i)_{i\in I}$, the formula $j(\bigcup_{i\in I}p_i)\iffarr\bigcup_{i\in I}j(p_i)$ is realizable.
Similarly, $j$ is {\em realizably $\bigcap$-preserving} if, for any collection $(p_i)_{i\in I}$, the formula $j(\bigcap_{i\in I}p_i)\iffarr\bigcap_{i\in I}j(p_i)$ is realizable.

\begin{prop}
Let $j\colon\Omega\to\Omega$ be a monotone operation.
If $j$ is realizably $\bigcap$-preserving, then there exists a monotone $\bigcap$-preserving operation $k\colon\Omega\to\Omega$ such that $k\equiv_rj$.
If $j$ is realizably $\bigcup,\bigcap$-preserving, then there exists a monotone $\bigcup,\bigcap$-preserving operation $k\colon\Omega\to\Omega$ such that $k\equiv_rj$.
\end{prop}

\begin{proof}
Assume that $j$ is realizably $\bigcap$-preserving.
Then, by Proposition \ref{prop:oracle-to-modality-preservation-basic}, $j_{U_j}$ preserves arbitrary intersection, so it suffices to show that $j\equiv_rj_{U_j}$.
If $x\in j(p)$ then by the definition of $U_j$ we have $U_j(x)=\bigcap\{q:x\in j(q)\}\subseteq p$, so $x\in j_{U_j}(p)$ by the definition of $j_{U_j}$.
Hence, the identity map witnesses $j\leq_r j_{U_j}$.
Conversely, first note that we clearly have $x\in\bigcap\{j(p):x\in j(p)\}$ for any $x$.
As $j$ is realizably $\bigcap$-preserving, some ${\tt a}\in \tplf{N}$ realizes $\bigcap_{x\in j(p)}j(p)\arr j(\bigcap_{x\in j(p)}p)$, so ${\tt a}\ast x\in j(\bigcap_{x\in j(p)}p)=j(U_j(x))$, where the last equality follows from the definition of $U_j$.
Now, $x\in j_{U_j}(p)$ implies $U_j(x)\subseteq p$ by the definition of $j_{U_j}$, so the identity map realizes $U_j(x)\arr p$.
As $j$ is monotone, some ${\tt b}\in\tplf{N}$ realizes $j(U_j(x))\arr j(p)$.
Hence, we get ${\tt b}\ast({\tt a}\ast x)\in j(p)$; that is, $\lambda x.{\tt b}\ast({\tt a}\ast x)\in \tplf{N}$ realizes $j_{U_j}(p)\arr j(p)$, which shows $j\equiv_r j_{U_j}$.

We next assume that $j$ is realizably $\bigcup,\bigcap$-preserving.
By the previous argument, we may assume that $j$ preserves arbitrary intersection (by replacing $j$ with $j_{U_j}$ if necessary).
We first claim that $j$ preserves disjointness, that is, $p\cap q=\emptyset$ implies $j(p)\cap j(q)=\emptyset$.
To see this, as $j$ is realizably $\bigcup$-preserving, in particular, $j(\emptyset)\iffarr\emptyset$ is realizable, but this property implies $j(\emptyset)=\emptyset$.
If $p\cap q=\emptyset$ then, as $j$ preserves intersection, we have $j(p)\cap j(q)=j(p\cap q)=j(\emptyset)=\emptyset$.
This verifies the claim.
Now, by our assumption, we have an element ${\tt a}\in \tplf{N}$ realizing $j(\bigcup_ip_i)\arr\bigcup_ij(p_i)$ whatever $(p_i)_{i\in I}$ is.
In particular, $x\in j(p)=j(\bigcup_{y\in p}\{y\})$ implies ${\tt a}\ast x\in\bigcup_{y\in p}j(\{y\})$.
As $j$ preserves disjointness, one can see that there exists a unique $y_x\in p$ such that ${\tt a}\ast x\in j(\{y_x\})$.
Then, we define $U_j'(x)=\{y_x\}$ for such $y_x$.
As $U_j'$ is single-valued, by Proposition \ref{prop:oracle-to-modality-preservation-basic}, $j_{U_j'}$ preserves arbitrary union and intersection.

It remains to show that $j\equiv_r j_{U_j'}$.
Given $x\in j(p)$, for the unique $y_x\in p$ as above, we have $U_j'(x)=\{y_x\}\subseteq p$, so we have $x\in j_{U_j'}(p)$ by definition.
Hence, the identity map witnesses $j\leq_rj_{U_j'}$
Conversely, if $x\in j_{U_j'}(p)$, by definition, we have $U_j'(x)\subseteq p$; that is the identity map realizes $U_j'(x)\arr p$.
Thus, by monotonicity of $j$, some ${\tt b}\in N$ realizes $j(U_j'(x))\arr j(p)$.
By our choice of $y_x$, we have ${\tt a}\ast x\in j(\{y_x\})=j(U_j'(x))$, so we get ${\tt b}\ast({\tt a}\ast x)\in j(p)$.
Therefore, $\lambda x.{\tt b}\ast({\tt a}\ast x)\in \tplf{N}$ realizes $j_{U'_j}(p)\arr j(p)$; hence we obtain $j\equiv_r j_{U'_j}$.
\end{proof}

In particular, the $r$-degrees of monotone realizably $\bigcap$-preserving operations are equal to those of monotone $\bigcap$-preserving operations.
The same is true for the $\bigcup,\bigcap$ preservation property.

%%%%%%%%%%%%%%%%%%%%%%%%%%%%%%%%%%%%%%%%%%%%%%%%%%%%%%%%%%%%
%%%%%%%%%%%%%%%%%%%%%%%%%%%%%%%%%%%%%%%%%%%%%%%%%%%%%%%%%%%%
%%%%%%%%%%%%%%%%%%%%%%%%%%%%%%%%%%%%%%%%%%%%%%%%%%%%%%%%%%%%
%%%%%%%%%%%%%%%%%%%%%%%%%%%%%%%%%%%%%%%%%%%%%%%%%%%%%%%%%%%%
%%%%%%%%%%%%%%%%%%%%%%%%%%%%%%%%%%%%%%%%%%%%%%%%%%%%%%%%%%%%
%%%%%%%%%%%%%%%%%%%%%%%%%%%%%%%%%%%%%%%%%%%%%%%%%%%%%%%%%%%%
%%%%%%%%%%%%%%%%%%%%%%%%%%%%%%%%%%%%%%%%%%%%%%%%%%%%%%%%%%%%

\section{Oracles as factors that change representations of spaces}\label{sec:rep-sp-main}

The goal of this section is to provide an explanation of the notion of ``jump of representation (change of coding) \cite{Zie07,dB14}'' in the context of ``universal closure operator \cite{elephant}'', which is known to be closely tied with Lawvere-Tierney topology.

\subsection{Represented space}\label{sec:represented-space}

A triple $(\tplf{N},\tpbf{N},\ast)$, dealt with in Section \ref{sec:realizability-PCA} (see also Definition \ref{def:partial-combinatory-algebra-int}), can be used to encode various mathematical objects.
For example, the case $\tplf{N}=\tpbf{N}=\N$ corresponds to encoding various mathematical notions by using natural numbers (or finite sequences of symbols), and it is not difficult to imagine that this idea is used in various fields, including the theory of computation.
Ershov's theory of numbering \cite{Er99}, for example, is well known as an abstract theory of $\N$-coding.
The case $\tpbf{N}=\N^\N$ corresponds to the idea of encoding abstract mathematical objects by streams (possibly infinite sequences of symbols), which has also been used in various fields, including set theory, reverse mathematics and  computable analysis.
In modern computable analysis, this idea is formulated as follows:

\begin{definition}[see e.g.~\cite{SchCCA}]
A {\it represented space} is a set $X$ equipped with a partial surjection $\delta_X\pcolon\tpbf{N}\to X$.
If $\delta_X(p)=x$ then $p$ is called a {\it name} or a {\it code} of $x$.
A map $f\colon X\to Y$ is {\it continuous} if
one can transform a name of $x\in X$ into a name of $f(x)\in Y$ in a continuous manner; that is, there exists a continuous function $F\pcolon\tpbf{N}\to\tpbf{N}$ such that $\delta_Y\circ F=F\circ \delta_X$.
If such an $F$ is computable, we say that $f$ is {\it computable}.
\[
\xymatrix{
X\ar[rr]^f & & Y\\
\tpbfxy{N}\ar[rr]_{F} \ar[u]^{\delta_X} & & \tpbfxy{N}\ar[u]_{\delta_Y}
}
\]
\end{definition}

\begin{example}[Coding topological spaces]\label{exa:coding-topological-space}
As an example of coding by stream, a real $x$ can be coded by a (rapidly converging) Cauchy sequence of rationals; that is, a stream $\alpha$ codes $x$ if and only if $|x-q_{\alpha(n)}|<2^{-n}$, where $q_i$ is the $i$th rational.
In general, every point in a Polish space can be coded by a (rapidly converging) Cauchy sequence in a countable dense set $\{q_i\}_{i\in\N}$; see e.g.~\cite{MosBook,Handbook-CCA}.
Similarly, every point in a second-countable $T_0$ space can be coded by (an enumeration of) its neighborhood filter restricted to a fixed countable basis $\{B_i\}_{i\in\N}$; that is, a stream $\alpha$ codes $x$ if and only if $\{\alpha(n):n\in\N\}=\{i\in\N:x\in B_i\}$.
By using some variant (e.g., cs-network and k-network) of (Arhangel'skii's notion of) network instead of a basis, one can represent an even wider class of topological spaces in such a way that computability and (sequential) topology are compatible; see Schr\"oder \cite{Sch02}.
\end{example}

\begin{example}[Coding function spaces]\label{exa:coding-continuous functions}
It is known that the category of represented spaces and continuous maps is cartesian closed.
In detail, given represented spaces $X$ and $Y$, a code for a continuous map $f\colon X\to Y$ is given by an index of a partial continuous function $F$ on $\tpbf{N}$ that tracks $f$.
This yields a representation of the function space $C(X,Y)$.

One important special case of this construction is a coding of a hyperspace.
For a topological space $X$, recall that $U\subseteq X$ is open if and only if the characteristic function $\chi_U\colon X\to\mathbb{S}$ is continuous, where $\mathbb{S}$ is Sierpi\'nski's connected two point space.
One can then introduce the hyperspace $\mathcal{O}(X)$ of open sets in $X$ as the function space $C(X,\mathbb{S})$.
For the details, see \cite{Handbook-CCA}.
\end{example}

\begin{example}[Coding Borel sets]
One of the most famous representated spaces in set theory is the hyperspace of Borel sets, whose representation is given by the notion of Borel codes, which is first introduced by Solovay; see e.g.~\cite{MosBook}.
% \cite[Section 1]{Sol69} and \cite[Section II.1]{Sol70} to introduce the notion of random forcing.
% in his monumental work on a model of Zermelo-Fraenkel (ZF) set theory in which all sets of reals are Lebesgue measurable.
%This notion also yields (multi-)representations of Borel-generated $\sigma$-ideals such as Lebesgue null sets and meager sets (see below), which leads, for example, to the notion of idealized forcing.
%Representations of such $\sigma$-ideals are evidently useful when dealing with randomness, genericity, forcing, etc.
%This way of thinking has become ubiquitous in set theory; see also \cite{BaJu95,Zap08}.
\end{example}

In the following, a multi-surjection refers to a multimap $\delta\pcolon X\tto Y$ such that for any $y\in Y$ there exists $x$ such that $y\in \delta(x)$.

\begin{definition}[see e.g.~\cite{Sch02,SchCCA}]\label{def:multi-represented-space}
A {\it multi-represented space} is a set $X$ equipped with a partial multi-surjection $\delta_X\pcolon\tpbf{N}\tto\tpbf{N}$.
If $x\in\delta_X(p)$ then $p$ is called a {\it name} or a {\it code} of $x$.
A map $f\colon X\to Y$ is {\it continuous} if there exists a continuous function $F\pcolon\tpbf{N}\to\tpbf{N}$ which, given a $\delta_X$-name of $x\in X$, returns a $\delta_Y$-name of $f(x)\in Y$.
%that is, there exists a continuous function $F\pcolon\N^\N\to\N^\N$ such that $\delta_Y\circ F=F\circ \delta_X$.
If such an $F$ is computable, we say that $f$ is {\it computable}.
\end{definition}

Also, if a transparent map $U$ is given, one may define the notion of $U$-continuity ($U$-computability) of a map $f\colon X\to Y$ by replacing the condition of $F$ with that of being $U$-continuous ($U$-computable).

\begin{example}
A name of a null set $A$ in a Polish space may be given by a certain $G_\delta$-witness of its nullness, i.e., a $G_\delta$-code of a null $G_\delta$ set $L\supseteq A$ (see e.g.~\cite{DHBook}).
However, a null $G_\delta$ set may cover many null sets, so a $G_\delta$-code of such a null $G_\delta$ set gives a name of a lot of null sets.
The (multi-)represented space of null sets plays important roles in both set theory and computability theory (for the latter, especially in the theory of algorithmic randomness \cite{DHBook}).
%
%As another example, consider the space of countable structures in a fixed countable language: A countable structure $M$ is isomorphic to a structure on $\N$, which can be coded as a stream, and we think that such a stream codes all isomorphic copies of $M$.
%The space of countable structures (with certain group action) is also studied in various fields, including model theory, invariant descriptive set theory and computable structure theory; see e.g.~\cite{Gao,MelMon18}.
\end{example}

Above we have given examples of coding by natural numbers and streams, but more generally, the theory of coding with a relative PCA (that unifies the notions of numbering, representation, multi-representation, etc.)~has also been studied in realizability theory and related areas.
There, an equivalent notion to multi-represented space (with respect to a relative PCA) is called assembly; see also Section \ref{sec:assembly}.

%\subsection{de Brecht-Pauly's approach}

We now explain an operation of changing the way a space is accessed using an oracle, which is the basis of the second view of oracle mentioned in Section \ref{sec:intro-oracle}.
In de Brecht-Pauly \cite{dB14,PaBr15}, the relativization of a represented space by an oracle is defined in the following manner.
%The following definition is due to de Brecht 2014; de Brecht-Pauly 2015

\begin{definition}[\cite{dB14,PaBr15}]\label{def:space-relative-to-transparent}
Let $U\pcolon\tpbf{N}\to\tpbf{N}$ be a transparent map.
Then given a represented space $X=(X,\delta_X)$ and a continuous map $f\colon X\to Y$ on represented spaces, define a represented space $U(X)$ and a function $U(f)\colon U(X)\to U(Y)$ as follows:
\begin{align*}
U(X)=(X,\delta_X\circ U) & & U(f)=f
\end{align*}
\end{definition}

Note that a map $f\colon X\to Y$ is $U$-continuous ($U$-computable, respectively) if and only if $f\colon X\to U(Y)$ is continuous (computable, respectively).

\begin{remark}
The origin of this idea is the notion of the {\em jump of a representation} by Ziegler \cite{Zie07}, so this notion may be called the $U$-jump of the representation space $X$.
We refer to the space $U(X)$ as the {\em $U$-relativization} of $X$.
\end{remark}

\begin{example}[Ziegler \cite{Zie07}]
Recall the limit operation introduced below Definition \ref{def:basic-computably-transparent}.
The $\lim$-relativization of the rapid Cauchy representation of $\mathbb{R}$ (see Example \ref{exa:coding-topological-space}) is equivalent to the na\"ive Cauchy representation of $\mathbb{R}$.
\end{example}

%{sec:sdst-notation}

\begin{remark}[see also \cite{dB14,PaBr15}]
Surjectivity of a transparent map implies that $U(X)$ is also a represented space.
Moreover, if $f\colon X\to Y$ is continuous, so is $U(f)\colon U(X)\to U(Y)$.
Indeed, transparency is a condition for $U$ to induce an endofunctor on the category ${\bf Rep}$ of represented spaces and continuous functions.
Furthermore, if $U$ is inflationary and idempotent, i.e., ${\rm id}=U\circ\eta$ and $U\circ U=U\circ\mu$ for some computable functions $\eta$ and $\mu$, then such $\eta$ and $\mu$ yield a monad unit and a monad multiplication, respectively.
To be more precise, both the monad unit $\eta_X\colon X\to U(X)$ and the monad multiplication $\mu_X\colon U\circ U(X)\to U(X)$ are given as identity maps, which are tracked by $\eta$ and $\mu$, respectively.
Thus, the triple $(U,\eta,\mu)$ forms a {\em monad} on ${\bf Rep}$.
Indeed, $f\colon X\to Y$ is $U$-continuous if and only if $f$ is a {\em Kleisli morphism} for the monad $(U,\eta,\mu)$.
If $U$ is computably transparent, the same is true for the category of representated spaces and computable functions.
\end{remark}

\begin{example}[Synthetic descriptive set theory \cite{PaBr15,dBPa17}]\label{exa:synthetic-DST}
Definition \ref{def:space-relative-to-transparent} forms the basis of synthetic descriptive set theory.
For instance,
%recall that the limit operator is computably transparent, so it yields an endofunctor as above.
%Then
the hyperspace of {$\tpbf{\Sigma}^0_2$ sets} in a Polish space $X$, {$\tpbf{\Sigma}^0_2(X)$}, can be given by applying the {endofunctor $\lim$} appropriately to {$\mathcal{O}(X)=C(X,\mathbb{S})$}, the represented hyperspace of {open sets} in $X$ (Example \ref{exa:coding-continuous functions}): 
To be precise, put $\mathcal{O}^{\lim}(X):=C(X,\lim(\mathbb{S}))=\tpbf{\Sigma}^0_2(X)$.
Similarly, the discrete limit operator $\lim_\Delta$ yields the hyperspace of {$\tpbf{\Delta}^0_2$ sets} in $X$, i.e., $\mathcal{O}^{\lim_\Delta}(X):=C(X,\lim_\Delta(\mathbb{S}))=\tpbf{\Delta}^0_2(X)$.
Note that $\lim(\mathbb{S})$ and $\lim_\Delta(\mathbb{S})$ are pre-dominances (in the sense of Rosolini; see e.g.~\cite[Section 4.1.2]{LoPhD95}), and $\mathcal{O}^{\lim}$ and $\mathcal{O}^{\lim_\Delta}$ give the corresponding representable classes of subobjects.
Through this particular endofunctor $\lim_\Delta$, various topological notions (e.g.~Hausdorffness, compactness, overtness) can be lifted to notions about $\tpbf{\Delta}^0_2$ sets.
As a sample result, de Brecht-Pauly \cite{dBPa17} showed that a quasi-Polish space is {$\lim_\Delta$-compact} if and only if it is {Noetherian}, i.e. every strictly ascending chain of open sets is finite.
Of course, there are endofunctors not only for $\tpbf{\Delta}^0_2$, but for various descriptive set-theoretic notions as well.
As other applications of Definition \ref{def:space-relative-to-transparent}, see e.g.~\cite{dB14,PaBr15,dBPa17}.
\end{example}

%In Section \ref{sec:assembly}, we extend Definition \ref{def:space-relative-to-transparent} to partial multifunctions.
%However, the way to extend it is not so straightforward.
%When the definition is extended in an appropriate way, even if $U$ is multivalued, when $U$ is surjective in a strong sense, the same conclusion as in the above Remark holds for the category of multi-represented spaces by the same argument.

%%%%%%%%%%%%%%%%%%%%%%%%%%%%%%%%%%%%%%%%%%%%%%%%%%%%%%%%%%%%
%%%%%%%%%%%%%%%%%%%%%%%%%%%%%%%%%%%%%%%%%%%%%%%%%%%%%%%%%%%%
%%%%%%%%%%%%%%%%%%%%%%%%%%%%%%%%%%%%%%%%%%%%%%%%%%%%%%%%%%%%
%%%%%%%%%%%%%%%%%%%%%%%%%%%%%%%%%%%%%%%%%%%%%%%%%%%%%%%%%%%%
%%%%%%%%%%%%%%%%%%%%%%%%%%%%%%%%%%%%%%%%%%%%%%%%%%%%%%%%%%%%
%%%%%%%%%%%%%%%%%%%%%%%%%%%%%%%%%%%%%%%%%%%%%%%%%%%%%%%%%%%%
%%%%%%%%%%%%%%%%%%%%%%%%%%%%%%%%%%%%%%%%%%%%%%%%%%%%%%%%%%%%

\subsection{Assembly}\label{sec:assembly}

There are several advantages to transforming the partial multifunctions on $\tpbf{N}$ into the operations on truth values, but first of all, operations on $\Omega=\mathcal{P}(\tpbf{N})$ are more compatible with multi-represented spaces.
Here, it is convenient to rephrase multi-represented spaces as follows.

\begin{definition}[see e.g.~\cite{vOBook}]
A pair of a set $X$ and a function $E_X\colon X\to\Omega^+$ is called an {\em assembly}, where $\Omega^+=\Omega\setminus\{\emptyset\}$.
\end{definition}

\begin{remark}
For a multi-represented space $(X,\delta_X)$, let $E_X(x)\subseteq\N^\N$ be the set of all {names} of $x$.
Clearly, $(X,E_X)$ is an assembly.
Conversely, if $(X,E_X)$ is an assembly, one can recover a multi-represented space $(X,\delta_X)$ by setting $x\in \delta_X(p)$ if $p\in E_X(x)$.
Therefore, a multi-represented space can be identified with an assembly.
%Such a pair $(X,E_X)$ is called an {\it assembly} (over Kleene's second algebra or the Kleene-Vesley algebra).
\end{remark}

Now, one can relativize a multi-represented space by using an operation $j$ on $\Omega$ preserving nonemptieness, i.e., $p\not=\emptyset$ implies $j(p)\not=\emptyset$.
Such an operation can be thought of as an operation on $\Omega^+$.
Let us analyze when $j$ preserves nonemptieness.
We say that $j$ is {\em trivial} if $j(p)=\emptyset$ for any $p\in\Omega$.
Recall that we impose nonemptiness (i.e., ${\rm dom}(U)\not=\emptyset$) on the definition of a transparent map $U$ (Definition \ref{def:basic-computably-transparent}); hence $j_U$ is always nontrivial.

\begin{obs}\label{obs:mono-pres-nonemp}
Any nontrivial monotone operation $j\colon\Omega\to\Omega$ preserves nonemptieness.
\end{obs}

\begin{proof}
As $j$ is nontrivial, there exists $q\in\Omega$ such that $j(q)\not=\emptyset$.
%Note that such a $q$ can be nonempty.
%Otherwise, if $q'\not=\emptyset$ then $j(q')=\emptyset$.
%In particular, for $q=\emptyset$, we have $j(q)\not=\emptyset$, and so $j(j(q))=\emptyset$ by our assumption.
%As $\emptyset\arr j(\emptyset)$ is realizable, by monotonicity, $j(\emptyset)\arr j(j(\emptyset))$ is realizable, which is impossible since $j(\emptyset)\not=\emptyset$ and $j(j(\emptyset))=\emptyset$.
%Hence, we can pick a nonempty set $q\in\Omega$ such that $j(q)\not=\emptyset$.
%
Assume that a nonempty set $p\in\Omega$ is given.
Pick ${\tt a}\in p$.
Then $\lambda x.{\tt a}$ realizes $q\arr \{{\tt a}\}$, and the identity map realizes $\{\tt a\}\arr p$.
By composing these two realizers, we see that $q\arr p$ is realizable.
Thus, by monotonicity, $j(q)\arr j(p)$ is also realizable.
As $j(q)$ is nonempty, so is $j(p)$.
Hence, $j$ preserves nonemptieness.
\end{proof}

In particular, by Theorem \ref{prop:transparent-to-modality}, if $U$ is computably transparent, then $j_U$ preserves nonemptieness.
Note also that any computably transparent map $U\pcolon\tpbf{N}\tto\tpbf{N}$ is surjective in the sense that for any $y\in\tpbf{N}$ there exists $x\in{\rm dom}(U)$ such that $U(x)\subseteq\{y\}$ (consider $f=\lambda z.y$).

\begin{definition}\label{def:mult-represented-monad-trans}
Let $j\colon\Omega^+\to\Omega^+$ be a monotone operation.
Then given an assembly $X=(X,E_X)$, and a continuous map $f\colon X\to Y$ on multi-represented spaces, define an assembly $j(X)$ and a map $j(f)\colon j(X)\to j(Y)$ as follows:
\begin{align*}
j(X)=(X,j\circ E_X),
& &
j(f)=f.
\end{align*}

We say that a function $f\pcolon X\to Y$ is {\em $j$-computable} if $f\pcolon X\to j(X)$ is computable.
The $j$-continuity is defined in the same way.
\end{definition}

%\begin{example}
%Let $U\pcolon\tpbf{N}\to\tpbf{N}$ be a partial multifunction.
%Then, for a multi-represented space $(X,\delta_X)$, define the new representation $\delta_{U(X)}$ as follows:
%
% $j_U(X)$
%\end{example}

For any monotone operation $j\colon\Omega^+\to\Omega^+$, if $f\colon X\to Y$ is computable, so is $j(f)$:
This is because computability of $f$ means that $E_X(x)\arr E_Y(f(x))$ is realizable, so monotonicity of $j$ implies that $j\circ E_X(x)\arr j\circ E_Y(f(x))$ is realizable, which means that $j(f)\colon j(X)\to j(Y)$ is computable.
The same is true for continuity.
Hence, monotonicity of $j$ ensures that $j$ yields an endofunctor on the category ${\bf MultRep}$ of multi-represented spaces (equivalently, assemblies).
Note that this endofunctor preserves underlying sets, and such a functor is called an {\em S-functor} in \cite[Section 1.6]{vOBook}.

Similarly, if $j$ is inflationary, ${\rm id}\colon X\to j(X)$ is computable, and if $j$ is idempotent, ${\rm id}\colon j\circ j(X)\to j(X)$ is computable.
This ensures that if $j$ is a Lawvere-Tierney topology, then $j$ induces a monad on ${\bf MultRep}$.
As before, a $j$-continuous function is a Kleisli morphism for this monad.

\begin{remark}
An assembly of the form $j(X)$ is also known as a $j$-assembly (see e.g.~\cite{FavO16}).
If $j$ is a Lawvere-Tierney topology, then the category of $j$-assemblies is essentially the Kleisli category of the corresponding monad.
\end{remark}

Let us make sure that Definition \ref{def:mult-represented-monad-trans} is an extension of Definition \ref{def:space-relative-to-transparent}.
Recall from Definition \ref{def:transform-mult-op-truth} (and a subsequent comment) that $j_U=U^{-1}$ for $U\pcolon\tpbf{N}\to\tpbf{N}$.

\begin{obs}\label{obs:relative-rep-sp-equiv}
Let $X$ be a represented space, and $U\pcolon\tpbf{N}\to\tpbf{N}$ be a transparent map.
Then, the $U$-relativization $U(X)$ in the sense of Definition \ref{def:space-relative-to-transparent} is equivalent to the $j_U$-relativization $j_U(X)$ in the sense of Definition \ref{def:mult-represented-monad-trans}.
\end{obs}

\begin{proof}
On the one hand, $x\in\delta_X\circ U(p)$ if and only if $p\in\tpbf{N}$ is a $U(X)$-name of $x$ if and only if $U(p)$ is an $X$-name of $x$.
On the other hand, $E_X(x)$ is the set of all $X$-names of $x$, and so $j_U\circ E_X(x)$ is the set of all $j_U(X)$-names of $x$.
As $j_U=U^{-1}$, $p\in j_U\circ E_X(x)$ if and only if $U(p)\in E_X(x)$ if and only if $U(p)$ is an $X$-name of $x$.
Thus, $p$ is a $j_U(X)$-name of $x$ if and only if $p$ is a $U(X)$-name of $x$.
\end{proof}

Definition \ref{def:mult-represented-monad-trans} leads us to the right extension of Definition \ref{def:space-relative-to-transparent} to partial multifunctions.
Let $U$ be a partial multifunction on $\tpbf{N}$, and $X=(X,\delta_X)$ be a multi-represented space.
Then, define a new multi-representation $\delta_X^U$ as follows:
$p\in\tpbf{N}$ is a $\delta_X^U$-name of $x\in X$ if and only if $U(p)$ is defined and any solution $q\in U(p)$ is a $\delta_X$-name of $x$. 
Then define $U(X)=(X,\delta_X^U)$.
One can check that $U(X)$ is equivalent to $j_U(X)$ in the sense of Definition \ref{def:mult-represented-monad-trans}.

In general, $j$ does not necessarily preserve intersection.
Thus, relativization of spaces by operations on $\Omega$ is a broader concept than relativization of spaces by partial multifunctions on $\tpbf{N}$.

\begin{remark}
The relativization of a space by an operation on $\Omega$ which is not $\bigcap$-preserving can also be useful.
For instance, one can consider $\neg\neg X$ for a multi-represented space $X$, where note that $\neg\neg \circ E_X(x)=\tpbf{N}$ for any $x\in X$.
On the other hand, there is a notion called a strong subobject.
In the category ${\bf MultRep}$, strong subobjects of a multi-represented space $X$ are exactly the ones isomorphic to subspaces of $X$; that is, multi-represented spaces of the form $(Y,E_X\upto Y)$ for some $Y\subseteq X$.
Then the $\neg\neg$-relativization $\neg\neg 2$ of a two-point space $2$ gives a strong-subobject classifier in the category ${\bf MultRep}$, which ensures that ${\bf MultRep}$ forms a {\em quasi-topos} (see e.g.~\cite{MenniPhD00,FaberMsc14}).
For other applications of $\neg\neg$-relativization, see also Sections \ref{sec:secret-input} and \ref{sec:realizability-relative-to-oracle}.
\end{remark}

Two multi-represented spaces $X$ and $Y$ are computably isomorphic if there exists a bijection $h\colon X\to Y$ such that both $h$ and $h^{-1}$ are computable.

\begin{prop}
For any operations $j,k\colon\Omega^+\to\Omega^+$, if $j\equiv_rk$ then $j(X)$ and $k(X)$ are computably isomorphic for any multi-represented space $X$.
If $j$ and $k$ are Lawvere-Tierney topologies, then the converse also holds.
\end{prop}

\begin{proof}
For the first assertion, we claim that $j\leq_rk$ if and only if the identity map (on the underlying set $X$), ${\rm id}\colon j(X)\to k(X)$, is computable for any multi-represented space $X$.
For the forward direction, if $j\leq_rk$ then some ${\tt a}\in N$ realizes $j(p)\arr k(p)$.
In particular, for any $x\in X$, ${\tt x}\in j(E_X(x))$ (i.e., ${\tt x}$ is a name of $x$ in $j(X)$) implies ${\tt a}\ast{\tt x}\in k(E_X(x))$ (i.e., ${\tt a}\ast{\tt x}$ is a name of $x$ in $k(X)$).
This means that ${\tt a}$ tracks ${\rm id}\colon j(X)\to k(X)$, so ${\rm id}\colon j(X)\to k(X)$ is computable.
%Similarly, $k\leq_rj$ implies that ${\rm id}\colon k(X)\to j(X)$ is computable.
Therefore, if $j\equiv_rk$ then $j(X)$ is computably isomorphic to $k(X)$ via the identity map.

For the backward direction, consider a multi-represented space $\hat{\Omega}$, where its underlying set is $\Omega^+$ and an assembly map is $E_{\hat{\Omega}}={\rm id}$.
Then, consider the space $j(\hat{\Omega})$, where note that the set of names of $p\in\Omega^+$ in $j(\hat{\Omega})$ is $j(p)$.
Hence, computability of ${\rm id}\colon j(\hat{\Omega})\to k(\hat{\Omega})$ means that $j(p)\arr k(p)$ is realizable via a single realizer; that is, $j\leq_rk$.

To show the second assertion, we claim that if $j(\hat{\Omega})$ computably embeds into $k(\hat{\Omega})$ then $k\leq_rj$ holds whenever $k$ is a Lawvere-Tierney topology.
%note that one can think of $\Omega^+$ as a multi-represented space, where $E_{\Omega^+}={\rm id}$.
%Then, consider the space $j(\Omega^+)$, where note that the set of names of $p\in\Omega^+$ in $j(\Omega^+)$ is $j(p)$.
Let $h\colon j(\hat{\Omega})\to k(\hat{\Omega})$ be a computable embedding; that is, $p\mapsto h(p)$ and $h(p)\mapsto p$ are computable.
This means that $j(p)\iffarr k(h(p))$ is realizable, where a realizer is independent of $p$.
Applying monotonicity of $k$ to this property, then idempotence of $k$, and then using the above property again,
\[
k(j(p))\iffarr k(k(h(p)))\iffarr k(h(p))\iffarr j(p)
\]
is realizable.
Furthermore, as $j$ is inflationary, i.e., $p\arr j(p)$ is realizable, so applying monotonicity of $k$ to this, and using the above formula, we obtain that $k(p)\arr k(j(p))\iffarr j(p)$ is realizable independent of $p$, and in particular $k\leq_r j$.
Now, for any computable isomorphism $h\colon j(X)\to k(X)$, both $h$ and $h^{-1}$ are computable embeddings, so we get $k\equiv_rj$.
\end{proof}

%Moreover, if $k$ is a Lawvere-Tierney topology, this is also equivalent to that $j(X)$ computably embeds into $k(X)$.

\begin{question}
Does there exist locally monotone operators $j,k\colon\Omega^+\to\Omega^+$ such that $j(X)$ and $k(X)$ are computably isomorphic for any multi-represented space $X$, but $j\not\equiv_r k$?
\end{question}

\subsection{Universal closure operator}\label{sec:universal-closure-operator-intro}
It is known that there is a correspondence between Lawvere-Tierney topology and {\em universal closure operator} \cite[Sections A4.3 and A4.4]{elephant}.
That is, given a subobject $A\mono X$, one can consider its {\em $j$-closure} ${\sf cl}_j(A)\mono X$.
One can also consider $j$-closedness (i.e., ${\sf cl}_j(A)\equiv A$), $j$-density (i.e., ${\sf cl}_j(A)\equiv X$) and so on.

We will give precise definitions of these notions.
Our target is the category of represented spaces.
No abstract setting is required, and the discussion is entirely elementary.
The key idea is the following:
%The realizability interpretation gives the notion of witness for correctness of a formula.
One may use a formula $\varphi$ to define a subset $A=\{x\in X:\varphi(x)\}$ of a represented space $X$, which entails the notion of witness for $x\in A$.
That is, a witness for $x\in A$ is a witness for the formula $\varphi(x)$ being true, i.e., a realizer for $\varphi(x)$ with respect to a corresponding realizability interpretation.

\begin{definition}
A {\em witnessed subset} $A$ of a represented space $X$ is a represented space such that $A\subseteq X$ and every name of $x\in A$ is a pair $\langle w,p\rangle$ of an $X$-name $p$ of $x$ and some $w\in\tpbf{N}$.
In this case, $w$ is called a {\em witness} for $x\in A$.
\end{definition}

One can see that a subobject of a represented space is nothing more than a witnessed subset.

To explain this, we first introduce the notion of subobject.
In the category of (multi-)represented spaces, a monomorphism is merely an injective computable function.

\begin{definition}
A monomorphism $i\colon A\mono X$ is included in $j\colon B\mono X$ if there exists a morphism $k\colon A\to B$ such that $i=j\circ k$.
If $i$ is included in $j$ and vice versa, we write $i\equiv j$.
A subobject is the $\equiv$-equivalence class of a mono.
\end{definition}

We use ${\sf Sub}(X)$ to denote the set of all subobjects of $X$ ordered by the inclusion relation.
%In the case of sets, one can identify an injection $m\colon S\to X$ with a subset $m[S]\subseteq X$.
%Then the inclusion relation between subsets of $X$ can be characterized using injections as follows.
%
In the category of represented spaces, any $\equiv$-equivalence class of a mono contains an inclusion map; that is, its underlying set is a subset $A\subseteq X$ such that the inclusion map $A\embed X$ is computable.
Thereafter, we assume that a subobject is an inclusion map $A\mono X$.
In this case, there is no need to name the monomorphism $A\mono X$ since it is uniquely determined.
Therefore, it is not confusing to write $A\leq B$ and $A\equiv B$ for the inclusion relation and the bi-inclusion relation, respectively.
One can also show the following:

\begin{obs}[see also \cite{Kih24}]\label{obs:subobject-witnessed-subset}
A witnessed subset of $X$ is a subobject of $X$.
Conversely, every subobject of $X$ is equivalent to a witnessed subset of $X$.
\end{obs}

\begin{definition}\label{def:universal-closure-operator-b}
Given an operation $j\colon\Omega\to\Omega$ and an assembly $X$, the {\em $j$-closure} of a subobject $A\mono X$ is a subobject ${\sf cl}_jA\mono X$ defined as follows:
\begin{align*}
E_{{\sf cl}_j(A)}(x)=j(\hat{E}_A(x))\land E_X(x), & & {\sf cl}_j(A)=\{x\in A:E_{{\sf cl}_j(A)}(x)\not=\emptyset\},
\end{align*}
where $\hat{E}_A(x)$ is the set of all $A$-names of $x\in X$.
Note that $\hat{E}_A(x)=\emptyset$ if $x\in X\setminus A$.

A subobject $A\mono X$ is {\em $j$-closed} if its $j$-closure is equivalent to $A\mono X$.
Similarly, a subobject $A\mono X$ is  {\em $j$-dense} if its $j$-closure is equivalent to $X\mono X$.
\end{definition}

If $j$ is a Lawvere-Tierney topology, this notion deserves the name ``closure'', but this name is less appropriate if $j$ does not satisfy sufficient conditions.
If $j$ is a Lawvere-Tierney topology, then ${\sf cl}_j$ is what is known as a universal closure operator.

\begin{definition}\label{def:subobject-universal-closure}
Let $c$ be an assignment of a map $c_X\colon{\sf Sub}(X)\to{\sf Sub}(X)$ to each object $X$.
Then, consider the following properties for any subobject $A,B\mono X$:
\begin{enumerate}
\item ({\sf monotone}) If $A\leq B$ then $c_X(A)\leq c_X(B)$.
\item ({\sf inflationary}) $A\leq c_X(A)$.
\item ({\sf idempotent}) $c_X(c_X(A))\leq c_X(A)$.
\item ({\sf natural}) $c_Y(f^\ast A)\equiv f^\ast (c_X(A))$ for any $f\colon Y\to X$.
\end{enumerate}

Here, the underlying set of the pullback $f^\ast A$ is $\{x\in Y:f(x)\in A\}$, and a name of $x\in f^\ast A$ is the pair of a name of $x\in Y$ and a name of $f(x)\in A$.
Often, $c_X(A)$ is simply written as $c(A)$.
A {\em universal closure operator} is an assignment satisfying all of the above conditions (1)--(4).
\end{definition}

The following is well-known; see e.g.~\cite[Lemma A4.4.2]{elephant}.

\begin{fact}
An operation $j\colon\Omega\to\Omega$ is a Lawvere-Tierney topology if and only if ${\sf cl}_j$ is a universal closure operator.
\end{fact}

As we have already seen, not only Lawvere-Tierney topologies but also (computably) monotone operations are important in the context of oracle computations.
Recall that $j\colon\Omega\to\Omega$ is monotone if $\forall p,q.\,(p\arr q)\arr(j(p)\arr j(q))$ is realizable.
A {\em universal monotone operator} is an assignment satisfying ({\sf monotone}) and ({\sf natural}) in Definition \ref{def:subobject-universal-closure}.

\begin{prop}
An operation $j\colon\Omega\to\Omega$ is monotone if and only if ${\sf cl}_j$ is a universal monotone operator.
\end{prop}

\begin{proof}
In this proof, by an abuse of notation, $\hat{E}_A$ is abbreviated to $E_A$.

($\Rightarrow$)
The condition $A\leq B$ means that $\forall x.[E_A(x)\arr E_B(x)]$ is realizable.
Hence, by monotonicity, $\forall x.[j(E_A(x))\arr j(E_B(x))]$ is realizable, which means ${\sf cl}_j(A)\leq{\sf cl}_j(B)$.
For ${\sf cl}_j(f^\ast A)\leq f^\ast({\sf cl}_j(A))$, since $E_Y(x)\land E_A(f(x))\arr E_A(f(x))$ and $E_Y(x)\arr E_X(f(x))$ are clearly realizable, by computable monotonicity, the following is also realizable:
\begin{align*}
E_{{\sf cl}_j(f^\ast A)}(x)&=E_Y(x)\land j(E_{f^\ast A}(x))\\
&=E_Y(x)\land j(E_Y(x)\land E_A(f(x)))\\
&\arr E_Y(x)\land j(E_A(f(x)))\\
&\iffarr E_Y(x)\land E_X(f(x))\land j(E_A(f(x)))=E_{f^\ast({\sf cl}_j(A))}(x).
\end{align*}

To see $f^\ast({\sf cl}_j(A))\leq{\sf cl}_j(f^\ast A)$, we claim that computable monotonicity implies that $\forall p,q.[j(p)\land q\arr j(p\land q)]$ is realizable.
Independently of $p,q$, some ${\tt i}$ realizes $p\land q\arr p\land q$, so by currying, some ${\tt a}$ realizes $q\arr p\arr p\land q$.
Given ${\tt b}\in q$, ${\tt a}\ast{\tt b}$ realizes $p\arr p\land q$, so some ${\tt u}\ast({\tt a}\ast {\tt b})$ realizes $j(p)\arr j(p\land q)$ by computable monotonicity.
Hence, $\lambda {\tt b}.{\tt u}\ast({\tt a}\ast {\tt b})$ realizes $q\arr j(p)\arr j(p\land q)$.
By uncurrying, some ${\tt c}$ realizes $j(p)\land q\arr j(p\land q)$.
Then, the following is realizable:
\begin{align*}
E_{f^\ast({\sf cl}_j(A))}(x)
&\iffarr j(E_A(f(x)))\land E_Y(x)\\
&\arr j(E_A(f(x))\land E_Y(x))\land E_Y(x)\\
&= j(E_{f^\ast A}(x))\land E_Y(x)=E_{{\sf cl}_j(f^\ast A)}(x).
\end{align*}

($\Leftarrow$)
Let the underlying sets of $P$ and $Q$ be $\Omega\times\Omega$ and define $E_P(p,q)=p$ and $E_Q(p,q)=q$.
One can consider $P,Q$ as subobjects of the trivial assembly $\Omega\times\Omega$.
Let $\iota_Q\colon Q\mono\Omega\times\Omega$ be the inclusion map.
By (4), we have ${\sf cl}_j(\iota_Q^\ast P)\equiv\iota_Q^\ast({\sf cl}_j(P))$, so the following is realizable:
\begin{align*}
&j(E_{\iota_Q^\ast P}(p,q))=j(E_Q(p,q)\land E_P(p,q))=j(q\land p)\\
\iffarr\ & E_{\iota_Q^\ast({\sf cl}_j(P))}(p,q)=E_Q(p,q)\land E_{{\sf cl}_j(P)}(p,q)=q\land j(E_P(p,q))=q\land j(p).
\end{align*}

Hence, $\forall p,q.[j(p)\land q\iffarr j(p\land q)]$ is realizable.
In particular, independently of $p,q$, some ${\tt a}$ realizes $(j(p)\land (p\arr q))\arr j(p\land (p\arr q))$.
By currying, some ${\tt b}$ realizes $j(p)$
Let the underlying set of $R$ be $\Omega\times\Omega$ and define $E_R(p,q)=p\land (p\arr q)$.
Again, $R$ is a subobject of $\Omega\times\Omega$, and $R\leq B$ holds since $\forall p,q.[(p\land (p\arr q))\arr q]$ is realizable.
By monotonicity of ${\sf cl}_j$, we get ${\sf cl}_j(R)\leq {\sf cl}_j(B)$, which means that some ${\tt b}$ realizes $\forall p,q.[j(p\land (p\arr q))\arr j(q)]$.
Composing ${\tt a}$ and ${\tt b}$, we get some realizer ${\tt c}$ of $(j(p)\land (p\arr q))\arr j(q)$.
By currying, some ${\tt d}$ realizes $(p\arr q)\arr (j(p)\arr j(q))$.
\end{proof}

Conversely, a monotone operation on $\Omega$ can always be reconstructed from a universal monotone operator.

\begin{definition}
Let $c$ be a universal monotone operator.
Consider the trivial assembly $\Omega$, that is, anything is a name of $p\in\Omega$.
This has a subobject $\hat{\Omega}\mono\Omega$, where the underlying set of $\hat{\Omega}$ is $\Omega^+=\Omega\setminus\{\emptyset\}$ and its assembly map is given by $E_{\hat{\Omega}}(p)=p$.
Then the operation $j_c\colon\Omega\to\Omega$ is defined by $j_c(p)=E_{c(\hat{\Omega})}(p)$.
\end{definition}

\begin{obs}
For $j,k\colon\Omega\to\Omega$, if $j\leq_rk$ then ${\sf cl}_j(A)\leq{\sf cl}_k(A)$ for any $A\mono X$.
For universal monotone operators $c$ and $d$, if $c(A)\leq d(A)$ for any $A\mono X$ then $j_c\leq_rj_d$.
\end{obs}

\begin{proof}
If $j\leq_rk$ then $j(E_A(x))\arr k(E_A(x))$ is realizable, which means ${\sf cl}_j(A)\leq{\sf cl}_k(A)$.
For $\hat{\Omega}\mono\Omega$, if $c(\hat{\Omega})\leq d(\hat{\Omega})$ then $E_{c(\hat{\Omega})}(p)\arr E_{d(\hat{\Omega})}(p)$ is realizable.
By definition, this means $j_c\leq_rj_d$.
\end{proof}

\begin{prop}
For any $j\colon\Omega\to\Omega$, $j_{{\sf cl}_j}\equiv_rj$ holds.
For any universal monotone operator $c$, ${\sf cl}_{j_c}(A)\equiv c(A)$ holds for any $A\mono X$.
\end{prop}

\begin{proof}
By definition, we get $j_{{\sf cl}_j}(p)=E_{{\sf cl}_j(\hat{\Omega})}(p)=j(E_{\hat{\Omega}}(p))\land E_\Omega(p)=j(p)\land\tpbf{N}$.
Clearly, $j(p)\land\tpbf{N}\iffarr j(p)$ is realizable, so $j_{{\sf cl}_j}\equiv_rj$ holds.

For each $A\mono X$,
%we extend $E_A\colon A\to\Omega^+$ to $\hat{E}_A\colon X\to\Omega$, where $\hat{E}_A(x)=\emptyset$ for $x\in X\setminus A$.
%Then 
$\hat{E}_A\colon X\to\Omega$ is computable if $\Omega$ is thought of as a trivial assembly.
By naturality, we have $c_X(\hat{E}_A^\ast\hat{\Omega})\equiv \hat{E}_A^\ast(c_\Omega(\hat{\Omega}))$.
First, a name of $x\in \hat{E}_A^\ast\hat{\Omega}\mono X$ is the pair of a name of $x\in X$ and a name of $\hat{E}_A(x)\in\hat{\Omega}$, i.e., an element of $\hat{E}_A(x)$.
Note that $x\in \hat{E}_A^\ast\hat{\Omega}$ implies that $\hat{E}_A(x)\not=\emptyset$, so $\hat{E}_A(x)=E_A(x)$.
Thus, we get $\hat{E}_A^\ast\hat{\Omega}\equiv A$; therefore $c_X(\hat{E}_A^\ast\hat{\Omega})\equiv c_X(A)$ by monotonicity.
Next, a name of $x\in \hat{E}_A^\ast(c_\Omega(\hat{\Omega}))\mono X$ is the pair of a name of $x\in X$ and a name of $\hat{E}_A(x)\in c_\Omega(\hat{\Omega})$, which is an element of $E_{c_\Omega(\hat{\Omega})}(\hat{E}_A(x))=j_c(\hat{E}_A(x))$.
This is also a ${\sf cl}_{j_c}(A)$-name of $x$, so we obtain ${\sf cl}_{j_c}(A)=\hat{E}_A^\ast(c_\Omega(\hat{\Omega}))$.
Consequently, ${\sf cl}_{j_c}(A)\equiv c(A)$.
\end{proof}

As we have already seen, a computably transparent map corresponds to a monotone operator on $\Omega$, so we expect that a jump of representation can be regarded as an application of a universal monotone operator.
It should be noted, however, that a jump of representation is for represented spaces, i.e., objects, while an application of a universal monotone operator is for subobjects.
To bridge this gap, it is necessary to consider the trivial assembly $\nabla X$ for a set $X$.
Here, the underlying set of $\nabla X$ is $X$, and anything is a name of $x\in\nabla X$.
Any (multi-)represented space $X$ can be thought of as a subobject $X\mono\nabla X$ since the identity map $X\to\nabla X$ is obviously computable.

\begin{obs}
Let $X$ be a multi-represented space and  $j\colon\Omega^+\to\Omega^+$ be a monotone operation.
Then the $j$-relativization $j(X)$ in the sense of Definition \ref{def:mult-represented-monad-trans} is equal to ${\sf cl}_j(X)\mono \nabla X$.
\end{obs}

\subsection{Applicative morphism}\label{sec:applicative-morphism}

Next, let us discuss the relationship between computably transparent maps and applicative morphisms (PCA-morphisms).
The relationship between these notions has been suggested by Thomas Streicher (in private communication), and the goal of this section is to provide a rigorous analysis of this relationship.

\medskip
\noindent
{\it Partial combinatory algebra:}
As one of the oldest Turing-complete computational models, Sch\"onfinkel's {\em combinatory logic} is well known (especially in the context of lambda calculus).
An algebraic generalization of this model, called a partial combinatory algebra (abbreviated as PCA), has been studied in depth.
As a benefit of such a generalization, a PCA also encompasses, in a rather broad sense, notions such as computation relative to an oracle, infinitary computation, etc.
A PCA is also used as the basis for the construction of a type of topos called a realizability topos \cite{vOBook}.

\begin{definition}[see e.g.~\cite{vOBook}]\label{def:partial-combinatory-algebra-int}
A partial magma is a pair $(M,\ast)$ of a set $M$ and a partial binary operation $\ast$ on $M$.
We often write $xy$ instead of $x\ast y$, and as usual, we consider $\ast$ as a left-associative operation, that is, $xyz$ stands for $(xy)z$.
A partial magma is {\em combinatory complete} if, for any term $t(x_1,x_2,\dots,x_n)$, there is $a_t\in M$ such that $a_tx_1x_2\dots x_{n-1}\downarrow$ and $a_tx_1x_2\dots x_n\simeq t(x_1,x_2,\dots,x_n)$.
By abusing the notation, we use $\lambda x_1x_2\dots x_n.t(x_1,x_2,\dots,x_n)$ to denote such an $a_t$.
The symbols ${\sf k}$ and ${\sf s}$ are used to denote $\lambda xy.x$ and $\lambda xyz.xz(yz)\in M$, respectively.
A combinatory complete partial magma is called a {\em partial combinatory algebra} (abbreviated as {\em PCA}).
A {\em relative PCA} is a triple $(\tplf{N},\tpbf{N},\ast)$ such that $\tplf{N}\subseteq\tpbf{N}$, both $(\tpbf{N},\ast)$ and $(\tplf{N},\ast\upto \tplf{N})$ are PCAs, and share combinators ${\sf s}$ and ${\sf k}$.
\end{definition}

For basics on a relative PCA, we refer the reader to van Oosten \cite[Sections 2.6.9 and 4.5]{vOBook}.
Typical examples are computations on natural numbers and on streams, which were treated in Section \ref{sec:realizability-PCA}.
%In this article, the boldface algebra $\tpbf{P}$ is always the set $\om^\om$ of all infinite sequences.
%In Brouwerian intuitionistic mathematics, a relative PCA may be considered as a pair of {\em lawlike} sequences and {\em lawless} sequences.
In descriptive set theory, the idea of a relative PCA is ubiquitous, which usually occurs as a pair of {\em lightface} and {\em boldface} pointclasses.
Hence, one might say that a large number of nontrivial, deep, examples of relative PCAs have been (implicitly) studied in descriptive set theory; see Section \ref{sec:pointclass-partial-combinatory-algebra}.

\medskip
\noindent
{\it Applicative morphism:}
The basic idea behind applicative morphism is the concept of structure-compatible coding.
Coding here can be numbering, represented space, assembly, or whatever, but if an object to be coded has some mathematical structure, such as an algebraic structure or topological structure, then coding should be compatible with that structure.
For example, given an algebra $A$ with a binary operation $\ast$, a coding of $A$ should be compatible with $\ast$, which means that the coding should be such that $\ast\colon A^2\to A$ is computable.

The algebraic structure considered here is a (relative) PCA.
An applicative morphism is a compatible coding of PCA $\mathcal{A}=(\tpbf{A},\ast_A)$ by PCA $\mathcal{B}=(\tpbf{B},\ast_B)$ in the above sense; that is, it is one such multi-coding of $\mathcal{A}$ by $\mathcal{B}$ (i.e., a multi-representation or an assembly) that has the property of allowing the application $\ast_A$ to be simulated in $\mathcal B$.

\begin{definition}[Longley \cite{LoPhD95}, see also \cite{vOBook}]\label{def:longley-app-morphi}
Let $\mathcal{A}=(\tpbf{A},\ast_A)$ and $\mathcal{B}=(\tpbf{B},\ast_B)$ be PCAs.
An {\em applicative morphism} from $\mathcal A$ to $\mathcal B$ is a map $\gamma$ such that $(\tpbf{A},\gamma)$ is an assembly over $\mathcal{B}$ for which $\ast_A\colon (\tpbf{A},\gamma)^2\to(\tpbf{A},\gamma)$ is computable.
% map $\gamma\colon\tpbf{A}\to\mathcal{P}^+(\tpbf{B})$ such that a computable evaluation map ${\rm eval}_\gamma\pcolon \gamma([\subseteq \tpbf{A}\to\tpbf{A}])\times\gamma(\tpbf{A})\to\gamma(\tpbf{A})$ exists, where ${\rm eval}_\gamma(f,x)=f(x)$ for any $x\in{\rm dom}(f)$.
\end{definition}

To be more specific, $\gamma$ is an applicative morphism if and only if there exists ${\tt ev}\in \tplf{B}$ such that, for any ${\tt f},{\tt x}\in\tpbf{A}$ and any ${\tt F},{\tt X}\in\tpbf{B}$, the conditions ${\tt F}\in\gamma({\tt f})$ and ${\tt X}\in\gamma({\tt x})$ and ${\tt f}\ast_A{\tt x}\downarrow$ imply ${\tt ev}\ast_B\langle{\tt F},{\tt X}\rangle\downarrow$ and ${\tt ev}\ast_B\langle{\tt F},{\tt X}\rangle\in\gamma({\tt f}\ast_A{\tt x})$.
The latter condition is essentially the same as the usual definition of applicative morphism (see \cite{vOBook}).
%Roughly speaking, an applicative morphism is a coding of $\mathcal A$ in $\mathcal B$, i.e., a $\mathcal B$-multi-representation of $\mathcal A$ that allows computations on $\mathcal A$ to be simulated on $\mathcal B$.

To be precise, Definition \ref{def:longley-app-morphi} is for (non-relative) PCAs, not for relative PCAs.
The definition of an applicative morphism for relative PCAs is given by Zoethout \cite{ZoePhD}, for example.
%Here, the above definition seems to be considered only for (non-relative) PCAs in the existing literature;
If we are dealing with relative PCAs, we assume that any computable element ${\tt a}$ in $\mathcal{A}$ (i.e., ${\tt a}\in \tplf{A}$) has a computable $\gamma$-name in $\tplf{B}$.
In particular, both ${\tt s},{\tt k}\in\tplf{A}$ have computable $\gamma$-names; that is, both $\gamma({\tt s})\cap\tplf{B}$ and $\gamma({\tt k})\cap\tplf{B}$ are nonempty.

The notion of applicative morphism is usually used as a tool to convert coding by one PCA to coding by another PCA.
%Specifically, let $\mathcal{A}=(A,\tpbf{A},\ast_A)$ and $\mathcal{B}=(B,\tpbf{B},\ast_B)$ be relative PCAs.
To transform an argument on $\mathcal A$ into an argument on $\mathcal B$, all we need is a (multi-)map from $\mathcal A$ to $\mathcal B$, so suppose that an applicative morphism $\gamma\colon\tpbf{A}\to\mathcal{P}^+(\tpbf{B})$ is given to us, where $\mathcal{P}^+(\tpbf{B})=\mathcal{P}(\tpbf{B})\setminus\{\emptyset\}$.
Then, any multi-represented space over $\mathcal{A}$ can be thought of as a multi-represented space over $\mathcal{B}$.
To be more precise:

\begin{definition}
Given $E\colon X\to\mathcal{P}^+(\tpbf{A})$, define $\gamma\circ E\colon X\to\mathcal{P}^+(\tpbf{B})$ by $\gamma\circ E(x)=\bigcup\{\gamma(a):a\in E(x)\}$.
For an assembly $X=(X,E)$, define $\gamma(X)=(X,\gamma\circ E)$.
\end{definition}

Note that in \cite{vOBook}, the symbol $\gamma^\ast(X,E)$ is used instead.
This is considered a special case of Definition \ref{def:mult-represented-monad-trans}.
This construction suggests that an applicative morphism yields an $S$-functor (i.e., a set-preserving functor) from ${\bf MultRep}_\mathcal{A}$ to ${\bf MultRep}_\mathcal{B}$.
Indeed, an applicative morphism is known to correspond to a {\em regular $S$-functor} (see \cite[Theorem 1.6.2]{vOBook}).

Applicative morphism has another equivalent definition.
We note that $\tpbf{A}$ itself can be viewed as a represented space over $\mathcal{A}$ by the identity representation (or equivalently $E\colon a\mapsto\{a\}$).
Also note that we have the represented function space $[\subseteq\tpbf{A}\to\tpbf{A}]$, where the underlying set is the collection of all partial continuous (i.e., boldface realizable) functions, and $a\in\tpbf{A}$ is a name of $f\pcolon\tpbf{A}\to\tpbf{A}$ if and only if $f(x)=a\ast_A x$ for any $x\in{\rm dom}(f)$; see also Example \ref{exa:coding-continuous functions}.
%In the following, $\gamma$-relativizations $\gamma(\tpbf{A})$ and $\gamma([\subseteq\tpbf{A}\to\tpbf{A}])$ of these spaces is discussed.

\begin{obs}
An applicative morphism from $\mathcal A$ to $\mathcal B$ is exactly a map $\gamma\colon\tpbf{A}\to\mathcal{P}^+(\tpbf{B})$ such that the evaluation map ${\rm eval}_\gamma\pcolon \gamma([\subseteq \tpbf{A}\to\tpbf{A}])\times\gamma(\tpbf{A})\to\gamma(\tpbf{A})$ is computable, where ${\rm eval}_\gamma(f,x)=f(x)$ for any $x\in{\rm dom}(f)$.
\end{obs}

As mentioned in Section \ref{sec:assembly}, a transparent map always yields an $S$-(endo)functor.
On the other hand, an applicative morphism corresponds to a regular $S$-functor.
The question then arises, what is the relationship between a transparent map and an applicative morphism?

\begin{obs}
A multimap $U\pcolon \tpbf{N}\tto \tpbf{N}$ is computably transparent if and only if the evaluation map ${\rm eval}_U\pcolon[\subseteq \tpbf{N}\to\tpbf{N}]\times U(\tpbf{N})\to U(\tpbf{N})$ is computable if and only if the partial application map $\ast_A\pcolon \tpbf{N}\times U(\tpbf{N})\to U(\tpbf{N})$ is computable.
\end{obs}

Here, recall that the $U$-relativization of a (multi-)represented space has been introduced below Observation \ref{obs:relative-rep-sp-equiv}, and $\tpbf{N}$ can be thought of as a represented space via the identity representation.

\begin{proof}
By definition, $U$ is computably transparent if and only if given a continuous function $f$ one can effectively find $F$ such that $U\circ F$ refines $f\circ U$.
The latter means that if ${\bf x}\in\tpbf{N}$ is a $U$-name of $x\in\tpbf{N}$ then $F({\bf x})$ is a $U$-name of $f(x)$.
Hence, if $U$ is computably transparent, then given a name ${\tt f}$ of $f$ and a $U$-name ${\tt x}$ of $x$, one one can effectively find a $U$-name $F({\bf x})$ of $f(x)$; that is, the evaluation map is computable.
Conversely, if the evaluation map is computable, then given a name ${\tt f}$ of $f$ and a $U$-name ${\tt x}$ of $x$, one one can effectively find a $U$-name $u({\tt f},{\tt x})$ of $f(x)$, which means that $\lambda x.U\circ u({\tt f},x)$ refines $f\circ U$.
Thus, $U$ is computably transparent via $\lambda fx.u(f,x)$.
The equivalence of the second and third conditions is obvious.
\end{proof}

Given a multimap $U\pcolon \tpbf{N}\tto \tpbf{N}$, consider $\gamma_U\colon\tpbf{N}\to\mathcal{P}(\tpbf{N})$ defined by ${\tt a}\mapsto j_U(\{{\tt a}\})$; that is, $\gamma_U({\tt a})=\{{\tt x}\in{\rm dom}(U):U({\tt x})\subseteq\{{\tt a}\}\}$.
As mentioned below Observation \ref{obs:mono-pres-nonemp}, if $U$ is computably transparent, then $U$ is surjective, so $\gamma_U({\tt a})$ is nonempty.
By Observation \ref{obs:relative-rep-sp-equiv}, one can see that $\gamma_U(X)$ is equivalent to $U(X)$; that is, ${\tt a}$ is a $\gamma_U(X)$-name of $x$ if and only if ${\tt a}$ is a $U(X)$-name of $x$.

The condition of $U$ being computably transparent is similar to
%$\gamma_U\colon\tpbf{N}\to\mathcal{P}^+(\tpbf{N})$ being an applicative morphism,
$\gamma_U$ is an applicative morphism from $\tpbf{N}$ to $\tpbf{N}$,
but slightly different.
We say that a computably transparent map $U$ is {\em binary parallelizable} if there exists a computable function $p$ such that $U(p(x,y))\subseteq U(x)\times U(y)$ for any $x,y\in{\rm dom}(U)$.
Note that such a map $U$ preserves binary product, i.e., $U(X)\times U(Y)\simeq U(X\times Y)$, since there is an algorithm which, given $U$-names of $x\in X$ and $y\in Y$, returns a $U$-name of $(x,y)\in X\times Y$.
Note also that if $U$ is binary parallelizable then given any $f\pcolon\tpbf{N}^m\to\tpbf{N}$ one can effectively find $F\pcolon\tpbf{N}^m\to\tpbf{N}$ such that $f(U(x_1),\dots,U(x_m))=U(F(x_1,\dots,x_m))$.

\begin{obs}
If a computably transparent map $U\pcolon \tpbf{N}\tto \tpbf{N}$ is binary parallelizable, then 
%$\gamma_U\colon\tpbf{N}\to\mathcal{P}^+(\tpbf{N})$ is an applicative morphism.
$\gamma_U$ is an applicative morphism from $\tpbf{N}$ to $\tpbf{N}$.
\end{obs}

\begin{proof}
This is because, given $U$-names of $f$ and $x$, one can effectively find an $U$-name of the pair $(f,x)$.
As $\tpbf{N}$ is a PCA, we know a code of an evaluation map, ${\rm eval}$, on $\tpbf{N}$.
Then, ${\rm eval}_U({\rm eval},(f,x))={\rm eval}(f,x)=f(x)={\rm eval}_\gamma(f,x)$, where $\gamma=\gamma_U$.
As ${\rm eval}_U$ and ${\rm eval}$ are computable, so is ${\rm eval}_\gamma$.
Next, if ${\tt a}\in \tplf{N}$ then we have $U({\tt u}(\lambda x.{\tt a}){\tt b})\subseteq (\lambda x.{\tt a})U({\tt b})\subseteq \{{\tt a}\}$; that is, ${\tt a}$ has a $U$-name ${\tt u}(\lambda x.{\tt a}){\tt b}$ in $\tplf{N}$.
%In particular, ${\tt s}$ and ${\tt k}$ have computable $\gamma_U$-names.
\end{proof}

However, there are many types of oracle computation that do not preserve binary product (where note that $(j_U(p)\land q)\arr j_U(p\land q)$ is always realizable by internal monotonicity, but $(j_U(p)\land j_U(q))\arr j_U(p\land q)$ is not always so).
These include the type of computation where multiple oracles are provided, but only one oracle can be accessed during a single computation process:

\begin{example}\label{exa:comp-transparent-not-parallel}
Given $h_0,h_1\pcolon\tpbf{N}\to\tpbf{N}$, define $U(i,e,x)=\varphi_e(h_i(x))$.
Then, clearly $\varphi_d\circ U(i,e,x)=\varphi_{d}\circ\varphi_e(h_i(x))=U(i,b(d,e),x)$, where $\varphi_{b(d,e)}=\varphi_{d}\circ\varphi_e$, so $U$ is computably transparent.
However, if $h_0$ and $h_1$ are sufficiently different, $U$ does not preserve binary product:
Given an $h_0$-name of $f$ and an $h_1$-name of $x$ one may find an $h_0\oplus h_1$-name of $f(x)$, but in general, $h_0(p)\oplus h_1(p)\leq_Th_i(p)$ does not hold for $i\in\{0,1\}$.
\end{example}

On the other hand, it is known that an applicative morphism $\gamma$ preserves all finite limits (see \cite[Proposition 2.2.2]{LoPhD95} or \cite[Theorem 1.6.2]{vOBook}); in particular:

\begin{obs}\label{obs:appl-morphism-binary-product}
An applicative morphism $\gamma$ from $\mathcal A$ to $\mathcal B$ preserves binary product; that is, $\gamma(X)\times\gamma(Y)\simeq\gamma(X\times Y)$ for any $\mathcal A$-multirepresented spaces $X$ and $Y$.
\end{obs}

\begin{proof}
Let ${\tt p}$ be a code of a paring map in $\mathcal A$.
This can be made by combining ${\tt s}$ and ${\tt k}$, so it has a computable $\gamma$-name.
As $\gamma$ is applicative morphism, given $\gamma$-names ${\tt p}'$ and ${\tt x}'$ of ${\tt p}$ and ${\tt x}\in A$, one can effectively find a $\gamma$-name ${\tt q}$ of ${\tt p}\ast_A{\tt x}$.
Again by the definition of applicative morphism, given a $\gamma$-name ${\tt y}$, one can effectively find a $\gamma$-name of ${\tt p}\ast_A{\tt x}\ast_A{\tt y}$ using the information of ${\tt q}$ and ${\tt y}$.
This means that a $\gamma$-name of the pair ${\tt p}\ast_A{\tt x}\ast_A{\tt y}$ is computable from the pair of $\gamma$-names of ${\tt x}$ and ${\tt y}$ in a uniform manner.
It is straightforward to see that this ensures $\gamma$ to preserve binary product.
\end{proof}

Hence, a computably transparent map is not necessarily an applicative morphism.
This may be seen as an answer to Streicher's suggestion.

\begin{cor}
There exists a computably transparent map $U\pcolon\tpbf{N}\tto\tpbf{N}$ such that $\gamma_U$ is not an applicative morphism from $\tpbf{N}$ to $\tpbf{N}$.
\end{cor}

\begin{proof}
By Example \ref{exa:comp-transparent-not-parallel} and Observation \ref{obs:appl-morphism-binary-product}.
\end{proof}

Although they are not precisely corresponding notions, there are some known results that suggest a relationship between applicative morphisms and oracles.

The concept of applicative morphism allows us to explore the relationships among various PCAs, rather than just discussing one PCA.
An attempt to capture the relationship among several PCAs in the context of oracle computation was made, for example, by Longley \cite{LoPhD95}, van Oosten \cite{vO06}, and Golov-Terwijn \cite{GoTe22}.
For example, Longley \cite{LoPhD95} showed that $X$ is Turing reducible to $Y$ if and only if a decidable (projective and modest) applicative morphism from $K_1^X$ to $K_1^Y$ exists if and only if an applicative inclusion (\cite[Definition 2.5.2]{LoPhD95}) from $K_1^Y$ into $K_1^X$ exists.
As an applicative inclusion induces a geometric inclusion between the realizability toposes, this corresponds to the embedding of the Turing degrees into the Lawvere-Tierney topologies.
Not only this, Longley \cite[Chapter 3]{LoPhD95} also examines the relationships among various other PCAs using the notion of an applicative morphism.

%Later, a similar study was conducted by Golov-Terwijn \cite{GoTe22}:
%They showed, for example, that $X$ is Turing reducible to $Y$ if and only if a magma-embedding from $K_1^X$ into $K_1^Y$ exists if and only if a magma-embedding from $K_2^X$ into $K_2^Y$ exists, where a magma-embedding from $A$ into $B$ is an injection $\gamma\colon A\to B$ such that $\gamma(x\ast_Ay)=\gamma(x)\ast_B\gamma(y)$ whenever $x\ast_Ay$ is defined.
%Note that a magma-embedding is exactly a projective modest applicative morphism $\gamma$ such that ${\rm eval}_\gamma$ is realized by the identity map.

As above, certain applicative morphisms can be associated with oracle-relativizations.
The important point here is that applicative morphisms can be used to compare completely different kinds of PCAs.
Therefore, if we could find a way to always understand a certain kind of applicative morphism (or its variant) as an oracle-relativization, we would be able to greatly expand the notion of oracle-relativization.
As a first step toward this, let us present the question of whether it is possible to extend the notion of a computably transparent map, an abstraction of oracle computability, to one between arbitrary PCAs.

\begin{question}
Is there a natural notion of a morphism between relative PCAs that is identical to a computably transparent map when restricted to endomorphisms?
% those whose domain and codomain are the same?
\end{question}

%%%%%%%%%%%%%%%%%%%%%%%%%%%%%%%%%%%%%%%%%%%%%%%%%%%%%%%%%%%%
%%%%%%%%%%%%%%%%%%%%%%%%%%%%%%%%%%%%%%%%%%%%%%%%%%%%%%%%%%%%
%%%%%%%%%%%%%%%%%%%%%%%%%%%%%%%%%%%%%%%%%%%%%%%%%%%%%%%%%%%%
%%%%%%%%%%%%%%%%%%%%%%%%%%%%%%%%%%%%%%%%%%%%%%%%%%%%%%%%%%%%
%%%%%%%%%%%%%%%%%%%%%%%%%%%%%%%%%%%%%%%%%%%%%%%%%%%%%%%%%%%%
%%%%%%%%%%%%%%%%%%%%%%%%%%%%%%%%%%%%%%%%%%%%%%%%%%%%%%%%%%%%
%%%%%%%%%%%%%%%%%%%%%%%%%%%%%%%%%%%%%%%%%%%%%%%%%%%%%%%%%%%%

\section{Oracles on multi-represented spaces}\label{sec:mmmap}

\subsection{Predicate and reducibility}\label{sec:secret-input}

%When dealing with a realizability topos, various $\Omega$-valued predicates play important roles.
Now, our realizability world follows $\Omega$-valued logic, and as a result, each predicate turns into an $\Omega$-valued map.
Formally, a {\em realizability predicate} (or simply a {\em predicate}) on a multi-represented space $X$ is a map $\varphi\colon X\to\Omega$.
Note that, since $\Omega=\mathcal{P}(\tpbf{N})$, a predicate can also be viewed as a multimap $\varphi\colon X\tto\tpbf{N}$ on a multi-represented domain.
Bauer \cite{Bau21} introduced a reducibility notion on predicates, which can be viewed as Weihrauch reducibility on multi-represented spaces.

\begin{remark}[Predicates as problems]
To explain the background of this terminology, in general, a predicate $\varphi(x^X)$ is interpreted as a subobject $\eval{x^X:\varphi(x)}\mono X$.
By Observation \ref{obs:subobject-witnessed-subset}, a subobject of $X$ is just a witnessed subset.
For each $x\in X$, let $\chi_A(x)$ be the set of all witnesses for $x\in A$, where $\chi_A(x)=\emptyset$ for $x\in X\setminus A$.
This is a map $\chi_A\colon X\to\Omega$, and hence a multimap $X\tto\tpbf{N}$.
Remembering that a multimap is just a search problem, this means that a predicate $\varphi(x^X)$ is regarded as a witness-search problem.

This may be viewed as a rigorous mathematical justification of Kolmogorov's ``{\em propositions as problems}'' interpretation.
\end{remark}

So far we have only discussed Weihrauch reducibility for multimaps on $\tpbf{N}$.
However, in computable analysis, the notion of Weihrauch reducibility is usually defined for functions on representated spaces \cite{CCA}:
For represented spaces $X,Y,X',Y'$, a multimap $f\pcolon X\tto Y$ is {\em Weihrauch reducible to $g\pcolon X'\tto Y'$} if there exist partial computable functions $\varphi_-,\varphi_+$ on $\tpbf{N}$ such that
\begin{itemize}
\item the inner reduction $\varphi_-$ transforms a given name ${\tt x}$ of any instance $x\in{\rm dom}(f)$ into a name $\varphi_-({\tt x})$ of some instance $\tilde{x}\in{\rm dom}(g)$,
\item and the outer reduction $\varphi_+$ transforms a given name ${\tt y}$ of any solution $y\in g(\tilde{x})$ into a name $\varphi_+({\tt x},{\tt y})$ of some solution $\tilde{y}\in f(x)$.
\end{itemize}

Since the definition is a bit complicated, let us rewrite the above definition using symbols.
For a (multi-)represented space $X$, we write ${\tt x}\vdash_Xx$ if ${\tt x}$ is a name of $x$.
Then:
\begin{definition}
Let $X,Y,X',Y'$ be multi-represented spaces.
For $f\pcolon X\tto Y$ and $g\pcolon X'\tto Y'$, we say that $f$ is {\em Weihrauch reducible to $g$} if there exist partial computable functions $\varphi_-,\varphi_+$ on $\tpbf{N}$ such that
\[
\begin{cases}
(\forall {\tt x},x)\;\big[{\tt x}\vdash_Xx\in{\rm dom}(f)\implies(\exists\tilde{x})\;\big(\varphi_-({\tt x})\vdash_{X'}\tilde{x}\in{\rm dom}(g),\\
(\forall {\tt y},y)\;\;[{\tt y}\vdash_{Y'}y\in g(\tilde{x})\implies(\exists\tilde{y})\;\varphi_+({\tt x},{\tt y})\vdash_Y\tilde{y}\in f(x)]\big)\big].
\end{cases}
\]
\end{definition}

If $X$ and $Y$ are represented spaces, $\tilde{x}$ and $\tilde{y}$ are uniquely determined.
However, in the case of multi-represented spaces\footnote{Weihrauch reducibility on multi-represented spaces has also recently been introduced independently by Matthias Schr\"oder. However, his definition (presented in CCA 2022) seems to differ from ours.}, these are not uniquely determined, so we need to be careful about the range enclosed by the existential quantifications.
Specifically, $\tilde{x}$ in the second line is assumed to be bounded by the existential quantifier in the first line.
In other words, this reduction also involves a map $\str\colon ({\tt x},x)\mapsto \tilde{x}$.
% and $\str_+\colon ({\tt x},x,{\tt y},y)\mapsto \tilde{y}$.

%\begin{remark}
%This notion may be treated in a more unified manner by using multivalued reductions.
%A multimap $h\pcolon X\tto Y$ is said to be {\em computable} if there exists a computable function ${\tt h}$ such that for any name ${\tt x}$ of $x\in{\rm dom}(h)$, ${\tt h}({\tt x})$ is a name of some $y\in h(x)$.
%Then, $f$ is Weihrauch reducible to $g$ if and only if there exist computable multimaps $\varphi_-\pcolon X\tto X'$ and $\varphi_+\pcolon X\times Y'\tto Y$ such that
%\[y\in g[\varphi_-(x)]\implies \varphi_+(x,y)\subseteq f(x).\]
%\end{remark}

\begin{remark}[Kihara's definition \cite{Kih21}]
Note that the definition of Weihrauch reducibility on multi-represented spaces can also be viewed as an imperfect information three-player game, as introduced in \cite{Kih21}:
\[
\begin{array}{|rccc|}\hline 
	& \Me & \Ar & \Ni \\[0.5em]
\texttt{1}\colon	& {\tt x}\vdash_X {x}\in{\rm dom}(f)	& & \\
\texttt{2}\colon	&				& {\tilde{\tt x}} \vdash_{X'} & {\tilde{x}}\in{\rm dom}(g)	\\
\texttt{3}\colon & {\tt y}\vdash_{Y'} {y}\in g({\tilde{x}})	& & \\	 
\texttt{4}\colon & 		& {\tilde{\tt y}}  \vdash_{Y} &{\tilde{y}}\in f({x}) \\[0.5em] \hline
\end{array}
\]

Here, ${\tt x},\tilde{\tt x},{\tt y},\tilde{\tt y}$ are public moves and open to all players.
On the other hand, $x,\tilde{x},y,\tilde{y}$ are secret moves, visible only to \Me and \Nim.
Also, computability is imposed only on \Art's moves; that is, $\varphi_-\colon{\tt x}\mapsto\tilde{\tt x}$ and $\varphi_+\colon({\tt x},{\tt y})\mapsto\tilde{\tt y}$ are computable.
Then $f\leq_Wg$ if \Art-\Ni have a winning strategy for the above game.
For the details, see Definition \ref{def:game-LT-reducibility} and also \cite{Kih21}.
\end{remark}

\begin{remark}[Bauer's defintion \cite{Bau21}]
As noted by Bauer \cite{Bau21}, it is the realizability interpretation of the statement $\forall x\in{\rm dom}(f)\exists\tilde{x}\in{\rm dom}(g).\;g(\tilde{x})\arr f(x)$.
\end{remark}

In this way, the definition of Weihrauch reducibility has been extended.
Of course, when broadening the framework, it is obligatory to provide a number of concrete examples that could not be dealt with in the previous framework.
Some basic examples are given below.

\begin{example}\label{exa:double-negation-space1}
The {\em double negation representation} of a set $X$ is a multi-representation of $X$ such that any ${\tt a}\in\tpbf{N}$ is a name of any element of $X$.
The double negation space $\neg\neg X$ is the set $X$ equipped with the double negation representation; that is, $p\vdash_{\neg\neg X}x$ always holds whenever $p\in\tpbf{N}$ and $x\in X$.
Note that $\neg\neg X$ is usually denoted as $\nabla X$ in realizability theory; see e.g.~\cite{vOBook}.

The {\em double negation elimination on the natural numbers}, ${\sf DNE}_\N\colon\neg\neg\N\to\N$ defined by ${\sf DNE}_\N(n)=n$.
If $(\tplf{N},\tpbf{N},\ast)$ is the Kleene-Vesley algebra:

\begin{claim}
A partial multimap $f\pcolon\tpbf{N}\tto\tpbf{N}$ is Weihrauch reducible to ${\sf DNE}_\N$ if and only if $f$ is non-uniformly computable; that is, for any $x\in{\rm dom}(f)$ there exists $y\leq_Tx$ such that $y\in f(x)$.
\end{claim}
%See also \cite{Kih21} with Corollary \ref{cor:eWeih-equ-W-mrep-sp}).

\begin{proof}
($\Rightarrow$)
Since $\tpbf{N}=\N^\N$ and a name of $x\in\N^\N$ is $x$ itself, we have $y=\varphi_+(e,x)$ for some $e$, but since $\varphi_+$ is computable and $e\in\N$, we get $y\leq_Tx$.

($\Leftarrow$)
By the assumption, $y=\varphi_{e(x)}(x)$ for some $e(x)$.
For a public reduction, $\varphi_-(x)$ can be any value, and $\varphi_+(e,x)=\varphi_e(x)$.
By the definition of $\neg\neg\N$, we have $\varphi_-(x)\vdash_{\neg\neg\N}e$ for any $e$.
Therefore, a secret reduction can be given by $\str(x)=e_x$.
\end{proof}
\end{example}

\begin{example}\label{exa:double-negation-space2}
Let ${\rm Meas}_>0$ be the set of all subsets of $2^\N$ of positive measure, and consider the nonuniform positive measure choice ${\sf PC}\colon \neg\neg {\rm Meas}_{>0}\tto 2^\N$ defined by ${\sf PC}(A)=A$; that is, any element $r\in A$ is a solution to ${\sf PC}(A)$.
If $(\tplf{N},\tpbf{N},\ast)$ is the Kleene-Vesley algebra, then a partial multimap $f\pcolon\tpbf{N}\tto\tpbf{N}$ is Weihrauch reducible to ${\sf PC}$ if and only if $f$ is computable with random advice in the sense of Brattka-Pauly \cite{BrPa10}.
\end{example}

To analyze the structure of Weihrauch degrees on multi-represented spaces, it is sufficient to consider only Weihrauch degrees of predicates, where we consider $\tpbf{N}$ to be the space represented by the identity map.

\begin{obs}\label{obs:W-multi-predicate}
Every multimap on multi-represented spaces is Weihrauch equivalent to a predicate.
\end{obs}

\begin{proof}
Given a multimap $f$, let $\underline{f}(x)$ be the set of all names of all solutions of the $x$-th instance of problem $f$; that is, $\underline{f}(x)=\{{\tt y}\in\tpbf{N}:(\exists y)\;{\tt y}\vdash_Y y\in f(x)\}$.
Obviously, $\underline{f}$ is a predicate.
Using this notation, the second line of the definition of Weihrauch reducibility can be rewritten as follows:
\[{\tt y}\in\underline{g}(\tilde{x})\implies\varphi_+({\tt x},{\tt y})\in\underline{f}(x).\]

Hence, the Weihrauch degree of $f$ only depends on its predicate part $\underline{f}$.
\end{proof}

Now note that, when the relation ${\tt x}\vdash_Xx$ holds for a multi-represented space $X$, the name ${\tt x}$ is {\em public} information that can be accessed during a computation and the point $x$ is {\em secret} information that cannot be accessed during a computation.
Often, a multimap $f\pcolon X\tto Y$ is easier to handle if it is thought of as $\hat{f}\pcolon\tpbf{N}\times X\tto Y$ defined by $\hat{f}({\tt x},x)=f(x)$ for ${\tt x}\vdash_Xx$ (sometimes $\hat{f}({\tt x},x)$ is written as $\hat{f}({\tt x}\vdash_Xx)$ or $\hat{f}({\tt x}\mid x)$).
In particular, a predicate $\varphi\colon X\tto\tpbf{N}$ can be viewed as a function $\hat{\varphi}\pcolon \tpbf{N}\times X\tto\tpbf{N}$.
We call $\hat{\varphi}$ the {\em extension} of $\varphi$.
In Kihara \cite{Kih21}, such a function has been called a bilayered function.
Essentially the same notion has been called an extended Weihrauch predicate by Bauer \cite{Bau21}, which sounds better, so we adopt Bauer's terminology in this article.
%Since Bauer's terminology seems better, we adopt it in this article.
%In this article, however, we take the position that this is a presentation of a {\bf m}ultimap with a {\bf m}ulti-represented domain (i.e., a realizability predicate) and call it an extended predicate.

\begin{definition}
An {\em extended predicate} is a partial multimap $f\pcolon\tpbf{N}\times\Lambda\tto\tpbf{N}$, where $\Lambda$ is a set.
We write $f(n\mid p)$ for $f(n,p)$, and $n$ is called a {\em public input} and $p$ a {\em secret input}.
\end{definition}

\begin{example}
Many natural examples of extended predicates have been studied in \cite{Kih21}.
For example, various computational notions of the type of allowing for minute errors can be expressed as computations relative to some extended predicate oracles.
Examples include computability with error probability less than $\ep$ and error density (in the sense of the lower asymptotic density) less than $\ep$.
However, most extended predicates treated in \cite{Kih21} can also be described using double negation spaces, as in Examples \ref{exa:double-negation-space1} and \ref{exa:double-negation-space2}.
\end{example}

\begin{definition}[Bauer \cite{Bau21}; see also Kihara \cite{Kih21}]\label{def:extended-Weihrauch}
Let $f$ and $g$ be extended predicates.
Then $f$ is {\em extended Weihrauch reducible to} $g$ (written $f\leq_{eW}g$) if
there exist partial computable functions $\varphi_-,\varphi_+$ on $\tpbf{N}$ such that
\[
\begin{cases}
(\forall {\tt x},x)\;\big[({\tt x}\mid x)\in{\rm dom}(f)\implies(\exists\tilde{x})\;\big((\varphi_-({\tt x})\mid\tilde{x})\in{\rm dom}(g),\\
(\forall {\tt y})\;\;[{\tt y}\in g(\varphi_-({\tt x})\mid\tilde{x})\implies \varphi_+({\tt x},{\tt y})\in f({\tt x}\mid x)]\big)\big].
\end{cases}
\]

Note that $\tilde{x}$ in the second line is bounded by the existential quantifier in the first line.
In other words, it involves a map $\str\colon({\tt x},x)\mapsto \tilde{x}$.
\end{definition}

As noted above, a predicate (on a multi-represented space) can be regarded as an extended predicate.
Obviously, restricting extended Weihrauch reducibility to extensions of predicates is equivalent to restricting Weihrauch reducibility on multi-represented spaces to predicates.
Bauer \cite{Bau21} has shown that the extended Weihrauch degrees form a Heyting algebra.

%\begin{remark}
%Note that the definition of extended Weihrauch reducibility can be viewed as an imperfect information three-player game, as introduced in \cite{Kih21}:
%\[
%\begin{array}{rcccccc}
%\Me\colon	& ({\tt x}\mid x)\in{\rm dom}(f)	&		&		& {\tt y}\in g(\tilde{\tt x}\mid\tilde{x})	&		&		 \\
%\Ar\colon	&				& \tilde{\tt x}	&		&		& \tilde{\tt y}\in f({\tt x}\mid x)	& 		 \\
%\Ni\colon		&				&		& \tilde{x}\ \mbox{ s.t. }(\tilde{\tt x}\mid\tilde{x})\in{\rm dom}(g)	&		&		& 
%\end{array}
%\]
%
%Here, ${\tt x},\tilde{\tt x},{\tt y},\tilde{\tt y}$ are public moves and open to all players.
%On the other hand, $x,\tilde{x}$ are secret moves, visible only to \Me and \Nim.
%Also, computability is imposed only on \Art's moves; that is, $\varphi_-\colon{\tt x}\mapsto\tilde{\tt x}$ and $\varphi_+\colon({\tt x},{\tt y})\mapsto\tilde{\tt y}$ are computable.
%For the details, see \cite{Kih21}.
%\end{remark}

Note also that obviously Definition \ref{def:extended-Weihrauch} depends only on the set $\{g(a\mid p):p\in\Lambda\}$.
For extended predicates $f\pcolon\tpbf{N}\times\Lambda\tto\tpbf{N}$ and $g\pcolon\tpbf{N}\times\Lambda'\tto\tpbf{N}$, we say that {\em $f$ is equivalent to $g$} if $\{f(x\mid p):p\in\Lambda\}=\{g(x\mid q):q\in\Lambda'\}$ for any $x\in\tpbf{N}$.
An extended predicate $f$ is {\em canonical} if $f(x\mid p)=p$ for any $(x\mid p)\in{\rm dom}(f)$.

Note that a canonical extended predicate is always an extension of a predicate.
Indeed, a canonical extended predicate (seen as a predicate) is a map $i\colon(\Omega,\delta)\to\Omega$ (where $\delta$ is a multi-representation) that is identity on the underlying set $\Omega$; that is, $i(p)=p$.
In other words, it is a subobject of the trivial assembly $\Omega$.

\begin{obs}[see also \cite{LvO,Bau21}]\label{obs:bimap-canonical}
Every extended predicate is equivalent to a canonical one.
\end{obs}

\begin{proof}
This is because given an extended predicate $g$ define $\tilde{g}\pcolon\tpbf{N}\times\Omega\tto\tpbf{N}$ by $\tilde{g}(x\mid y)=y$ if $g(x\mid p)=y$ for some $(x\mid p)\in{\rm dom}(g)$; otherwise $\tilde{g}(x\mid y)$ is undefined.
Then $\tilde{g}$ is clearly equivalent to $g$.
\end{proof}

%As a consequence of Observation \ref{obs:bimap-canonical}, 
In partiuclar, every extended predicate is Weihrauch equivalent to an extension of a predicate.
What this means is that the extended Weihrauch degrees are isomorphic to the (extended) Weihrauch degrees of predicates (see also \cite{Bau21}), and thus, by Observation \ref{obs:W-multi-predicate}, we get the following:

\begin{cor}\label{cor:eWeih-equ-W-mrep-sp}
The Heyting algebra of the extended Weihrauch degrees is isomorphic to the Weihrauch degrees on multi-represented spaces.
\qed
\end{cor}

\begin{remark}
Under this isomorphism, Bauer \cite{Bau21} named the ones in the extended Weihrauch degrees corresponding to the Weihrauch degrees on represented spaces (i.e., the ordinary Weihrauch degrees) the {\em modest} degrees.
In other words, the modest extended Weihrauch degrees are the ordinary Weihrauch degrees.
This terminology is derived from the fact that a represented space is called a modest set in realizability theory.
\end{remark}

In proving various properties, it is often easier to discuss them if they are transformed into a canonical form via Observation \ref{obs:bimap-canonical}.
However, natural concrete examples often appear in non-canonical form, and it is intuitively easier to understand them if non-canonical forms are also allowed; see also \cite{Kih21}.

\begin{remark}
Several variants of Weihrauch reducibility have been studied in recent years, but they can be treated in a unified manner by using Weihrauch reducibility on multi-represented spaces (or extended predicates).
To see this, for a set $X$, consider the double negation elimination ${\sf DNE}_X\colon\neg\neg X\to X$ defined by ${\sf DNE}_X(x)=x$.
Then, for any partial multimaps $f\pcolon X\tto Y$ and $g\pcolon Z\tto W$, we have the following:
\begin{itemize}
\item $f$ is computably reducible to $g$ in the sense of Dzhafarov \cite{Dzh16} if and only if $f\leq_W{\sf DNE}_{\N}\star g\star{\sf DNE}_{\N}$.
\item $f$ is omnisciently Weihrauch reducible to $g$ in the sense of Dzhafarov-Patey \cite{DzPa20} if and only if $f\leq_W g\circ{\sf DNE}_{Z}$.
\item $f$ is omnisciently computably reducible to $g$ in the sense of Monin-Patey \cite{MoPa19} if and only if $f\leq_W {\sf DNE}_{\N}\star g\circ{\sf DNE}_{Z}$.
\end{itemize}

Here, $\star$ denotes the compositional product (see \cite[Theorem 5.2 and Definition 5.3]{CCA}).
\end{remark}

Next, let us consider the composition of two extended predicates.
For $f\colon X\tto Y$ and $g\colon Y\tto Z$, the composition $g\circ f\colon X\tto Z$ is defined as usual (see Section \ref{sec:notations}).
However, if $f$ and $g$ are presented as extended predicates $\theta_f,\theta_g\pcolon\tpbf{N}\times\Lambda\tto \tpbf{N}$, then the composition $\theta_g\circ\theta_f$ would not be defined, of course.

%Given extended predicates $\alpha$ and $\beta$, we say that $\alpha$ is extended many-one reducible to $\beta$ (written $\alpha\leq_{em}\beta$) if there exists a computable function $\varphi$ such that for any $(x\mid p)\in{\rm dom}(\alpha)$ there exists $q$ such that $(\varphi(x)\mid q)\in{\rm dom}(\beta)$ and $\beta(\varphi(x)\mid q)\subseteq \alpha(x\mid p)$.
%
%\begin{prop}
%There exists a binary operation $\diamond$ on extended predicates such that $\theta_{g\circ f}\equiv_{em}\theta_g\diamond\theta_f$ for any $f\colon X\tto Y$ and $g\colon Y\tto Z$.
%\end{prop}

\begin{discussion}
To overcome this difficulty, we introduce a composite-like operation for extended predicates.
It almost adopts the definition of the composition of multifunctions on multi-represented spaces.
Recall that for a given $f\colon X\tto Y$, we first consider its extension $\hat{f}\colon\tpbf{N}\times X\tto Y$, and then take the set of names of solutions, i.e.,  $\underline{\hat{f}}\colon\tpbf{N}\times X\tto\tpbf{N}$, which is the corresponding extended predicate $\theta_f$.
One can see the following:
\[{\tt z}\in \underline{\widehat{g\circ f}}({\tt x}\mid x)\iff(\exists {\tt y},y)\;[{\tt y}\in \underline{\hat{f}}({\tt x}\mid x)\mbox{ and }{\tt z}\in \underline{\hat{g}}({\tt y}\mid y)].\]

In represented spaces, ${\tt y}\mapsto y$ is uniquely determined, but this is not the case in multi-represented spaces.
For this reason, let us always specify a map $\str\colon{\tt y}\mapsto y$ as a secret input (a nonuniform advice).
Then, we declare that $({\tt x}\mid x,\str)$ is in the domain of the composition as long as the process always terminates along this map $\str$.
This discussion leads to the following definition:

\begin{definition}
Let $\alpha\pcolon\tpbf{N}\times\Lambda\tto\tpbf{N}$ and $\beta\pcolon\tpbf{N}\times\Lambda'\tto\tpbf{N}$ be extended predicates.
Define $\beta\circ\alpha\pcolon\tpbf{N}\times(\Lambda\times[{\tpbf{N}}\to\Lambda'])\tto\tpbf{N}$ as follows:
For the domain, $(\beta\circ\alpha)({\tt x}\mid p,\str)$ is defined if $\alpha({\tt x}\mid p)$ is defined and  $\beta({\tt y}\mid \str({\tt y}))$ is defined for any ${\tt y}\in \alpha({\tt x}\mid p)$.
In such a case, 
\[{\tt z}\in(\beta\circ \alpha)({\tt x}\mid p,\str)\iff (\exists{\tt y})\;[{\tt y}\in \alpha({\tt x}\mid p)\mbox{ and }{\tt z}\in\beta({\tt y}\mid \str({\tt y}))].\]
\end{definition}

The definition of the composition is complicated, but the idea comes from the notion of non-uniform advice, or \Nim's strategy in games ending in two rounds \cite{Kih21}.
That is, $({\tt x}\mid p)$ are first moves of \Ar and \Nim, respectively, and $\str$ is \Nim's strategy to decide her next move according to \Mer's reply.
\end{discussion}

\subsection{Universal extended predicate}\label{sec:universal-bimap}
Now, let us generalize the properties for multimaps (Observation \ref{obs:transparent-multmap}) to extended predicates.

\begin{definition}\label{def:transparency-for-mmmap}
Let $U\pcolon\tpbf{N}\times\Lambda\tto\tpbf{N}$ be an extended predicate.
\begin{itemize}
\item $U$ is {\em computably transparent} if and only if
there exists ${\tt u}\in \tplf{N}$ such that for all ${\tt f},{\tt x}\in \tpbf{N}$ and $p\in\Lambda$, whenever ${\tt f}\ast U({\tt x}\mid p)$ is defined,
\[(\exists q\in\Lambda)\quad ({\tt u\ast f\ast x}\mid q)\in{\rm dom}(U)\mbox{ and }U({\tt u\ast f\ast x}\mid q)\subseteq{\tt f}\ast U({\tt x}\mid p).\]

\item $U$ is {\em inflationary} if and only if
there exists $\eta\in \tplf{N}$ such that for all ${\tt x}\in \tpbf{N}$,
\[(\exists p\in\Lambda)\quad(\eta\ast {\tt x}\mid p)\in{\rm dom}(U)\mbox{ and }U(\eta\ast {\tt x}\mid p)\subseteq\{{\tt x}\}.\]

\item $U$ is {\em idempotent} if and only if there exists $\mu\in \tplf{N}$ such that for all ${\tt x}\in{\rm dom}(U\circ U)$ and $p\in\Lambda$,
\[(\exists q\in\Lambda)\quad (\mu\ast{\tt x}\mid q)\in{\rm dom}(U)\mbox{ and }U(\mu\ast {\tt x}\mid q)\subseteq U\circ U({\tt x}\mid p).\]
\end{itemize}
\end{definition}

\begin{remark}
By Corollary \ref{cor:eWeih-equ-W-mrep-sp}, it seems more natural to deal directly with a multimap on multi-represented spaces rather than an extended predicate.
The reason why we still use an extended predicate instead of a multimap on multi-represented spaces is Definition \ref{def:transparency-for-mmmap}.
For example, the definition of idempotence uses self-composition $U\circ U$, which is always defined for an extended predicate, but is generally meaningless for a multimap on multi-represented spaces (since the domain and the codomain might be different).
\end{remark}

The following is the extended predicate version of Definition \ref{def:m-morphi-r-morphi}.

\begin{definition}\label{def:extended-many-one-reducibility}
For extended predicates $f$ and $g$, we say that a computable element $e\in N$ is an {\em extended $m$-morphism from $f$ to $g$} if for any $(x\mid p)\in{\rm dom}(f)$ there exists $q$ such that $(\varphi_e(x)\mid q)\in{\rm dom}(g)$ and $g(\varphi_e(x)\mid q)\subseteq f(x\mid p)$.
We also say that {\em $f$ is extended $m$-reducible to $g$} (written $f\leq_{em}g$) if there exists an extended $m$-morphism from $f$ to $g$.
If $U$ is computably transparent, we often say that $f$ is $U$-computable if $f$ is extended $m$-reducible to $U$.
\end{definition}

One can also think of extended $m$-reducibility for extended predicates as $m$-reducibility for predicates on multi-represented spaces.

\begin{example}[Extended Weihrauch oracle]\label{exa:oracle-Weihrauch-reducibility-bimap}
For an extended predicate $g$, the universal computation relative to the extended Weihrauch oracle $g$, ${\tt EWeih}_g$, is defined as the following extended predicate:
\begin{align*}
{\rm dom}({\tt EWeih}_g)&=\{(h,k,x\mid p,\str):(\varphi_h(x)\mid\str(x,p))\in{\rm dom}(g)\\
&\qquad\mbox{ and }\varphi_k(x,y)\downarrow\mbox{ for all $y\in g(\varphi_h(x)\mid\str(x,p))$}\}\\
{\tt EWeih}_g(h,k,x\mid p,\str)&=\{\varphi_k(x,y):y\in g(\varphi_h(x)\mid\str(x,p))\}
\end{align*}

One can think of ${\tt EWeih}_g$ as an extended predicate that can simulate all the plays in the extended Weihrauch reduction game mentioned in Remark above, where $(x\mid p)$ is \Mer's first move, $(h,k)$ is \Art's strategy, and $\str$ is \Nim's strategy.

Note that ${\tt EWeih}_g$ is computably transparent.
This is because we have $\varphi_\ell[{\tt EWeih}_g(h,k,x\mid p,\str)]=\{\varphi_\ell\circ\varphi_k(x,y):y\in g(\varphi_h(x)\mid\str(x,p))\}={\tt EWeih}_g(h,b_{\ell,k},x\mid p,\str)$, where $b_{\ell,k}$ is a code of $\varphi_\ell\circ\varphi_k$.
%This is because we have $\varphi_\ell[{\tt Weih}_g(h,k,x)]=\{\varphi_\ell\circ\varphi_k(x,y):y\in g(\varphi_h(x))\}={\tt Weih}_g(h,b_{\ell,k},x)$.

Moreover, an extended predicate $f$ is ${\tt EWeih}_g$-computable if and only if $f\leq_{eW}g$.
This is because $f\leq_{eW}g$ via $h,k,\str$ if and only if $(h,k)$ and $\str$ are \Ar and \Nim's winning strategies for the reduction game in the above Remark if and only if ${\tt EWeih}_g(h,k,x\mid p,\str)\subseteq f(x\mid p)$, so $f$ is ${\tt EWeih}$-computable.
The converse direction can be shown by an argument similar to Example \ref{exa:oracle-Weihrauch-reducibility}
%Conversely, if $f$ is ${\tt EWeih}_g$-computable via $x\mapsto (h(x),k(x),z(x))$ then $f\leq_Wg$ via $x\mapsto\varphi_{h(x)}(z(x))$ and $(x,y)\mapsto\varphi_{k(x)}(z(x),y)$.
\end{example}

We will discuss the exact relationship between computable transparency and universal computation in the context of extended predicate.

\begin{definition}
Let $ex\mathcal{MR}ed$ be the set of all extended predicates on $\tpbf{N}$ preordered by extended many-one reducibility $\leq_{em}$.
The restriction of this preordered set to those that are computably transparent, inflationary, and idempotent, respectively, is expressed by decorating it with superscripts ${\rm ct}$, $\eta$, and $\mu$, respectively.
\end{definition}

We next see an extension of Proposition \ref{prop:computably-transparent-many-one-Weih} stating that the notions of extended many-one reducibility and extended Weihrauch reducibility coincide on computably transparent extended predicates.
As we have seen in Example \ref{exa:oracle-Weihrauch-reducibility-bimap}, the construction ${\tt EWeih}\colon g\mapsto{\tt EWeih}_g$ yields a computably transparent map from a given extended predicate.
Moreover, it is easy to see that $f\leq_mg$ implies ${\tt EWeih}_f\leq_{em}{\tt EWeih}_g$.
Hence, ${\tt EWeih}$ can be viewed as an order-preserving map from $ ex\mathcal{MR}ed$ to $ ex\mathcal{MR}ed^{\rm ct}$.
The following guarantees that ${\tt EWeih}(f)$ is the $\leq_{em}$-least computably transparent map which can compute $f$.

\begin{prop}\label{thm:computably-transparent-many-one-eWeihx}
The order-preserving map ${\tt EWeih}\colon ex\mathcal{MR}ed\to ex\mathcal{MR}ed^{\rm ct}$ is left adjoint to the inclusion $i\colon ex\mathcal{MR}ed^{\rm ct}\mono ex\mathcal{MR}ed$.
In other words, for any extended predicates $f$ and $g$, if $g$ is computably transparent, then $f\leq_{em}g$ if and only if ${\tt EWeih}(f)\leq_{em}g$.
\end{prop}

\begin{proof}
Obviously, ${\tt EWeih}(f)\leq_{em}g$ implies $f\leq_{em}g$.
For the converse direction, if $f\leq_{em}g$ then ${\tt EWeih}(f)\leq_{em}{\tt EWeih}(g)$ by monotonicity; hence it suffices to show that ${\tt EWeih}(g)\leq_{em}g$.
Given ${\tt h},{\tt k},{\tt x}\in \tpbf{N}$ and $p,\str$, note that ${\tt y}\in g({\tt h}\ast{\tt x}\mid\str({\tt x},p))$ implies ${\tt k}\ast\langle{\tt x},{\tt y}\rangle\in {\tt EWeih}_g({\tt h},{\tt k},{\tt x}\mid p,\str)$.
Put ${\tt k'}=\lambda xy.{\tt k}\ast\langle x,y\rangle$.
As $g$ is computably transparent, there exists ${\tt u}\in \tplf{N}$ and $q$ such that $g({\tt u}\ast ({\tt k'\ast x})\ast ({\tt h\ast x})\mid q)\subseteq ({\tt k'\ast x})\ast g({\tt h\ast x}\mid \str({\tt x},p))$.
Note that if ${\tt y}\in g({\tt h\ast x}\mid\str({\tt x},p))$ then  ${\tt k'\ast x\ast y}={\tt k}\ast \langle{\tt x},{\tt y}\rangle\in {\tt EWeih}_g({\tt h},{\tt k},{\tt x}\mid p,\str)$.
Hence, we get $g({\tt u}\ast ({\tt k'\ast x})\ast ({\tt h\ast x})\mid q)\subseteq {\tt EWeih}_g({\tt h},{\tt k},{\tt x}\mid p,\str)$, so the term $\lambda hkx.{\tt u}\ast (({\tt a}\ast k)\ast x)\ast (h\ast x)\in \tplf{N}$ witnesses ${\tt EWeih}_g\leq_{em}g$, where ${\tt a}\in \tplf{N}$ is a computable element such that ${\tt a}\ast{\tt k}={\tt k}'$.
\end{proof}

This implies that the notions of extended many-one reducibility and extentded Weihrauch reducibility coincide on computably transparent extended predicates.

%For a preorder, its quotient by the induced equivalence relation is called the poset reflection.

\begin{cor}\label{prop:computably-transparent-many-one-eWeih}
%The poset reflection of $\mathcal{MR}ed^{\rm ct}$ (i.e.,
The Heyting algebra of extended many-one degrees of computably transparent extended predicates on $\tpbf{N}$ is isomorphic to the extended Weihrauch degrees.
\end{cor}

\begin{proof}
We first claim that if $g$ is computably transparent, then $f\leq_{em}g$ if and only if $f\leq_{eW}g$.
The forward direction is obvious.
For the backward direction, first note that $f\leq_{eW}g$ if and only if $f\leq_{em}{\tt EWeih}_g$ as seen in Example \ref{exa:oracle-Weihrauch-reducibility-bimap}.
By Proposition \ref{thm:computably-transparent-many-one-eWeihx}, ${\tt EWeih}_g\leq_{em}g$ since $g$ is computably transparent.
Hence, $f\leq_{em}g$.
This claim ensures that the identity map is an embedding from the poset reflection of $ ex\mathcal{MR}ed^{\rm ct}$ into the extended Weihrauch degrees.
For surjectivity, we have $f\equiv_{eW}{\tt EWeih}_f$ as seen in Example \ref{exa:oracle-Weihrauch-reducibility-bimap}, and ${\tt EWeih}_f$ is computably transparent.
\end{proof}

\begin{remark}
There are several candidates for the definition of Weihrauch reducibility on multi-represented spaces, and it is debatable which one to adopt.
Of course, all definitions seem valuable in the context for which they are appropriate, and indeed, Proposition \ref{thm:computably-transparent-many-one-eWeihx} supports the validity of our definition from at least one perspective.
The key point is that our notion of Weihrauch reducibility automatically appears when we have the notions of computable transparency and many-one reducibility as before, so in fact, there is no need to define Weihrauch reducibility:
Any left-adjoint ${\tt EW}$ of the inclusion $i\colon ex\mathcal{MR}ed^{\rm ct}\mono ex\mathcal{MR}ed$ recovers the notion of extended Weihrauch reducibility as $f\leq_{eW}g$ if and only if $f\leq_{em}{\tt EW}(g)$ since ${\tt EW}(g)\equiv_{em}{\tt EWeih}(g)$.
By Corollary \ref{cor:eWeih-equ-W-mrep-sp}, this means that our definition of Weihrauch reducibility on multi-represented spaces is automatically obtained from any ${\tt EW}\dashv i$ without explicitly defining it.
%The point is that ${\tt W}$ has no meaning a priori; nevertheless ${\tt W}$ somehow automatically recovers the notion of Weihrauch reducibility.
\end{remark}

%%%%%%%%%%%%%%%%%%%%%%%%%%%%%%%%%%%%%%%%%%%%%%%%%%%%%%%%%
%%%%%%%%%%%%%%%%%%%%%%%%%%%%%%%%%%%%%%%%%%%%%%%%%%%%%%%%%
%%%%%%%%%%%%%%%%%%%%%%%%%%%%%%%%%%%%%%%%%%%%%%%%%%%%%%%%%
%%%%%%%%%%%%%%%%%%%%%%%%%%%%%%%%%%%%%%%%%%%%%%%%%%%%%%%%%
%%%%%%%%%%%%%%%%%%%%%%%%%%%%%%%%%%%%%%%%%%%%%%%%%%%%%%%%%

One can also introduce an inflationary version of extended Weihrauch reducibility.
For extended predicates $f$ and $g$, we say that $f$ is {\em pointed extended Weihrauch reducible to $g$} (written $f\leq_{peW}g$) if $f$ is extended Weihrauch reducible to ${\rm id}\sqcup g$, where $({\rm id}\sqcup g)(\underline{\tt 0},x\mid p)=x$ and $({\rm id}\sqcup g)(\underline{\tt 1},x\mid p)=g(x\mid p)$.
It is straightforward to show analogues of Proposition \ref{prop:computably-transparent-many-one-pWeih} and Corollary \ref{thm:computably-transparent-many-one-pWeihx}.

\begin{prop}
The order-preserving map ${\tt pEWeih}\colon ex\mathcal{MR}ed\to ex\mathcal{MR}ed^{{\rm ct},\eta}$ is left adjoint to the inclusion $i\colon ex\mathcal{MR}ed^{{\rm ct},\eta}\mono ex\mathcal{MR}ed$.
In other words, for any extended predicates $f$ and $g$, if $g$ is computably transparent and inflationary, then $f\leq_{em}g$ if and only if ${\tt pEWeih}(f)\leq_{em}g$.
\qed
\end{prop}

\begin{cor}\label{cor:computably-transparent-inflationary-many-one-peWeih}
%The poset reflection of $\mathcal{MR}ed^{\rm ct}$ (i.e.,
The Heyting algebra of extended many-one degrees of computably transparent, inflationary, extended predicates on $\tpbf{N}$ is isomorphic to the pointed extended Weihrauch degrees.
\qed
\end{cor}

%%%%%%%%%%%%%%%%%%%%%%%%%%%%%%%%%%%%%%%%%%%%%%%%%%%%%%%%%
%%%%%%%%%%%%%%%%%%%%%%%%%%%%%%%%%%%%%%%%%%%%%%%%%%%%%%%%%
%%%%%%%%%%%%%%%%%%%%%%%%%%%%%%%%%%%%%%%%%%%%%%%%%%%%%%%%%
%%%%%%%%%%%%%%%%%%%%%%%%%%%%%%%%%%%%%%%%%%%%%%%%%%%%%%%%%
%%%%%%%%%%%%%%%%%%%%%%%%%%%%%%%%%%%%%%%%%%%%%%%%%%%%%%%%%

The notion of extended Weihrauch reducibility is given by a relative computation that makes exactly one query to an oracle, so it is not idempotent.
Like Turing-Weihrauch reducibility corresponding to the idempotent version of Weihrauch reducibility, the idempotent version of extended Weihrauch reducibility (i.e., extended Turing-Weihrauch reducibility) was introduced in \cite{Kih21} (almost equivalent one had been studied in \cite{LvO}; see \cite[Remark 2.15]{Kih21}) and named LT-reducibility.

The definition of LT-reducibility can be described using an imperfect information game between three players, \Mer, \Art, and \Nim.
%The player \Me makes a public input $x_0$ and a secret input $c_0$ on his first move.
%Here, among the moves of \Mer, only the secret input $c_0$ is invisible to \Art.
%All of \Nim's moves are visible to \Mer, but not to \Art, a mere human being.
%The players \Me and \Nim, who are not mere humans, can see all the previous moves at each round.

\begin{definition}[{\cite[Definition 2.13]{Kih21}}]\label{def:game-LT-reducibility}
For extended predicates $f$ and $g$, let us consider the following imperfect information three-player game $\mathfrak{G}(f,g)$:
%\[
%\begin{array}{rcccccccccc}
%\Me\colon	& (x_0\mid c_0)	&		&		& x_1	&		&		& x_2	&		& 		& \dots \\
%\Ar\colon	&				& y_0	&		&		& y_1	& 		&		& y_2	& 		& \dots \\
%\Ni\colon		&				&		& z_0	&		&		& z_1	& 		&		& z_2	& \dots
%\end{array}
%\]
\[
\begin{array}{|ccc|}\hline 
 \Me & \Ar & \Ni \\[0.5em]
 ({{\tt p}_0}\mid {x_0})\in{\rm dom}(f)	& & \\
				& {\tt Query}\colon \;({{\tt q}_0} \mid & {z_0})\in{\rm dom}(g)	\\
{{\tt p}_1}\in g({\tt q}_0\mid {z_0})	& & \\	 
				& {\tt Query}\colon \;({{\tt q}_1} \mid & {z_1})\in{\rm dom}(g)	\\
{{\tt p}_2}\in g({\tt q}_1\mid {z_1})	& & \\	 
			\vdots	& \vdots & \vdots	\\
%\vdots	& & \\	 
				& {\tt Query}\colon \;({{\tt q}_n} \mid & {z_n})\in{\rm dom}(g)	\\
{{\tt p}_{n+1}}\in g({\tt q}_n\mid {z_n})	& & \\	 
 		& {{\tt Halt}}\colon\; {{\tt q}_{n+1}}  &\in f({\tt p}_0\mid{x_0})  \\[0.5em] \hline
\end{array}
\]

%Each player chooses a natural number at each round.
%Hereafter, we use $[p]$ to denote the partial continuous function on $\N ^\N $ coded by $p$.

\noindent
{\it Game rules:}
Here, the players need to obey the following rules.
\begin{itemize}
\item First, \Me chooses a pair $({\tt p}_0\mid x_0)\in{\rm dom}(f)$.
\item At the $n$th round, \Ar reacts with $y_n=\langle {\tt A}_n,{\tt q}_n\rangle$.
\begin{itemize}
\item The choice ${\tt A}_n={\tt Query}$ (coded by $\underline{\tt 0}$) indicates that \Ar makes a new query ${\tt q}_n$ to $g$.
\item The choice ${\tt A}_n={\tt Halt}$ (coded by $\underline{\tt 1}$) indicates that \Ar declares termination of the game with ${\tt q}_n$.
\end{itemize}
\item At the $n$th round, \Ni makes an advice parameter $z_n$ such that $({\tt q}_n\mid z_n)\in{\rm dom}(g)$.
\item At the $(n+1)$st round, \Me responds to the query made by \Ar and \Ni at the previous stage.
This means that ${\tt p}_{n+1}\in g({\tt q}_n\mid z_n)$.
\end{itemize}

Then, {\em \Ar and \Ni win the game $\mathfrak{G}(f,g)$} if either \Me violates the rule before \Ar or \Ni violates the rule, or both \Ar and \Ni obey the rule and \Ar declares termination with ${\tt q}_n\in f({\tt p}_0\mid x_0)$.

\medskip

\noindent
{\it Strategies:}
As noted above, \Ar can only read the moves ${\tt p}_0,{\tt p}_1,{\tt p}_2,\dots$, and the other players can see all the moves.
In other words, \Art's strategy is a partial map $\tau\pcolon \tpbf{N}^{<\om}\to\tpbf{N}$.
\Art's strategy is always computable or continuous, coded in the lightface part $\tplf{N}$ for the former case and in $\tpbf{N}$ for the latter.
%Moreover, we require that \Art's moves are chosen in a computable manner.
%In other words, \Art's strategy is a code $\tau$ of a partial {\em computable} function $h_\tau\pcolon\N ^{<\N }\to\N$, which reads \Mer's moves $x_0,\dots,x_n$ and then returns $y_n$.
On the other hand, \Me and \Nim's strategies are any partial functions (which are not necessarily computable).

If $\sigma$, $\tau$, and $\eta$ are strategies of \Mer, \Art, and \Nim, respectively, then the play that follow these strategies are defined as follows:
\Mer's first move is $({\tt p}_0\mid x_0):=\sigma()$, and ($n+1$)th move is ${\tt p}_{n+1}:=\sigma({\tt q_0},z_0,\dots,{\tt q}_n,z_n)$.
\Art's $n$th move is $\langle{\tt A}_n,{\tt q}_n\rangle:=\tau({\tt p}_0,\dots,{\tt p}_n)$.
\Nim's $n$th move is $z_n:=\eta(x_0,{\tt p_0},{\tt q_0},\dots,{\tt p}_n,{\tt q}_n)$.
Here, ``{\em undefined}'' counts as a rule violation.

A pair $(\tau\mid \eta)$ of \Art's strategy $\tau$ and \Nim's strategy $\eta$ is called an \Art-\Ni strategy.
If \Art's strategy $\tau$ is computable (continuous, resp.)~then the pair $(\tau\mid\eta)$ is called a computable (continuous, resp.)~\Art-\Ni strategy.
An \Art-\Ni strategy $(\tau\mid\eta)$ is {\em winning} if, as long as \Ar and \Ni follow the strategy $(\tau\mid\eta)$, \Ar and \Ni win the game, no matter what \Mer's strategy $\sigma$ is.
\end{definition}

\begin{definition}[{\cite[Definition 2.14]{Kih21}}]\label{def:LT-reducibility}
Let $f$ and $g$ be extended predicates.
We say that {\em $f$ is LT-reducible to $g$} (written $f\leq_{LT} g$) if there exists a computable winning \Art-\Ni strategy for $\mathfrak{G}(f,g)$.
\end{definition}

Note that the rule of the game $\mathfrak{G}(f,g)$ does not mention $f$ except for Player I's first move.
Hence, if we skip Player I's first move, we can judge if a given play follows the rule without specifying $f$.
Such a restricted game is denoted by $\mathfrak{G}(g)$.
% {\em \Ar and \Ni win the game $\mathfrak{G}(g)$} if either \Me violates the rule before \Ar or \Ni violates the rule, or both \Ar and \Ni obey the rule and \Ar declares termination.

\begin{definition}[{\cite[Definition 2.20]{Kih21}}]
Given an extended predicate $h$, let us define the new extended predicate $h^\Game$ as follows:
An input for $h^\Game$ is a continuous \Art-\Ni strategy $(\tau\mid\eta)$, where \Art's strategy $\tau$ is a public input, and \Nim's strategy $\eta$ is a secret input.
\begin{itemize}
\item $h^\Game(\tau\mid\eta)$ is defined only if, along any play in $\mathfrak{G}(h)$ following the strategy $(\tau\mid\eta)$, either \Me violates the rule before \Ar or \Ni violates the rule, or both \Ar and \Ni obey the rule and \Ar declares termination, whatever \Mer's strategy is.
\item $u\in h^\Game(\tau\mid\eta)$ if and only if there is a play in $\mathfrak{G}(h)$ that follows the strategy $(\tau\mid\eta)$ such that \Ar declares termination with $u$ at some round, where all players obey the rule.
\end{itemize}
\end{definition}

Note that $h^\Game$ can be thought of as the extended predicate version of the diamond operator (recall a few paragraphs before Fact \ref{fact:Westrick}), which plays a role of the universal computation relative to the extended oracle $h$.
Observing the proof of \cite[Proposition 2.22]{Kih21}, one can see that an extended predicate $f$ is $g^\Game$-computable (i.e., $f\leq_{em}g^{\Game}$) if and only if $f$ is LT-reducible to $g$.

\begin{obs}
For any extended predicate $h$, $h^\Game$ is computably transparent, inflationary, and idempotent.
\end{obs}

\begin{proof}
For computable transparency, let $f$ be a partial continuous map.
For \Art's strategy $\tau$ in $\mathfrak{G}(h)$, if $\tau(\alpha)=\langle{\tt Halt},y\rangle$ then put $\tau^f(\alpha)=\langle{\tt Halt},f(y)\rangle$, and if $\tau(\alpha)=\langle{\tt Query},y\rangle$ then put $\tau^f(\alpha)=\tau(\alpha)$.
Then we have $f\circ h^\Game(\tau\mid\eta)=h^\Game(\tau^f\mid\eta)$, and the transformation $\tau\mapsto\tau^f$ is computable.
To see that $h^\Game$ is inflationary, consider \Art's strategy $i_x()=\langle{\tt Halt},x\rangle$.
Then $h^\Game(i_x\mid\eta)=x$, and $x\mapsto i_x$ is computable.

For idempotence, given \Art's strategy $\tau$ in $\mathfrak{G}(h)$, define the new strategy $\tau^\star$ in $\mathfrak{G}(h)$ as follows:
Given a sequence ${\tt p}_1,\dots,{\tt p}_n$, if there exists $i\leq n$ such that $\tau({\tt p}_1,\dots,{\tt p}_i)=\langle{\tt halt},\tau'\rangle$, define $\tau^\star({\tt p}_1,\dots,{\tt p}_n)=\tau'({\tt p}_{i+1},\dots,{\tt p}_n)$ for the least such $i$.
If there is no such $i$, put $\tau^\star({\tt p}_1,\dots,{\tt p}_n)=\tau({\tt p}_1,\dots,{\tt p}_n)$.
If $\tau'\in h^\Game(\tau\mid\eta)$ and $y\in h^\Game(\tau'\mid\eta')$ then, by letting $\zeta^\star$ be the concatenation of $\eta$ and $\eta'$ as above, we get $y\in h^\Game(\tau^\star\mid\eta^\star)$.
\end{proof}

In particular, the operator $\Game\colon g\mapsto g^\Game$ can be viewed as an order-preserving map from $ ex\mathcal{MR}ed$ to $ ex\mathcal{MR}ed^{{\rm ct},\eta,\mu}$.
Conversely, as in \cite{Wes21}, one can also see that:
\begin{prop}
If an extended predicate $g$ is computably transparent, inflationary, and idempotent, then $g^\Game\leq_{eW}g$.
\end{prop}

\begin{proof}
(Sketch)
The proof is almost the same as Westrick's proof \cite{Wes21} of Fact \ref{fact:Westrick}.
To prove the assertion, we construct a computable function $\Phi^\ast$, where note that, by the recursion theorem, we can refer to  $\Phi^\ast$ itself during the construction of  $\Phi^\ast$.
Since $g$ is computationally transparent and idempotent, a computation that makes two queries to $g$ can be simulated by a computation that makes one query to $g$.
Note that a given \Art-strategy $\tau$ in $\mathfrak{G}(g)$ is of the form $\tau(\alpha)=\langle\tau_0(\alpha),\tau_1(\alpha)\rangle$, where $\tau_0(\alpha)$ is either {\tt Query} or {\tt Halt}, and consider the following two-step computation.
\[
\begin{array}{|ccc|}\hline 
 \Me & \Ar & \Ni \\[0.5em]
				& (\tau_1({{\tt p}_1},\dots,{\tt p}_n) \mid & {z_n})\in{\rm dom}(g) \\
{{\tt p}_{n+1}}\in g(\tau({{\tt p}_1},\dots,{\tt p}_n) \mid {z_n})	& & \\	 
				& (\Phi^\ast(\tau,{{\tt p}_1},\dots,{\tt p}_n,{\tt p}_{n+1}) \mid & {z_{n+1}})\in{\rm dom}(g) \\
\alpha\in g(\Phi^\ast(\tau,{{\tt p}_1},\dots,{\tt p}_{n+1})  \mid {z_{n+1}})	& & \\	 	 
 		& \langle {\tt p}_{n+1},\alpha\rangle &  \\[0.5em] \hline
\end{array}
\]

Here, if $\tau_0({\tt p}_1,\dots,{\tt p}_n,{\tt p}_{n+1})={\tt Halt}$ then $\Phi^\ast(\tau,{{\tt p}_1},\dots,{\tt p}_n,{\tt p}_{n+1})$ above is replaced with a fixed computable input in the domain of $g$.
Such an input exists since $g$ is inflationary.
In this case, we further replace the last output $\langle {\tt p}_{n+1},\alpha\rangle$ above with $\langle {\tt p}_{n+1},{\tt Halt}\rangle$.
Note that the query-making part of \Art's strategy above consists of $\tau_1({\tt p}_1,\dots,{\tt p}_n)$ and $\lambda y.\Phi^\ast(\tau,{\tt p}_1,\dots,{\tt p}_n,y)$.
Therefore, the only public information needed to make the two queries is $(\tau,{\tt p}_1,\dots,{\tt p}_n)$.
Since there is an effective way to convert this into a single query, we write $\Phi^\ast(\tau,{\tt p}_1,\dots,{\tt p}_n)$ for the public input part of the query.

Let us see how this one-query computation works.
If we ask for a solution of $g$ with respect to the pair of $\Phi^\ast(\tau,\ep)$ (where $\ep$ is an empty sequence) and an appropriate secret input, oracle should answer ${\tt p}_1$ and a solution $\alpha$ of $g$ with respect to the pair of $\Phi^\ast(\tau,{\tt p}_1)$ and a secret input.
If this secret input is appropriate, this value $\alpha$ is the pair of ${\tt p}_2$ and a solution $\alpha'$ of $g$ with respect to the pair of $\Phi^\ast(\tau,{\tt p}_1,{\tt p}_2)$ and a secret input.
Repeating this process, we obtain a sequence $\langle{\tt p}_1,{\tt p}_2,\dots,{\tt p}_n,{\tt Halt}\rangle$.
Applying \Art's strategy $\tau$ to this sequence, since $\tau({\tt p}_1,\dots,{\tt p}_n)=\langle{\tt Halt},{\tt q}_n\rangle$, one can compute a solution ${\tt q}_n\in g^\Game(\tau)$ with a single query to $g$.

The question that remains is whether appropriate secret inputs always appear in this disassembly process.
Now, for secret input, assuming that $z_n$ depends on some $z'_1,\dots,z'_{n-1}$, \Nim's strategy above depends only on $(\tau,{\tt p}_1,\dots,{\tt p}_n,z'_1,\dots,z'_{n-1})$, the first move is to take an appropriate $z_n$, and the next move is an appropriate map ${\tt p}_{n+1}\mapsto z_{n+1}$.
Thus, the $z_n'$ guaranteed by the inner secret reduction for $g\circ g\leq_{eW}g$ depends only on $(\tau,{\tt p}_1,\dots,{\tt p}_n,z'_1,\dots,z'_{n-1})$.
The $p_{n+1}$ and $\alpha$ obtained as solutions to $g$ with respect to this secret input are for the appropriate $z_n$ and $z_{n+1}$ as shown in the above table.
\end{proof}

By Proposition \ref{thm:computably-transparent-many-one-eWeihx}, this is equivalent to $g^\Game\leq_{em}g$.
Therefore:

\begin{cor}\label{cor:adjoint-eTW}
The operator $\Game\colon ex\mathcal{MR}ed\to ex\mathcal{MR}ed^{{\rm ct},\eta,\mu}$ is left adjoint to the inclusion $i\colon ex\mathcal{MR}ed^{{\rm ct},\eta,\mu}\mono ex\mathcal{MR}ed$.
\qed
\end{cor}

\begin{center}
%\begin{tikzcd}
% ex\mathcal{MR}ed
%\arrow[rr, "{\tt EWeih}"{name=A}, bend left=35] & &
% ex\mathcal{MR}ed^{\rm ct}
%\arrow[ll, "\iota"{name=B}, bend left=35]
%\arrow[rr, "{\tt pEWeih}"{name=C}, bend left=35] & &
% ex\mathcal{MR}ed^{{\rm ct},\eta}
%\arrow[ll, "\iota"{name=D}, bend left=35]
%\arrow[rr, "\Game"{name=E}, bend left=35] & &
% ex\mathcal{MR}ed^{{\rm ct},\eta,\mu}
%\arrow[ll, "\iota"{name=F}, bend left=35]
%--- Adjunction Symbol
%\arrow[phantom, from=A, to=B, "\dashv" rotate=-90, no line]
%\arrow[phantom, from=C, to=D, "\dashv" rotate=-90, no line]
%\arrow[phantom, from=E, to=F, "\dashv" rotate=-90, no line]
%\end{tikzcd}
%\includegraphics[width=12cm]{adjoint2.png}
\includegraphics[width=15cm]{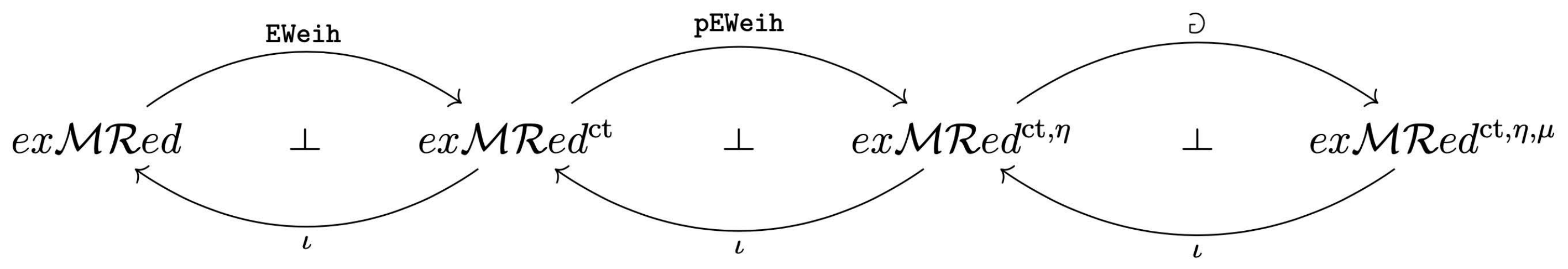}
\end{center}

\begin{cor}\label{prop:computably-transparent-many-one-TWeih-mm}
The Heyting algebra of many-one degrees of computably transparent, inflationary, idempotent, extended predicates on $\tpbf{N}$ is isomorphic to the LT-degrees (the extended Turing-Weihrauch degrees).
\qed
\end{cor}

%In a similar way, we can give a characterization of the extended Turing-Weihrauch degrees (LT-degrees).
%For the basics of extended Turing-Weihrauch reducibility (LT-reducibility), see \cite{Kih21}.
%Recall from Example \ref{example:LT-universal-oracle} that $g^\Game$ gives a universal computation relative to an LT-oracle.
%As in \cite{Kih20}, one can easily check that $g^\Game$ is computably transparent, inflationary, and idempotent.
%It is also shown $g^\Game\equiv_{LT}g$ in \cite{Kih21}.
%As in \cite{Wes21}, one can also see that if $g$ is computably transparent, inflationary, and idempotent, then $g^\Game\leq_{eW}g$.
%This means that for such a function $g$, $f\leq_{LT}g$ if and only if $f\leq_{eW}g^{\Game}\leq_{eW}g$.
%Hence, by Proposition \ref{prop:computably-transparent-many-one-Weih-bimap}, $f\leq_{LT}g$ if and only if $f\leq_mg$.

%%%%%%%%%%%%%%%%%%%%%%%%%%%%%%%%%%%%%%%%%%%%%%%%%%%%%%%%%
%%%%%%%%%%%%%%%%%%%%%%%%%%%%%%%%%%%%%%%%%%%%%%%%%%%%%%%%%
%%%%%%%%%%%%%%%%%%%%%%%%%%%%%%%%%%%%%%%%%%%%%%%%%%%%%%%%%
%%%%%%%%%%%%%%%%%%%%%%%%%%%%%%%%%%%%%%%%%%%%%%%%%%%%%%%%%
%%%%%%%%%%%%%%%%%%%%%%%%%%%%%%%%%%%%%%%%%%%%%%%%%%%%%%%%%

\subsection{Modality}\label{sec:mmmap-modality}

Using this notion, operations on $\Omega$ that do not necessarily preserve intersections can also be understood as properties related to oracle.
The discussion that follows is a further elaboration of the arguments in \cite{LvO,Kih21}.
Lee-van Oosten \cite{LvO} constructed a canonical extended predicate $U_j\pcolon\tpbf{N}\times\Omega\tto\tpbf{N}$ from a given operator $j\colon\Omega\to\Omega$ as follows:
\begin{align*}
{\rm dom}(U_j)=\{(x\mid p)\in\tpbf{N}\times\Omega:x\in j(p)\},
& &
U_j(x\mid p)=p.
\end{align*}

Compare also with Definition \ref{def:operation-to-pamulti}.
Since the intuitive meaning of this definition is difficult to grasp, we will analyze it with a slight restriction on $j$.
For an operation $j\colon\Omega\to\Omega$, we say that $j$ is {\em $\subseteq$-monotone} if, for any $p,q\in\Omega$, $q\subseteq p$ implies $j(q)\subseteq j(p)$.
If $j$ is $\subseteq$-monotone, it is easy to verify that the following holds.
\[x\in j(p)\iff (\exists q\in\Omega)\;U_j(x\mid q)\subseteq p\]

The forward direction is obvious since $x\in j(p)$ implies $U_j(x\mid p)=p$.
For the backward direction, by definition, $U_j(x\mid q)\subseteq p$ implies $x\in j(q)$ and $q\subseteq p$, so $\subseteq$-monotonicity ensures $x\in j(p)$.

Compare also with the equivalence (\ref{equ:from-modality-to-oracle-modest-case}) in Theorem \ref{thm:modality-to-oracle-various-preserve}.
The idea is to read the expression $x\in j(p)$ as ``If $x$ is given as input to the universal oracle computation corresponding to $j$, then $p$ can be solved under some secret advice''.
Let us use this to give the inverse construction.
Given an extended predicate $U\pcolon\tpbf{N}\times\Lambda\tto\tpbf{N}$, define a $\subseteq$-monotone operator $j_U\colon\Omega\to\Omega$ as follows:
\[j_U(p)=\{x\in\tpbf{N}:(\exists q\in\Lambda)\;[(x\mid q)\in{\rm dom}(U)\mbox{ and }U(x\mid q)\subseteq p]\}.\]

Note that this inverse construction differs from that used in \cite{LvO,Kih21}.
This simplification in this article is made possible by our use of the notion of computable transparency.

\begin{prop}\label{prop:transform-equ-bimap}
For any extended predicate $U\pcolon\tpbf{N}\times\Lambda\to\tpbf{N}$, $U_{j_U}$ is equivalent to $U$.
Conversely, for any $\subseteq$-monotone operation $j\colon \Omega\to\Omega$, we have $j_{U_j}=j$.
Moreover, any monotone operation is $r$-equivalent to a $\subseteq$-monotone operation.
\end{prop}

\begin{proof}
Given $x\in\tpbf{N}$ and $p\in\Omega$, $U_{j_U}(x\mid p)$ is defined if and only if $x\in j_U(p)$ if and only if $U(x\mid q)\subseteq p$ for some $q\in\Lambda$ by definition.
For such $q$, we have $U(x\mid q)\subseteq p=U_{j_U}(x\mid p)$; hence $\{U(x\mid p)\}_{p\in\Lambda}\subseteq\{U_{j_U}(x\mid p)\}_{p\in\Omega}$.
Moreover, as $U(x\mid q)\subseteq U(x\mid q)$, by the above equivalence, $U_{j_U}(x\mid U(x\mid q))$ is defined, and equal to $U(x\mid q)$ by canonicity.
This implies that $\{U_{j_U}(x\mid p)\}_{p\in\Omega}\subseteq\{U(x\mid p)\}_{p\in\Lambda}$; hence $U_{j_U}$ is equivalent to $U$.

Next, let $j\colon\Omega\to\Omega$ be an operation.
Then one can easily see the following equivalences:
\begin{align*}
x\in j_{U_j}(p)
&\iff (\exists q\in\Omega)\;U_j(x\mid q)\subseteq p\\
&\iff (\exists q\subseteq p)\;(x\mid q)\in{\rm dom}(U_j)
\iff (\exists q\subseteq p)\;x\in j(q)
\end{align*}

If $j$ is $\subseteq$-monotone, the last condition is equivalent to $x\in j(p)$.
Hence, $j_{U_j}=j$.
To verify the last assertion, note that $j_{U_j}$ is always $\subseteq$-monotone.
Thus, it suffices to show $j_{U_j}\equiv_r j$ for any monotone operation $j$.
By the above equivalence (with $q=p$), $x\in j(p)$ implies $x\in j_{U_j}(p)$, so the identity map witnesses $j\leq_r j_{U_j}$.
Conversely, if $x\in j_{U_j}(p)$ then $x\in j(q)$ for some $q\subseteq p$ by the above equivalence.
Assume that monotonicity of $j$ is witnessed by ${\tt u}\in \tplf{N}$.
As the identity map ${\tt i}$ realizes $q\arr p$, ${\tt u}\ast{\tt i}$ realizes $j(q)\arr j(p)$.
Hence, ${\tt u}\ast{\tt i}\ast x\in j(p)$.
Thus, $j_{U_j}\leq_r j$ is witnessed by $\lambda x.{\tt u}\ast{\tt i}\ast x$.
\end{proof}

\begin{lemma}\label{prop:compos-bifunction}
For any extended predicate $U$, $j_{U\circ U}=j_U\circ j_U$.
\end{lemma}

\begin{proof}
First, one can see that $x\in j_{U\circ U}(p)$ if and only if there exist $p'$ and $\eta$ such that $U\circ U(x\mid p',\eta)\subseteq p$.
By the definition of the composition, the last condition is equivalent to being $U(y\mid \eta(y))\subseteq p$ for any $y\in U(x\mid p')$.
By the definition of $j_U$, this implies that $y\in j_U(p)$ for any $y\in U(x\mid p')$; hence $U(x\mid p')\subseteq j_U(p)$.
By the definition of $j_U$ again, we get $x\in j_U(j_U(p))$.
Therefore, $j_{U\circ U}(p)\subseteq j_U\circ j_U(p)$.

Conversely, by applying the definition of $j_U$ twice, one can see that $x\in j_U\circ j_U(p)$ if and only if there exists $p'$ such that for any $y\in U(x\mid p')$ there exists $p''_y$ such that $U(y\mid p_y'')\subseteq p$.
Then put $\eta(y)=p_y''$.
By the definition of the composition, it is straightforward to see that $U\circ U(x\mid p',\eta)\subseteq p$ for such $p'$ and $\eta$.
By the definition of $j_{U\circ U}$, we get $x\in j_{U\circ U}(p)$; hence $j_U\circ j_U(p)\subseteq j_{U\circ U}(p)$.
\end{proof}

The following is a combination of the analogue of Theorem \ref{prop:transparent-to-modality} and the extended predicate version of Theorem \ref{thm:modality-to-oracle-various-preserve}.

\begin{theorem}\label{thm:modality-to-oracle-various-preserve-bimap}
Let $U\pcolon\tpbf{N}\times\Lambda\tto\tpbf{N}$ be an extended predicate.
If $U$ is computably transparent, then $j_U$ is monotone.
Moreover, if $U$ is inflationary, so is $j_U$; and if $U$ is idempotent, so is $j_U$.

Conversely, let $j\colon\Omega\to\Omega$ be an operation.
If $j$ is monotone, $U_j$ is computably transparent; if $j$ is inflationary, so is $U_j$; and if $j$ is monotone and idempotent, so is $U_j$.
\end{theorem}

\begin{proof}
Let $U$ be computably transparent.
To see that $j_U$ is monotone, assume that ${\tt a}$ realizes $p\arr q$.
Given ${\tt x}\in j_U(p)$, by the definition of $j_U$, we have $U({\tt x}\mid p')\subseteq p$ for some $p'$.
Hence, ${\tt a}\ast U({\tt x}\mid p')\subseteq q$.
By computable transparence, there exists $p''$ such that $U({\tt u}\ast{\tt a}\ast{\tt x}\mid p'')\subseteq {\tt a}\ast U({\tt x}\mid p')\subseteq q$.
Again, by the definition of $j_U$, we have ${\tt u}\ast{\tt a}\ast{\tt x}\in j_U(q)$.
Hence, $\lambda ax.{\tt u}\ast a\ast x\in \tplf{N}$ realizes $(p\arr q)\arr (j_U(p)\arr j_U(q))$.

Let $U$ be inflationary.
In this case, ${\tt x}\in p$ implies $U(\eta\ast{\tt x}\mid p')\subseteq\{{\tt x}\}\subseteq p$ for some $p'$.
By the definition of $j_U$, we have $\eta \ast{\tt x}\in j_U(p)$.
Therefore, $\lambda x. \eta\ast x\in \tplf{N}$ realizes $p\arr U(p)$.

Let $U$ be idempotent.
Assume $x\in j_U(j_U(p))$.
By Lemma \ref{prop:compos-bifunction}, we have $x\in j_{U\circ U}(p)$, which means $U\circ U(x\mid p')\subseteq p$ for some $p'$.
By idempotence, there exists $p''$ such that $U(\mu\ast x\mid p'')\subseteq U\circ U(x\mid p')\subseteq p$.
Thus, by the definition of $j_U$, we get $\mu\ast x\in j_U(p)$.
Hence, $\mu\in \tplf{N}$ realizes $j_U(j_U(p))\arr j_U(p)$.

Next, assume that $j$ is monotone, realized by ${\tt u}\in \tplf{N}$.
To show that $U_j$ is computably transparent, let ${\tt f},{\tt x}\in\tpbf{N}$ and $p\in\Omega$ be such that $f\ast U_j({\tt x}\mid p)$ is defined.
In particular, $({\tt x}\mid p)\in {\rm dom}(U_j)$, so ${\tt x}\in j(p)$ by the definition of $U_j$.
Next, note that ${\tt f}$ clearly realizes $U_j({\tt x}\mid p)\arr {\tt f}\ast U_j({\tt x}\mid p)$.
By monotonicity, ${\tt u}\ast{\tt f}$ realizes $j(U_j({\tt x}\mid p))\arr j({\tt f}\ast U_j({\tt x}\mid p))$.
As $U_j$ is canonical, we have ${\tt x}\in j(p)=j(U_j({\tt x}\mid p))$; hence ${\tt u}\ast {\tt f}\ast {\tt x}\in j({\tt f}\ast U_j({\tt x}\mid p))$.
In particular, $U_j({\tt u}\ast {\tt f}\ast {\tt x}\mid q)\subseteq {\tt f}\ast U_j({\tt x}\mid p)$ for some $q$, which means that $U_j$ is computably transparent.

Assume that $j$ is inflationary, realized by $\eta\in \tplf{N}$.
In particular, $\eta$ realizes $\{{\tt x}\}\arr j(\{{\tt x}\})$, so we have $\eta\ast{\tt x}\in j(\{\tt x\})$.
Then, by the definition of $U_j$, in particular, we have $p\in\Omega$ such that $U_j(\eta\ast{\tt x}\mid p)\subseteq\{\tt x\}$, which means that $U_j$ is inflationary.

Assume that $j$ is idempotent, realized by $\mu\in \tplf{N}$.
In particular, $\mu$ realizes $j\circ j(U_j\circ U_j({\tt x}\mid p,\eta))\arr j(U_j\circ U_j({\tt x}\mid p,\eta))$.
Note that from the definition of composition of multifunctions, if $({\tt x}\mid p,\eta)\in{\rm dom}(U_j\circ U_j)$ then ${\tt z}\in U_j({\tt x}\mid p)$ implies $({\tt z}\mid \eta({\tt z}))\in{\rm dom}(U_j)$.
Then, for such ${\tt x}$ and ${\tt z}$ we have $U_j({\tt z}\mid\eta({\tt z}))\subseteq U_j\circ U_j({\tt x}\mid p,\eta)$ by the definition of composition.
As $({\tt z}\mid\eta({\tt z}))\in{\rm dom}(U_j)$, we have ${\tt z}\in j(\eta({\tt z}))$ by the definition of $U_j$ and $\eta({\tt z})\subseteq U_j\circ U_j({\tt x}\mid p,\eta)$ by canonicity.
The identity map ${\tt i}$ realizes $\eta({\tt z})\arr U_j\circ U_j({\tt x}\mid p,\eta)$, so by monotonicity, ${\tt u}\ast{\tt i}\ast{\tt z}\in j(U_j\circ U_j({\tt x}\mid p,\eta))$.
Hence, ${\tt u}\ast{\tt i}$ realizes $U_j({\tt x}\mid p)\arr j(U_j\circ U_j({\tt x}\mid p,\eta))$.
By monotonicity, ${\tt u}\ast({\tt u}\ast{\tt i})$ realizes $j(U_j({\tt x}\mid p))\arr j\circ j(U_j\circ U_j({\tt x}\mid p,\eta))$.
Recall that $({\tt x}\mid p)\in{\rm dom}(U_j)$ implies $x\in j(p)=j(U_j({\tt x}\mid p))$, and thus ${\tt u}\ast({\tt u}\ast{\tt i})\ast{\tt x}\in j\circ j(U_j\circ U_j({\tt x}\mid p,\eta))$.
Hence, we obtain $\mu\ast({\tt u}\ast({\tt u}\ast{\tt i})\ast{\tt x})\in j(U_j\circ U_j({\tt x}\mid p,\eta))$.
By the definition of $U_j$, in particular, we have $U_j(\mu\ast({\tt u}\ast({\tt u}\ast{\tt i})\ast{\tt x})\mid q)\subseteq U_j\circ U_j({\tt x}\mid p,\eta)$ for some $q$.
This verifies that $U_j$ is idempotent.
\end{proof}

\subsection{Correspondence}

Let us now characterize reducibility notions on extended predicates as properties for operations on truth values.
First, we extend the correspondence between $m$-reducibility and $r$-reducibility (Proposition \ref{prop:two-category-isomorphic}) to extended predicates.

\begin{prop}\label{prop:two-category-isomorphic-bimap}
For any extended predicates $f$ and $g$, an extended $m$-morphism $e\colon f\to g$ can be thought of as an $r$-morphism $e\colon j_f\to j_g$.
Conversely, for any operations $j,k\colon\Omega\to\Omega$, an $r$-morphism $e\colon j\to k$ can be thought of as an extended $m$-morphism $e\colon U_j\to U_k$.
\end{prop}

\begin{proof}
Assume that $e$ is an $m$-morphism from $f$ to $g$.
Then for any $(x\mid p)\in{\rm dom}(f)$ there exists $q$ such that $g(\varphi_e(x)\mid q)\subseteq f(x\mid p)$.
Therefore, by definition, we have $x\in j_f(p)$ iff $f(x\mid p')\subseteq p$ for some $p'$, which implies $g(\varphi_e(x)\mid q')\subseteq p$ for some $q'$, iff $\varphi_e(x)\in j_g(p)$
for any $p\in\Omega$.
This means that $e$ realizes $j_f(p)\arr j_g(p)$, so $e$ is an $r$-morphism from $j_f$ to $j_g$.

Conversely, let $e$ be an $r$-morphism from $j$ to $k$.
Then, for any $p\in\Omega$ and ${\tt x}\in\tpbf{N}$, ${\tt x}\in j(p)$ implies $e\ast {\tt x}\in k(p)$.
By definition, this means that if $U_j({\tt x}\mid p)\downarrow=p$ implies $U_k(e\ast{\tt x}\mid p)\downarrow=p$.
In particular, by setting $p=U_j({\tt x}\mid p)$, we get $U_k(e\ast {\tt x}\mid p)\subseteq U_j({\tt x}\mid p)$.
This means that $e$ is an $m$-morphism from $U_j$ to $U_k$.
\end{proof}

\begin{cor}
The extended $m$-degrees (i.e., the $m$-degrees of predicates on multi-represented spaces) and the $r$-degrees of operations on $\Omega$ are isomorphic.
\qed
\end{cor}

Combining the results of this section and Section \ref{sec:universal-bimap}, various degree notions can be characterized using operations on truth values.
As we have seen that monotone operations on $\Omega$ correspond to computably transparent extended predicates (Proposition \ref{prop:transform-equ-bimap} and Theorem \ref{thm:modality-to-oracle-various-preserve-bimap}), which correspond to extended Weihrauch degrees (Corollary \ref{prop:computably-transparent-many-one-eWeih}), we get the following:

%By combining Theorems \ref{thm:modality-to-oracle-various-preserve} and \ref{prop:transparent-to-modality} and Propositions \ref{prop:two-category-isomorphic} and \ref{prop:computably-transparent-many-one-Weih}, we get the following:

\begin{cor}
The Heyting algebra of the Weihrauch degrees on multi-represented spaces is isomorphic to the $r$-degrees of monotone operations on $\Omega$.
\qed
\end{cor}

Similarly, we have seen that monotone, inflationary, operations on $\Omega$ correspond to computably transparent, inflationary, extended predicates (Proposition \ref{prop:transform-equ-bimap} and Theorem \ref{thm:modality-to-oracle-various-preserve-bimap}), which correspond to pointed extended Weihrauch degrees (Corollary \ref{cor:computably-transparent-inflationary-many-one-peWeih}), so we get the following:

\begin{cor}
The Heyting algebra of the pointed Weihrauch degrees on multi-represented spaces is isomorphic to the $r$-degrees of monotone, inflationary, operations on $\Omega$.
\qed
\end{cor}

We have also seen that Lawvere-Tierney topologies on $\Omega$ correspond to computably transparent, inflationary, idempotent, extended predicates (Proposition \ref{prop:transform-equ-bimap} and Theorem \ref{thm:modality-to-oracle-various-preserve-bimap}), which correspond to LT-degrees (extended Turing-Weihrauch degrees; Corollary \ref{cor:adjoint-eTW}).
This argument gives an alternative proof of the characterization in \cite{Kih21} that the Heyting algebra of LT-degrees is isomorphic to the $r$-degrees of Lawvere-Tierney topologies.
In other words:

\begin{cor}
The Heyting algebra of the Turing-Weihrauch degrees on multi-represented spaces is isomorphic to the $r$-degrees of Lawvere-Tierney topologies.
\qed
\end{cor}

The above extends the correspondence between the notions of degrees and operations on truth values.
In summary, when considering reducibility on (represented spaces over) a relative PCA, we could only obtain correspondence for operations that preserve intersection, but by considering reducibility on multi-represented spaces, we have obtained the complete correspondence (see also Table \ref{table}).

\begin{remark}
The above results give a more precise analysis of the previous correspondence theorems \cite{Kih21} that the extended Weihraruch degrees in the sense of Bauer \cite{Bau21} is isomorphic to the $r$-degrees of monotone operations on $\Omega$, and that the extended Turing-Weihrauch degrees (also called the LT-degrees in \cite{Kih21}) is isomorphic to the $r$-degrees of Lawvere-Tierney topologies.
%These precious correspondences will be re-discussed from our new perspective in Section \ref{sec:mmmap}.
\end{remark}

\begin{remark}[Joyal's theorem]
Satoshi Nakata (in private communication) pointed out that our construction of $j_U$ can be explained in terms of topos theory.
Joyal showed that any Lawvere-Tierney topology arises as a $\leq_r$-least topology that makes a subobject $m\colon R\mono X$ dense; see e.g.~Lee \cite[Section 2.2]{Lee}.

In a realizability topos, a subobject of $(X,\sim_X)$ can be expressed as a predicate $R\colon X\to\Omega$ which is strict and extensional with respect to $\sim_X$.
It is known that to obtain a Lawvere-Tierney topology on a realizability topos, we only need to consider a subobject of a multi-represented space.
That is, $X$ above can be assumed to be a multi-represented space, in which case $R$ is a realizability predicate in the sense of this section.
Our construction of $j_U$ can be regarded as an interpretation of the procedure for constructing a Lawvere-Tierney topology from a subobject $R\mono X$ as a relative computation with a predicate $R\colon X\to\Omega$ as an oracle.

To be explicit, recall from Definition \ref{def:universal-closure-operator-b} that a subobject $A\mono X$ is $j$-dense if ${\sf cl}_j(A)\equiv X$, where ${\sf cl}_j$ is the universal closure operator given by $j$.
In particular, $X\leq{\sf cl}_j(A)$ means that given a name of $x\in X$ one can compute a name of $x\in{\sf cl}_j(A)$; that is, one can compute a name of $x\in A$ with the help of $j$.
A name of $x\in A$ is nothing but a witness for $x\in A$, so this states that one can solve a witness-search problem for $x\in A$ with the help of $j$.
That $j$ is a least such topology means that $j$ is exactly an oracle just solving a witness-search problem for $x\in A$ (uniformly in $x\in X$) but nothing more.

As a concrete example, let us consider a topology obtained from a subterminal (i.e., a subobject of the terminal $\mathbf{1}$).
A subterminal is expressed as $R\colon\mathbf{1}\to\Omega$, i.e., it is a mass problem $R\in\Omega=\mathcal{P}(\tpbf{N})$, and therefore, it is predictably related to a Medvedev degree.
And it is known that an open topology can be defined as the least topology that makes a subterminal dense.
Compare this observation with Corollary \ref{cor:Medvedev-deg}.
\end{remark}

\subsection{Realizability relative to oracle}\label{sec:realizability-relative-to-oracle}

As an application of the correspondence results in Sections \ref{sec:modest-modality} and \ref{sec:mmmap-modality}, various notions of ``realizability relative to an oracle'' can be organized in a unified manner.
A specific example of such a realizability is Lifschitz realizability (see \cite[Section 4.4]{vOBook}), which can be thought of as realizability relative to the multivalued oracle ${\tt LLPO}$ (the lessor limited principle of omniscience).
As suggested in the last Remark in Section \ref{sec:sdst-notation}, the three fundamental properties (computable transparency, being inflationary, and idempotence) of multimaps are obtained by abstracting the conditions necessary for Lifschitz realizability to realize all the axioms of fundamental theories such as ${\sf HA}$ and ${\sf IZF}$.
Based on this observation, the author \cite{Kih20} has generalized the idea of Lifschitz realizability  to deal with realizability relative to any multifunctions on $\tpbf{N}$ which is computably transparent, inflationary, and idempotent (i.e., realizability relative to any Turing-Weihrauch oracle). 
By our correspondence theorem, this idea is absorbed by the notion of realizability relative to Lawvere-Tierney topologies.
%The reason why realizability relative to a Lawvere-Tierney topology has nice properties seems that a Lawvere-Tierney topology yields a subtopos, whose internal logic has a relationship with realizability relative to the topology; see also \cite{LvO,RaSw20}.

One benefit of the correspondence between Lawvere-Tierney topologies $j$ and oracles $U_j$ is that the internal logic of the $j$-sheaf topos can be analyzed by realizability relative to the corresponding oracle.
The realizability notion corresponding to a Lawvere-Tierney topology $j$ on the effective topos is described in Lee-van Oosten \cite{LvO}, but it is not written in the context of oracle, so we shall rewrite it here in the terms of oracle.
As explained in Section \ref{sec:universal-bimap}, the universal computation relative to an extended predicate $\theta$ is given as $\theta^\Game$.
Here, we understand the expression $\theta^\Game(e\mid z)\subseteq p$ to mean ``When the program $e$ is executed with the oracle $\theta$, all paths of nondeterministic computation that follow the advisor $z$ terminates and any output is a solution to $p$.''
Or, using the terminology of games (see Definition \ref{def:game-LT-reducibility} and also \cite{Kih21}), it could be read to mean ``$(e\mid z)$ is a winning \Art-\Ni strategy in the reduction game to $\theta$, and no matter what strategy \Me follows, \Art's final move always gives a solution to $p$''.

Let us now rewrite $\theta$-realizability for arithmetical formulas introduced in Lee-van Oosten \cite{LvO} using the notion of oracle.
Recalling the notation in Definition \ref{def:Kleene-realizability}, we define $p^\theta\subseteq\tpbf{N}$ for an arithmetical formula $p$ as follows:
\begin{align*}
(p\land q)^\theta&=p^\theta\land q^\theta,
\qquad
(p\lor q)^\theta=p^\theta\lor q^\theta,
\qquad
(\exists n. p(n))^\theta=\{\langle{\tt m},a\rangle:a\in (p({\tt m}))^\theta\}\\
(p\arr q)^\theta&=\{e\in\tpbf{N}:(\forall x)\;[x\in{p}^\theta\;\rightarrow\;(\exists z)\;\theta^\Game(e\ast x\mid z)\subseteq{q}^\theta]\}\\
(\forall n. p(n))^\theta&=\{e\in\tpbf{N}:(\forall n)\;(\exists z)\;\theta^\Game(e\ast {\tt n}\mid z)\subseteq(p({\tt n}))^\theta\}
\end{align*}

Then, we say that $p$ is $\theta$-realizable if $p$ has a computable $\theta$-realizer, i.e., $p^\theta\cap N\not=\emptyset$.
This is exactly the same as the definition in Lee-van Oosten \cite{LvO}, but considering the meaning of $\theta^\Game(e\mid z)\subseteq p$ mentioned above, it is persuasive that this is a very natural definition as a realizability relative to the oracle $\theta$.

As we saw in Section \ref{sec:universal-bimap}, if $j$ is the Lawvere-Tierney topology corresponding to $\theta^\Game$, then
\begin{align}\label{equ:corresp-LT-bimap}
x\in j(p)\iff(\exists z)\;\theta^\Game(x\mid z)\subseteq p
\end{align}
holds, so the $j$-transformation $(-)_j$ corresponding to Lee-van Oosten's $\theta$-realizability can be written as follows:
\begin{align*}
(p\land q)_j&=p_j\land q_j,
\qquad
(p\lor q)_j=p_j\lor q_j,
\qquad
(\exists n. p(n))_j=\exists n. (p(n))_j\\
(p\arr q)_j&=p_j\arr j(q_j)
\qquad
(\forall n. p(n))_j=\forall n.j((p(n))_j)
\end{align*}

%On the other hand, a superficially different definition of realizability relative to oracle has been used in Kihara \cite{Kih20}.
%The origin of the definition is Lifschitz realizability (see \cite[Section 4.4]{vOBook}), which can be thought of as realizability relative to the multivalued oracle ${\tt LLPO}$ (the lessor limited principle of omniscience).
%Generalizing it to any multivalued oracle yields the definition in \cite{Kih20}.
%This definition, expressed in terms of the corresponding Lawvere-Tierney topology, is given as follows:
On the other hand, generalized Lifschitz realizability in \cite{Kih20}, expressed in terms of the corresponding Lawvere-Tierney topology, is given as follows:
\begin{align*}
(p\land q)^j&=p^j\land q^j,
\qquad
(p\lor q)^j=j(p^j\lor q^j),
\qquad
(\exists n. p(n))^j=j(\exists n. (p(n))^j)\\
(p\arr q)^j&=p^j\arr q^j
\qquad
(\forall n. p(n))^j=\forall n.(p(n))^j
\end{align*}

That a generalization of Lifschitz realizability would result in such a definition is also pointed out in Rathjen-Swan \cite{RaSw20}.
Either way, they are known to behave very well, the latter being the G\"odel-Gentzen-style $j$-translation and the former closer to the Kuroda-style $j$-translation; see \cite{vdB18}.
Realizability relative to any extended predicate oracle based on the latter can also be easily defined via the equivalence (\ref{equ:corresp-LT-bimap}).

Some of the many advantages of treating $j$-translation as realizability relative to oracle are that oracles are far easier to make concrete examples of than logical operations, and also allow the use of ideas and techniques from computability theory.
One trivial example is hyperarithmetical realizability (and its relatives), but another, slightly more nontrivial examples are realizability relative to multimaps (i.e., Turing-Weihrauch oracles), which are useful for separating various nonconstructive principles in constructive reverse mathematics \cite{Kih20}.

Indeed, the author's goal in \cite{Kih20} was to solve a problem of separating the strengths of specific nonconstructive principles in constructive reverse mathematics presented by Ishihara-Nemoto (a problem of separating some principles that are equivalent under the countable choice in the absence of the countable choice), which was actually accomplished using the idea mentioned above.
In this sense, our work is not just an abstract theory, but has concrete applications.

And one of the benefits of generalizing realizability to any extended predicate oracle (not just multivalued oracle) is that it can handle realizability that cannot be expressed in the framework of a multivalued oracle, such as random realizability.
For example, the oracle ${\tt ProbError}_\ep$ introduced in Kihara \cite{Kih21}, which corresponds to probabilistic computation, is an extended predicate but not a multivalued oracle.
One of the notions of random realizability has also been introduced in Carl et al.~\cite{CGP21}.
The exact connection between this and ${\tt ProbError}_\ep$-realizability is not clear, since Carl et al.'s random realizability has a slightly convoluted form, but they seem to be deeply related.

\begin{question}
Find out the exact relationship between ${\tt ProbError}_\ep$-realizability and Carl et al.'s random realizability.
\end{question}

Also, extended oracles, rather than multivalued oracles, encompass topologies that induce classical logic, such as the double negation topology.
This may allow us to interpret some of the activities in classical reverse mathematics \cite{SOSOA:Simpson} in our context.
The familiar definition of Turing-Weihrauch reducibility is due to Hirschfeld-Jockusch \cite{HiJo16}, whose aim was to give a game-theoretic characterization of the implication relation between $\Pi^1_2$ principles over the base system ${\sf RCA}$ of classical reverse mathematics; see also \cite{DHR20}.
As a special case, one of the objects of their consideration is the preorder ${\sf P}\leq_\om {\sf Q}$ for $\Pi^1_2$ principles ${\sf P}$ and ${\sf Q}$, which means that every $\om$-model of ${\sf RCA}+{\sf Q}$ is a model of ${\sf P}$.
Note that, to characterize the $\om$-model reducibility ordering $\leq_\om$, the Turing-Weihrauch reduction game cannot actually be used as is, but instead its non-uniform version was used (\cite[Proposition 4.2]{HiJo16}).
As mentioned in the last Remark in Section \ref{sec:secret-input}, the use of extended oracle makes it possible to deal with the notion of non-uniform computability.
In our terminology, the $\om$-model reducibility ordering can be expressed as follows:

\begin{obs}
${\sf P}\leq_\om{\sf Q}$ if and only if ${\sf P}\leq_{LT}{\sf Q}\sqcup{\sf DNE}_\N$, where $\sqcup$ is the join in the Weihrauch lattice (over the Kleene-Vesley algebra).
\end{obs}

This also gives an embedding of the $\om$-model degrees of $\Pi^1_2$ principles into the poset of the Lawvere-Tierney topologies on the Kleene-Vesley topos.
However, the topology corresponding to ${\sf DNE}_\N$ is not Boolean, so there does not seem to be a perfect correspondence in the context of internal logic.

There are several other notable recent studies on the application of Lawvere-Tierney topology to logic.
On the relation between Lawvere-Tierney topology and intuitionistic arithmetic, a significant work has recently been done by Nakata \cite{Nak24}, which deserves serious attention.

\section{Oracles in descriptive set theory}\label{sec:DST}

%For an abstract theory to have value, it must explicitly present a wealth of concrete examples.
In this section, let us give rich concrete examples of computably transparent maps in the Kleene-Vesley algebra or Kleene's second algebra.
Such concrete examples are obtained from the theory of hierarchies.
One of the main aims of classical computability and descriptive set theory is to analyze the hierarchical structure of definability;
for example, the lowest level of the hierarchy is definability through arithmetical quantification (i.e., iteration of Turing jumps), but the idea was extended to infinitary logic (with well-founded syntax trees), and then to infinite nesting of quantifiers (which deals with formulas with ill-founded syntax trees), which gave rise to a mechanism to transform a given generalized quantifier into a more complex one (see e.g.~\cite{Kan88}), etc.; see also \cite{HinBook}.

Of course, these studies dealt with infinitary objects, but one of the major achievements of classical computability and descriptive set theory was the uncovering of a series of hidden computability-like structures of such infinitary notions.
This field, called {\em generalized recursion theory}, emerged as a synthesis of such infinitary computation theories and flourished in the last century (see e.g.~Hinman \cite[Part C]{HinBook} and Sacks \cite[Parts B--D]{sacks2}).
For example, admissible recursion theory (Kripke-Platek set theory) is a major part of the theory.
A related concept is Spector pointclass, around which the theory of pointclass was developed in descriptive set theory (see e.g.~\cite{MosBook}).

In these classical theories, concrete pointclasses corresponding to various types of definability have been studied, providing a wealth of concrete examples.
In this section, we shall see that from each pointclass that is good enough, we can obtain a computably transparent map and a relative pca.

%The arithmetic hierarchy is seen as a hierarchy of relativization to the halting problem oracle.
%More precisely, it is a hyperarithmetical hierarchy with infinitary formulas with a well-founded syntax tree rather than a formula with a finite syntax tree, but this analogy was subsequently studied at the level of more complex syntax trees, such as quantification (game quantifer, generalized quantifier) with an ill-founded sytanx tree, as well as computational analogies in description set theory.
%In a sense, these seem to be understandable in terms of relativization to more complex oracles.

\subsection{Pointclass and measurability}

From now on, we connect the notions we have dealt with so far with those that have been studied in depth in classical descriptive set theory.
The key idea here is to create a function concept from a set concept.
%This idea is often used, for example, in descriptive set theory.
For example, the notion of a continuous function is automatically obtained by defining a topology on a set $X$, i.e., a family of open sets $\mathcal{O}_X\subseteq\mathcal{P}(X)$.
When a family $\mathcal{O}_X$ of open sets in $X$ is specified, the family $\mathcal{B}_X\subseteq\mathcal{P}(X)$ of Borel sets is also defined as the smallest $\sigma$-algebra including $\mathcal{O}_X$, which yields Borel measurable functions.
In descriptive set theory, this kind of idea is generalized as follows.

\begin{definition}[see e.g.~\cite{MosBook}]
Let $\mathcal{C}$ be a class of sets.
If an assignment $X\mapsto \Gamma_X\subseteq\mathcal{P}(X)$ for each $X\in\mathcal{C}$ is given, such an assignment $\Gamma=(\Gamma_X)_{X\in\mathcal{C}}$ is called a {\em pointclass}.
\end{definition}

The most commonly used pointclasses in descriptive set theory and computability theory are $\Sigma^i_n$, $\Pi^i_n$, $\Delta^i_n$, and their boldface versions (see \cite{MosBook} and also below).
However, it should be noted that there are a number of other natural pointclasses that have been studied in great depth as mentioned above.

In descriptive set theory, $\mathcal{C}$ usually consists of (quasi-)Polish spaces and its subspaces.
Under this setting, each $X\in\mathcal{C}$ can be thought of as either a topological space or a represented space (see Example \ref{exa:coding-topological-space}).
Many natural pointclasses have good closure properties.

\begin{definition}
A {\em boldface} pointclass is a pointclass $\Gamma$ which is closed under continuous substitution; that is, for any continuous map $f\colon X\to Y$, $U\in \Gamma_Y$ implies $f^{-1}[U]\in\Gamma_X$.
Similarly, a {\em lightface} pointclass is a pointclass $\Gamma$ which is closed under computable substitution; that is, for any computable map $f\colon X\to Y$, $U\in \Gamma_Y$ implies $f^{-1}[U]\in\Gamma_X$.
\end{definition}

In other words:
\begin{itemize}
\item A boldface pointclass over $\mathcal{C}$ is a subfunctor of the powerset functor $\mathcal{P}\colon\mathcal{C}^{\rm op}\to{\bf Set}$, where $\mathcal{C}$ is a full subcategory of the category ${\bf Top}$ of topological spaces and continuous maps.
\end{itemize}

The above $\mathcal{C}$ can also be a full subcategory of the category of multi-represented spaces and {\em continuous} maps.
A similar characterization for a lightface pointclass is also given:
\begin{itemize}
\item A lightface pointclass over $\mathcal{C}$ is a subfunctor of the powerset functor $\mathcal{P}\colon\mathcal{C}^{\rm op}\to{\bf Set}$, where $\mathcal{C}$ is a full subcategory  of the category of multi-represented spaces and {\em computable} maps.
\end{itemize}

Most pointclasses considered in descriptive set theory are either lightface or boldface.
%The most commonly used pointclasses in descriptive set theory and computability theory are $\Sigma^i_n$, $\Pi^i_n$, $\Delta^i_n$, and their boldface versions (see \cite{MosBook} and also below).
The following representable pointclasses (representable subfunctors) are also important in synthetic descriptive set theory.

\begin{example}[Representable pointclass]\label{exa:representable-pointclass}
Recall from Example \ref{exa:coding-continuous functions} that the hyperspace $\mathcal{O}(X)$ of open sets in $X$ can be identified with the function space $C(X,\mathbb{S})$, where $\mathbb{S}$ is the Sierpi\'nski space.
Thus, the pointclass of open sets is exactly the representable functor $C(-,\mathbb{S})$.

In general, let $U$ be a transparent map.
Then $A\subseteq X$ is said to be {\em $U$-open} if $A\in \mathcal{O}_X^U$; that is, $A\in C(X,U(\mathbb{S}))$ (see Example \ref{exa:synthetic-DST}).
Then $\mathcal{O}^U=(\mathcal{O}_X^U)_{X\in {\bf Rep}}$ forms a boldface pointclass.
The core idea of synthetic descriptive set theory seems to attempt to reformulate the theory based on the observation that most natural pointclasses in descriptive set theory are representable.
For the details, see \cite{PaBr15,dBPa17}.

%If $A$ is a computable point in $U(\mathcal{O}_X)$, i.e., $\lambda x.A\colon\mathbf{1}\to U(\mathcal{O}_X)$ is computable, then $A$ is said to be {\em computably $U$-open}.
%
%More explicitly, $A$ is $U$-open if and only if the set of all names of points in $A$ can be written as the preimage $f^{-1}[B]$ of some open set $B\subseteq\N^\N$ under a partial $U$-continuous function $f$ such that ${\rm dom}(\delta_X)\subseteq{\rm dom}(f)$.
%Then $A$ is computably $U$-open if and only if such $B$ is c.e.~open, and $f$ is $U$-computable.
%
%Let $\tpbf{\Sigma}_{U,X}$ be the set of all $U$-open subsets of $X$, and $\Sigma_{U,X}$ be all computably $U$-open subsets of $X$.
%Then $\tpbf{\Sigma}_{U}=(\tpbf{\Sigma}_{U,X})_{X\in\tau}$ and ${\Sigma}_{U}=({\Sigma}_{U,X})_{X\in\tau}$ are pointclasses, where $\tau$ is the set of all represented spaces.

Perhaps this idea goes back at least to Longley \cite[Section 4.1.2]{LoPhD95}, which is to think the representable functor $C(-,\Sigma)$ for any pre-dominance $\Sigma$, where a pre-dominance is a two-point assembly $\Sigma$ with a distinguished computable element $\top\in\Sigma$ \cite{LoPhD95,Nak23}.
In this case, a subset $A\subseteq X$ with $A\in C(X,\Sigma)$ is usually called a $\Sigma$-subobject.
If $\mathcal{O}^\Sigma_X$ is the set of all $\Sigma$-subobjects of $X$, then $\mathcal{O}^\Sigma=(\mathcal{O}_X^\Sigma)_{X\in {\bf MultRep}}$ forms a boldface pointclass.
\end{example}

Hereafter, we assume that $\mathcal{C}$ is closed under binary product.
For $A\subseteq I\times X$, by $A_e$ we denote the $e$-th section $\{x\in X:(e,x)\in A\}$ for each $e\in I$.
An important pointclass in descriptive set theory is a pointclass that can be coded or represented.

\begin{definition}[see e.g.~\cite{MosBook}]
Let $(\Gamma_X)_{X\in\mathcal{C}}$ be a pointclass.
Then a set $G\in\Gamma_{I\times X}$ is {\em $\Gamma_X$-universal} if for any $A\in\Gamma_X$ there exists $e\in I$ such that $A=G_e$ (the $e$-th section of $G$).
A pointclass $\Gamma$ is {\em $I$-parametrized} if for any $X\in\tau$ there exists a $\Gamma_X$-universal set $G^X\in\Gamma_{I\times X}$.
\end{definition}

In descriptive set theory, $\N$- and $\N^\N$-parametrized pointclasses play important roles.
For instance, $\Sigma^i_n$ and $\Pi^i_n$ are $\N$-parametrized, and $\tpbf{\Sigma}^i_n$ and $\tpbf{\Pi}^i_n$ are $\N^\N$-parametrized.

\begin{obs}
If $\Gamma$ is $I$-parametrized, then for each $X\in\mathcal{C}$ one can consider $\Gamma_X$ as an $I$-represented space.
\end{obs}

\begin{proof}
The map $G^X_\bullet\colon I\to\Gamma_X$ defined by $p\mapsto G^X_p$ is a total $I$-representation of $\Gamma_X$.
\end{proof}

Thus, when parametrized, one can think of a pointclass as an indexed family of (total) represented spaces, and in many cases, indices are also represented spaces; that is, it often yields a functor ${\bf Rep}^{\rm op}\to{\bf Rep}$.
See also Gregoriades et al.~\cite{GKP} for an explanation of descriptive set theoretic notions using represented spaces.

For such represented spaces, we can discuss computability and continuity for functions on them, but of particular importance are computability and continuity of evaluation maps, ${\rm eval}_\Gamma\colon \Gamma_{X\times Y}\times X\to\Gamma_Y$, defined by $(A,x)\mapsto A_x$, where recall that $A_x$ is the $x$-th section of $A$.
Note that if $\Gamma$ is a representable pointclass $\mathcal{O}^\Sigma$ as in Example \ref{exa:representable-pointclass} then ${\rm eval}_\Gamma$ is computable by the property of the exponential object $\Sigma^X\simeq\mathcal{O}^\Sigma_X$.
However,
%there are several natural unrepresentable pointclasses (e.g.~the arithmetical pointclass), and
in general, evaluation maps are not always continuous, in which case we change representations by replacing a universal set with another one.

Before going into the details of this, we introduce some basic notions.
The first is the relativization of pointclass (see also \cite[Section 3H]{MosBook}).

\begin{definition}
For a pointclass $\Gamma$ and an oracle $\ep\in\N^\N$, define $\Gamma^\ep_X$ as the collection of $\ep$-th sections (i.e., sets of the form $P_\ep$) of some $P\in\Gamma_{\N^\N\times X}$.
Then define $\tpbf{\Gamma}_X=\bigcup_{\ep\in\N^\N}\Gamma_X^\ep$.
\end{definition}

For instance, $\tpbf{\Sigma}^0_1$ is the pointclass consisting of open sets, and $\tpbf{\Delta}^1_1$ is the pointclass consisting of Borel sets.

Below let $\mathcal{C}_0$ be the smallest collection of sets that contains $\N$ and $\N^\N$, and is closed under binary product, and let $\mathcal{C}$ be the collection of all Polish spaces.
The following fact is known as {\em smn theorem} or {\em good parameterization lemma}:

\begin{fact}[see Moschovakis {\cite[Lemma 3H.1]{MosBook}}]\label{thm:PCA-parametrization-theorem}
Let $\Gamma$ be an $\N$-parametrized lightface pointclass.
Then there is an $\N^\N$-parametrization $G$ for $\Gamma$ satisfying the following:
\begin{enumerate}
\item For any set $P\in\Gamma_X$ there exists a computable element $\ep\in \N^\N$ such that $P=G_\ep$.
\item For any $X\in\mathcal{C}_0$ and $Y\in\mathcal{C}$, the evaluation map ${\rm eval}_\Gamma\colon\tpbf{\Gamma}_{X\times Y}\times X\to\tpbf{\Gamma}_Y$ is computable with respect to the representations induced from $G$.
\end{enumerate}
\end{fact}

This follows from Moschovakis \cite[Lemma 3H.1]{MosBook} because ${\rm eval}_\Gamma\colon\tpbf{\Gamma}_{X\times Y}\times X\to\tpbf{\Gamma}_Y$ is computable if and only if there exists a computable function $S$ which, given a code $\ep$ of $A\in\tpbf{\Gamma}_{X\times Y}$ and $x\in X$, returns a code $S(\ep,x)$ of $A_x\in\tpbf{\Gamma}_Y$.
%Note also that such an $S$ is total.
Assume that all pointclasses mentioned from here on are given such parametrization.

A pointclass $\Gamma$ is {\em adequate} if it is a lightface pointclass which contains all computable sets, and is closed under $\land$, $\lor$, and bounded quantifiers; see \cite[Section 3E]{MosBook} for the details.
For example, $\Sigma^i_n,\Pi^i_n,\Delta^i_n$ are adequate.
Good parametrization lemma implies the following {\em uniform closure theorem}.

\begin{fact}[see Moschovakis {\cite[Theorem 3H.2]{MosBook}}]\label{fact:uniform-closure-theorem}
Let $\Gamma$ be an $\N$-parametrized adequate pointclass.
If $\Gamma$ is closed under binary intersection, then $\cap\colon\tpbf{\Gamma}_X\times\tpbf{\Gamma}_X\to\tpbf{\Gamma}_X$ is computable for any $X\in\tau$.
The same is true for binary union, projection, etc.
\end{fact}

%We now create a function concept from a set concept:

\begin{definition}
Let $\Gamma$ be a pointclass.
For a set $X\in\mathcal{C}$ and a topological space $Y$, a function $f\colon X\to Y$ is {\em $\Gamma$-measurable} if $f^{-1}[A]\in\Gamma_X$ for any open set $A\subseteq Y$.
\end{definition}

For a transparent map $U$, Pauly-de Brecht {\cite[Corollary 11]{PaBr15}} showed that $f\colon X\to Y$ is $U$-continuous (in the sense of Definitions \ref{def:basic-computably-transparent} and \ref{def:multi-represented-space}) if and only if it is $\mathcal{O}^U$-measurable whenever $Y$ is $U$-admissible.

%We next analyze the properties of $\tpbf{\Gamma}$-measurable functions.

%Let $\mathcal{F}=(\mathcal{F}_{X,Y})_{X,Y\in\tau}$ be an indexed collection, where $\mathcal{F}_{X,Y}$ is a family of functions from $X$ into $Y$.
%We say that a pointclass $\Gamma$ is {\em closed under $\mathcal{F}$-substitution} if, for any $X,Y\in\tau$ and any $f\in\mathcal{F}_{X,Y}$, if $U\in\Gamma_Y$ then $f^{-1}[U]\in\Gamma_X$.

A {\em $\Sigma$-pointlcass} is a pointclass containing all $\Sigma^0_1$ sets (i.e., c.e.~open sets), and is closed under trivial substitution, $\land$, $\lor$, bounded quantifiers, existential quantifiers on $\N$; see \cite[Section 3E]{MosBook} for the details.
For example, $\Sigma^i_n,\Pi^1_n,\Delta^1_n$ are $\Sigma$-pointclasses.
By the uniform closure theorem, one may assume that it is uniformly closed under these operations.

\begin{prop}\label{prop:Gamma-measurable-space}
Let $\Gamma$ be an $\N$-parametrized adequate $\Sigma$-pointclass, $X\in\tau$, and $Y$ be a Polish space.
Then, a function $f\colon X\to Y$ is $\tpbf{\Gamma}$-measurable if and only if the preimage map $f^{-1}[\cdot]\colon O_Y\to\tpbf{\Gamma}_X$ is continuous; that is, there exists a continuous function which, given a code of an open set $U\subseteq Y$, returns a code of $f^{-1}[U]\in\Gamma_X$.
\end{prop}

\begin{proof}
The forward direction is obvious.
For the backward direction, let $(B_n^Y)_{n\in\N}$ be a countable basis of $Y$.
For each $n\in\N$, since $f$ is $\tpbf{\Gamma}$-measurable, we have $f^{-1}[B_n^Y]\in\tpbf{\Gamma}_X$, so let $\beta_n$ be its code.
Then define $\beta\in\N^\N$ by $\beta(\langle n,m\rangle)=\beta_n(m)$.
Each code of an open set $U\subseteq Y$ is given as $\alpha$ such that $U=\bigcup_{n\in\N}B_{\alpha(n)}^Y$.
Then $x\in f^{-1}[U]$ if and only if there exists $n\in\N$ such that $(\beta_{\alpha(n)},x)\in G^{{\N^\N}\times X}$.
As $\tpbf{\Gamma}$ is closed under substitution by $n\mapsto\beta_{\alpha(n)}$ and existential quantifiers on $\N$, the latter set is in $\tpbf{\Gamma}$.
%Note that $f^{-1}[U]$ is the projection of the set $\{(n,x):(\beta_{\alpha(n)},x)\in G^{{\N^\N}\times X}\}$;
Thus, by the uniform closure theorem (Fact \ref{fact:uniform-closure-theorem}), one can continuously obtain a code of $f^{-1}[U]$.
\end{proof}

By $\tpbf{\Gamma}(X,Y)$ we denote the set of all $\tpbf{\Gamma}$-measurable functions from $X$ into $Y$.
By Proposition \ref{prop:Gamma-measurable-space}, one can consider $\tpbf{\Gamma}(X,Y)$ as a represented space, where a code of $f\in\tpbf{\Gamma}(X,Y)$ is given by a code of the continuous preimage map $f^{-1}[\cdot]\in C(O_Y,\tpbf{\Gamma}_X)$.
This idea can be extended to partial functions.
A partial function $f\pcolon X\to Y$ is $\tpbf{\Gamma}$-measurable if for any open set $U\subseteq Y$ there exists a $\Gamma$ set $A$ such that $f^{-1}[U]=A\cap{\rm dom}(f)$.
It is easy to see that Proposition \ref{prop:Gamma-measurable-space} also holds for partial functions, by taking a code of $A_n$ such that $f^{-1}[B_n^Y]=A_n\cap{\rm dom}(f)$ instead of $\beta_n$.
Thus the set $\tpbf{\Gamma}(\subseteq X,Y)$ of all partial $\tpbf{\Gamma}$-measurable functions from $X$ into $Y$ can also be (multi-)represented.

\begin{prop}[uniform measurability]\label{prop:intro-uniform-measurable}
Let $\Gamma$ be an $\N$-parametrized adequate $\Sigma$-pointclass.
Then the map, ${\tt preim}\colon\tpbf{\Gamma}(X,Y)\times O_Y\to\tpbf{\Gamma}_X$, defined by $(f,U)\mapsto f^{-1}[U]$ is computable.
\end{prop}

\begin{proof}
Define ${\tt frame}\colon\tpbf{\Gamma}(X,Y)\to C(O_Y,\Gamma_X)$ by $f\mapsto f^{-1}[\cdot]$.
By definition, the identity map on the codes witnesses that ${\tt frame}$ is computable.
Consider the evaluation map, ${\rm ev}\colon C(O_Y,\tpbf{\Gamma}_X)\times O_Y\to \tpbf{\Gamma}_X$, which is computable.
Then one can easily check ${\tt preim}={\rm ev}\circ({\tt frame}\times{\rm id})$, which is computable as it is the composition of computable maps.
%\[
%\xymatrix{
%	\tpbf{\Gamma}(X,Y)\times O_Y\ar[rr]^{{\rm frame}\times{\rm id}} \ar@/^2pc/[rrrr]^{{\rm preim}} & & C(O_Y,\tpbf{\Gamma}_X)\times O_Y \ar[rr]^{\qquad{\rm ev}} & & \tpbf{\Gamma}_Y\\
%	\N^\N \ar@{=}[rr]_{\rm id} \ar[u] & & \N^\N \ar[rr] \ar[u] & & \N^\N \ar[u]
%}
%\]
%
\end{proof}

A partial function $f\pcolon X\to Y$ is {\em $\Gamma$-computable} if $f$ is a computable element in $\tpbf{\Gamma}(\subseteq X,Y)$.
A $\Gamma$-computable function $g\pcolon\N^\N\times X\to Y$ is {\em universal} if for any partial $\tpbf{\Gamma}$-measurable function $f\pcolon X\to Y$ there exists $\ep\in\N^\N$ such that $f(x)=g(\ep,x)$ for any $x\in{\rm dom}(f)$.

\begin{prop}[universal function]\label{prop:universal-func-pointclass}
Let $\Gamma$ be an $\N$-parametrized adequate $\Sigma$-pointclass.
Then, for any $X,Y\in\tau$, there exists a universal $\Gamma$-computable function $\varphi^\Gamma\pcolon\N^\N\times X\to Y$.
\end{prop}

\begin{proof}
For $p\in\N^\N$, let $\varphi_p^\Gamma$ be the partial $\tpbf{\Gamma}$-measurable function from $X$ to $Y$ coded by $p$.
Then define $\varphi^\Gamma\pcolon\N^\N\times X\to Y$ by $\varphi^\Gamma(p,x)=\varphi^\Gamma_p(x)$.
By Proposition \ref{prop:intro-uniform-measurable}, given a code $u$ of an open set $U\subseteq Y$ and a code $p$ of a $\tpbf{\Gamma}$-measurable function, one can compute a code $r(u,p)$ of $(\varphi^\Gamma_p)^{-1}[U]\in\tpbf{\Gamma}_X$.
Hence, we get $(\varphi^\Gamma)^{-1}[U]=\{(p,x):(r(u,p),x)\in G^X\}$.
As $\Gamma$ is closed under computable substitution, $A=\{(u,p,x):(r(u,p),x)\in G^X\}$ is in $\Gamma_{\N^\N\times\N^\N\times X}$.
By Fact \ref{thm:PCA-parametrization-theorem}, $A$ has a computable code $e$, so by good parametrization lemma (Fact \ref{thm:PCA-parametrization-theorem}), one can compute a code $S(e,u)$ of $A_u=\psi^{-1}[U]$.
Therefore, $u\mapsto S(e,u)$ witnesses that $U\mapsto(\varphi^\Gamma)^{-1}[U]$ is computable.
Hence, $\varphi^\Gamma$ a computable element in $\tpbf{\Gamma}(\subseteq\N^\N\times X,Y)$.
The universality of $\varphi^\Gamma$ is obvious.
\end{proof}

\begin{prop}[computable transparency]\label{prop:computable-transparent-pointclass}
Let $\Gamma$ be an $\N$-parametrized adequate $\Sigma$-pointclass.
Then, the map, ${\tt comp}\colon\tpbf{\Gamma}(X,Y)\times C(Y,Z)\to\tpbf{\Gamma}(X,Z)$, defined by $(f,g)\mapsto g\circ f$ is computable.
\end{prop}

\begin{proof}
By Proposition \ref{prop:intro-uniform-measurable}, the maps $(g,U)\mapsto g^{-1}[U]\colon C(Y,Z)\times O_Z\to O_Y$ and $(f,V)\mapsto f^{-1}[V]\colon\tpbf{\Gamma}(X,Y)\times O_Y\to \tpbf{\Gamma}_X$ are computable.
Hence, $(f,g,U)\mapsto (g\circ f)^{-1}[U]$ is also computable.
By currying, the map $(f,g)\mapsto\lambda U.(g\circ f)^{-1}[U]\colon\tpbf{\Gamma}(X,Y)\times C(Y,Z)\to C(O_Z,\tpbf{\Gamma}_X)$ is also computable.
In other words, given a code $(f,g)$, one can effectively find a code of the map $U\mapsto (g\circ f)^{-1}[U]$, but the latter code is the same as the code of $g\circ f$.
\end{proof}

Let us recall that computable transparency is an abstraction of the notion of a universal oracle machine.
Thus, as expected, one can see that the universal function $\varphi^\Gamma$ is computably transparent.

\begin{cor}
Let $\Gamma$ be an $\N$-parametrized adequate $\Sigma$-pointclass.
Then the universal $\Gamma$-computable function $\varphi^\Gamma$ in Proposition \ref{prop:universal-func-pointclass} is computably transparent in the sense of Definition \ref{def:basic-computably-transparent}.
\end{cor}

\begin{proof}
First, $\varphi^\Gamma$ has a computable code $p$.
By Proposition \ref{prop:computable-transparent-pointclass}, given a code $c$ of a continuous function $f$, one can effectively find a code $r(p,c)$ of $f\circ\varphi^\Gamma$, which means that $f\circ\varphi^\Gamma(x)=\varphi^\Gamma(r(p,c),x)$.
Clearly, $F\colon x\mapsto (r(p,c),x)$ is continuous, and moreover, one can effectively find its code by using the information of $c$ (where recall that $p$ and $r$ are computable); hence, $f\mapsto F$ is computable.
Consequently, $\varphi^\Gamma$ is computably transparent.
\end{proof}

In conclusion, many pointclasses that appear in descriptive set theory can be thought of as oracles.
Note that $\varphi^\Gamma$ is always inflationary since an adequate pointclass $\Gamma$ always contains all computable sets (so all computable functions are $\Gamma$-computable).
Moreover, idempotence of $\varphi^\Gamma$ corresponds to the substitution property of $\Gamma$ (i.e., being closed under partial $\Gamma$-computable substitution; see \cite[Section 3G]{MosBook} for the precise definition), and an adequate $\Sigma$-pointclass having the substitution property is called a $\Sigma^\ast$-pointclass (see \cite[Section 7A]{MosBook} and also below).

%To mention a few specific examples, pointclasses corresponding levels of the Borel hierarchy, the C-hierarchy (the hierarchy by the Suslin operation), the R-hierarchy (the hierarchy by Kolmogorov's R-transform; see e.g.~\cite{Kan88}), the projective hierarchy, etc., are typical examples.
%Also, see \cite{MosBook} for a wealth of concrete examples of such pointclasses.

%%%%%%%%%%%%%%%%%%%%%%%%%%%%%%%%%%%%%%%%%%%%%%%%%%%%%%%%%%%%%%%%%%%
%%%%%%%%%%%%%%%%%%%%%%%%%%%%%%%%%%%%%%%%%%%%%%%%%%%%%%%%%%%%%%%%%%%
%%%%%%%%%%%%%%%%%%%%%%%%%%%%%%%%%%%%%%%%%%%%%%%%%%%%%%%%%%%%%%%%%%%
%%%%%%%%%%%%%%%%%%%%%%%%%%%%%%%%%%%%%%%%%%%%%%%%%%%%%%%%%%%%%%%%%%%
%%%%%%%%%%%%%%%%%%%%%%%%%%%%%%%%%%%%%%%%%%%%%%%%%%%%%%%%%%%%%%%%%%%
%%%%%%%%%%%%%%%%%%%%%%%%%%%%%%%%%%%%%%%%%%%%%%%%%%%%%%%%%%%%%%%%%%%

\subsection{Pointclass algebra}\label{sec:pointclass-partial-combinatory-algebra}

In this section, we see that $\Gamma$-measurability for a pointclass $\Gamma$ having a certain closure property always yields a partial combinatory algebra (recall the definition from Section \ref{sec:represented-space}).

Explicitly, put $\tpbf{N}=\N^\N$, and let $\tplf{N}\subseteq\tpbf{N}$ be all computable elements.
Then, define $a\ast_\Gamma b$ by $\varphi^\Gamma_a(b)$, where $\varphi^\Gamma$ is defined as in Proposition \ref{prop:universal-func-pointclass}.
Roughly speaking, the left $\ast_\Gamma$-application by a computable element $e\in \tplf{N}$ corresponds to a partial $\Gamma$-computable function, and the left $\ast_\Gamma$-application by an element $x\in\tpbf{N}$ corresponds to a partial $\tpbf{\Gamma}$-measurable function.

A pointclass $\Gamma$ is said to have the substitution property if it is closed under partial $\Gamma$-computable substitution (see Moschovakis \cite[Section 3G]{MosBook} for the details).
An adequate $\Sigma$-pointclass having the substitution property is called a {\em $\Sigma^\ast$-pointclass} (see also \cite[Section 7A]{MosBook}).

\begin{obs}[see also \cite{Kih20}]\label{prop:Kleene-first-algebra}
Let $\Gamma$ be an $\N$-parametrized $\Sigma^\ast$-pointclass.
Then $(\tplf{N},\tpbf{N},\ast_\Gamma)$ forms a relative PCA.
\end{obs}

\begin{proof}
As is well known, it suffices to construct ${\sf k}$ and ${\sf s}$.
Below we abbreviate $a\ast_\Gamma b$ as $a\ast b$.
To construct ${\sf k}$, as the projection is $\Gamma$-computable, so there is $\pi_0$ such that $\pi_0\ast\langle a,b\rangle=a$.
Then, by smn theorem (Fact \ref{thm:PCA-parametrization-theorem}), we have $S(\pi_0,a)\ast b=\pi_0\ast\langle a,b\rangle=a$.
Moreover, $a\mapsto S(\pi_0,a)$ is $\Gamma$-computable, so there exists ${\sf k}\in \tplf{N}$ such that ${\sf k}\ast a=S(\pi_0,a)$.
Hence, ${\sf k}\ast a\ast b=S(\pi_0,a)\ast b=a$.

To construct ${\sf s}$, since $\varphi^\Gamma$ in Proposition \ref{prop:universal-func-pointclass} is $\Gamma$-computable, and $\Gamma$ has the substitution property, the function $(x,y,z)\mapsto\varphi^\Gamma(\varphi^\Gamma(x,z),\varphi^\Gamma(y,z))=x\ast z\ast(y\ast z)$ is $\Gamma$-computable, so let $e\in \tplf{N}$ be an index of such a function.
By smn theorem (Fact \ref{thm:PCA-parametrization-theorem}), we have $S(e,x,y)\ast z=e\ast (x,y,z)$, and $(x,y)\mapsto S(e,x,y)$ is $\Gamma$-computable via some index $d\in\tplf{N}$.
By smn theorem again, we have $S(d,x)\ast y=d\ast(x,y)$, and $x\mapsto S(d,x)$ is computable, so let ${\sf s}\in \tplf{N}$ be its index.
Then ${\sf s}\ast x\ast y\ast z=S(d,x)\ast y\ast z=d\ast (x,y)\ast z=S(e,x,y)\ast z=x\ast z\ast(y\ast z)$.
Consequently, $(\tplf{N},\tpbf{N},\ast_\Gamma)$ is a relative PCA.
\end{proof}

In particular, we obtain a cartesian closed category whose morphisms are $\tpbf{\Gamma}$-measurably realizable functions (and even a topos, usually called a {\em realizability topos}; see e.g.~van Oosten \cite{vOBook}).
If $\Gamma=\Sigma^0_1$, the induced lightface PCA is equivalent to {\em Kleene's first algebra} (associated with Kleene's number realizability), and the boldface PCA is {\em Kleene's second algebra} (associated with Kleene's functional realizability).
The relative PCA induced from $\Gamma=\Sigma^0_1$ is known as Kleene-Vesley's algebra, cf.~\cite{vOBook}.
%In general, any $\Sigma^\ast$-pointclass (see Moschovakis \cite[Section 7A]{MosBook}) is an $\N$-parametrized adequate pointclass having the substitution property, so Observation \ref{prop:Kleene-first-algebra} is applicable.

In general, as a Spector pointclass (see Moschovakis \cite[Lemma 4C]{MosBook}) is a $\Sigma^\ast$-pointclass, many infinitary computation models yield (relative) PCAs.
The pointclass $\Pi^1_1$ is the best-known example of a Spector pointclass, and the induced lightface PCA obviously yields hyperarithmetical realizability; see e.g.~\cite[Example VI.3.3]{Beeson}.
For the boldface PCA, the associated total functions are exactly the Borel measurable functions.
Bauer \cite{Bau} also studied the PCA (and the induced realizability topos) obtained from infinite time Turing machines (ITTMs).

In conclusion, thus, a large number of natural pointclasses and their associated measurable functions that appear in descriptive set theory yield PCAs, and as well as various toposes.
%%%%%%%%%%%%%%%%%%%%%%%%%%%%%%%%%%%%%%%%%%%%%%%%%%%%%
%%%%%%%%%%%%%%%%%%%%%%%%%%%%%%%%%%%%%%%%%%%%%%%%%%%%%
%%%%%%%%%%%%%%%%%%%%%%%%%%%%%%%%%%%%%%%%%%%%%%%%%%%%%
%%%%%%%%%%%%%%%%%%%%%%%%%%%%%%%%%%%%%%%%%%%%%%%%%%%%%
%%%%%%%%%%%%%%%%%%%%%%%%%%%%%%%%%%%%%%%%%%%%%%%%%%%%%

%%%%%%%%%%%%%%%%%%%%%%%%%%%%%%%%%%%%
%%%%%%%%%%%%%%%%%%%%%%%%%%%%%%%%%%%%
%%%%%%%%%%%%%%%%%%%%%%%%%%%%%%%%%%%%
%%%%%%%%%%%%%%%%%%%%%%%%%%%%%%%%%%%%
%%%%%%%%%%%%%%%%%%%%%%%%%%%%%%%%%%%%

\section{Conclusion}

In this article, we presented new perspectives on oracles.
This makes it possible to understand various oracle computations or reducibility notions, which have been studied in computability theory in the past, in the context of synthetic descriptive set theory and realizability topos theory.
For example, in descriptive set theory, climbing various hierarchies (the Borel hierarchy, the C-hierarchy, the R-hierarchy, the projective hierarchy, etc.) now can be understood as giving oracles in the Kleene second algebra.

For another example, as a benefit of understanding oracles as operations on truth-values, it directly links ``oracle computations'' to ``non-constructive axioms'', and so we can expect new types of applications of oracle computation to (constructive) reverse mathematics.
One important difference from the old applications of oracle computation to classical reverse mathematics is that our new idea allows even formulas outside the framework of second-order arithmetic to be analyzed in the context of oracle computation.
This is because, if an operation on truth-values is a Lawvere-Tierney topology, then it yields a sheaf topos \cite{SGL}, whose internal logic naturally allows one to deal with formulas outside the framework of second-order arithmetic; Bauer's result \cite{Bau} on the ITTM-realizability topos is one such example.
%Here, note that ITTM-realizability can be treated in our framework by considering computations using a universal ITTM as an oracle (or by directly considering a PCA induced from ITTM-computability).

%to the best of our knowledge, categorical computability theory or realizability topos theory rarely reflected modern, non-trivial ideas and techniques on Turing reducibility (or other reducibility), which is one of the central objects of computability theory.
%For this reason, there appears to have been a disconnect between two communities.
%Now, however, a bridge has been built between these theories.
One last point:
A new bridge has now been built connecting degree theory, descriptive set theory, and realizability topos theory.
By crossing this bridge, our future goal is to bring modern ideas and techniques of degree theory to realizability topos theory and related areas (and vice versa).

\begin{ack}
The author is deeply grateful to Satoshi Nakata for his insightful comments and suggestions during our discussion.
The author also wishes to thank Ruiyuan Chen, Matthew de Brecht, S\=oichir\=o Okuda, Jaap van Oosten, Matthias Schr\"oder and Thomas Streicher for valuable comments.
\end{ack}

%\markboth{T. Kihara}{Rethinking the notion of oracle}

\bibliographystyle{plain}
\bibliography{references}

%\end{document}

%%%%%%%%%%%%%%%%%%%%%%%%%%%%%%%%%%%%%%%%%%%%%%%%%%%%%
%%%%%%%%%%%%%%%%%%%%%%%%%%%%%%%%%%%%%%%%%%%%%%%%%%%%%
%%%%%%%%%%%%%%%%%%%%%%%%%%%%%%%%%%%%%%%%%%%%%%%%%%%%%
%%%%%%%%%%%%%%%%%%%%%%%%%%%%%%%%%%%%%%%%%%%%%%%%%%%%%
%%%%%%%%%%%%%%%%%%%%%%%%%%%%%%%%%%%%%%%%%%%%%%%%%%%%%

%%%%%%%%%%%%%%%%%%%%%%%%%%%%%%%%%%%%%%%%%%%%%%%%%%%%%%%%%%%%
%%%%%%%%%%%%%%%%%%%%%%%%%%%%%%%%%%%%%%%%%%%%%%%%%%%%%%%%%%%%
%%%%%%%%%%%%%%%%%%%%%%%%%%%%%%%%%%%%%%%%%%%%%%%%%%%%%%%%%%%%
%%%%%%%%%%%%%%%%%%%%%%%%%%%%%%%%%%%%%%%%%%%%%%%%%%%%%%%%%%%%
%%%%%%%%%%%%%%%%%%%%%%%%%%%%%%%%%%%%%%%%%%%%%%%%%%%%%%%%%%%%
%%%%%%%%%%%%%%%%%%%%%%%%%%%%%%%%%%%%%%%%%%%%%%%%%%%%%%%%%%%%
%%%%%%%%%%%%%%%%%%%%%%%%%%%%%%%%%%%%%%%%%%%%%%%%%%%%%%%%%%%%

\end{document}